\documentclass[11pt]{amsart}

\usepackage{amsmath,amssymb}
\usepackage{accents}
\usepackage{graphicx}
\usepackage{verbatim}

\newlength{\dhatheight}

\newtheorem{theorem}{Theorem}[section]

\newtheorem{proposition}[theorem]{Proposition}
\newtheorem{lemma}[theorem]{Lemma}
\newtheorem{remark}[theorem]{Remark}
\newtheorem{corollary}[theorem]{Corollary}
\newtheorem{definition}[theorem]{Definition}
\newtheorem{notation}[theorem]{Notation}

\newcommand{\rfthmitisqi}{6.1}

\newcommand{\N}{{\mathbb N}}
\newcommand{\R}{{\mathbb R}}
\newcommand{\Z}{{\mathbb Z}}

\newcommand{\ep}{\epsilon}
\newcommand{\dd}{\Delta}
\newcommand{\ra}{\rightarrow}
\newcommand{\ras}{{\stackrel{~~*}{\ra}}}
\newcommand{\ga}{\Gamma}
\newcommand{\xx}{X^{(1)}}
\newcommand{\pp}{{\mathcal{P}}}
\newcommand{\bo}{{\partial}}
\newcommand{\td}{\widetilde d}

\newcommand{\tdd}{\widetilde \dd}
\newcommand{\nn}{\N[\frac{1}{4}]}
\newcommand{\ps}{\Psi}
\newcommand{\ph}{\Sigma}
\newcommand{\tht}{\Theta}
\newcommand{\vt}{\vartheta}
\newcommand{\slex}{<_{sl}}

\newcommand{\up}{\Upsilon}
\newcommand{\tp}{\widetilde \phi}
\newcommand{\tth}{\widetilde \theta}
\newcommand{\dmi}{diameter inequality}
\newcommand{\dms}{diameter inequalities}
\newcommand{\dmc}{diameter}
\newcommand{\cfl}{combed filling}
\newcommand{\tfi}{tame filling inequality}
\newcommand{\tfs}{tame filling inequalities}

\newcommand{\dia}{diagrammatic}
\newcommand{\vky}{{boundary 1-combing}}

\newcommand{\cnf}{{combed ${\mathcal N}$-filling}}%combed normal form filling?
\newcommand{\ehy}{{edge 1-combing}}%normal homotopy?
\newcommand{\ehs}{{edge 1-combings}}%****WARNING: if I change to "normal",
\newcommand{\edg}{{${\mathcal N}$-diagram}}%edge diagram? normal diagram?
\newcommand{\dhy}{{circular 1-combing}}
\newcommand{\dhs}{{circular 1-combings}}
\newcommand{\scf}{{$S^1$-combed filling}}
\newcommand{\cc}{{\mathcal{N}}}
\newcommand{\cld}{{\mathcal{D}}}
\newcommand{\cle}{{\mathcal{E}}}
\newcommand{\clf}{{\mathcal{F}}}

\newcommand\hc{{f}}
\newcommand\he{{\hat e}}

\newcommand\te{{\tilde e}}
\newcommand\pe{{e^{\prime\prime}}}
\newcommand\hhe{{\widetilde e}}
\newcommand\hhr{{\widetilde r}}

\newcommand\hhp{{\widetilde p}}

\newcommand{\stkbl}{{stackable}}%recursively fillable, striable

\newcommand{\fstkbl}{{stackable}}
\newcommand{\astkbl}{{algorithmically stackable}}

\newcommand{\afstkbl}{{algorithmically stackable}}
\newcommand{\prs}{{\stackrel{~~p*}{\ra}}}
\newcommand{\mui}{\mu^i}
\newcommand{\mue}{\mu^e}
\newcommand{\tj}{\widetilde j}
\newcommand\ves{{\vec E_{r}}}  % the bad edges = recursive edges
\newcommand\dgd{{\vec E_d}}  % the directed good edges
\newcommand{\rcnf}{{recursive combed $\cc$-filling}}
\newcommand{\rcf}{{recursive combed  filling}}
\newcommand{\edi}{{ediam}}
\newcommand{\idi}{{idiam}}
\newcommand{\maxr}{\zeta}

\newcommand{\kti}{k_\cc^i}
\newcommand{\kxi}{k_r^i}
\newcommand{\kte}{k_\cc^e}
\newcommand{\kxe}{k_r^e}

\newcommand{\alg}{S_\ff}
\newcommand{\tc}{z}

\newcommand{\ega}{e_{g,a}}

\newcommand{\ff}{\Phi}
\newcommand{\ttt}{{\mathcal T}}
\newcommand{\vece}{{\vec E}}
\newcommand{\lbl}{{\mathsf{label}}}%\kappa
\newcommand{\rep}{\chi}
\newcommand{\path}{{\mathsf{path}}}%\xi
\newcommand{\init}{{\mathsf{i}}}
\newcommand{\term}{{\mathsf{t}}}

\newcommand{\skz}{^{(0)}}
\newcommand{\sko}{^{(1)}}
\newcommand{\oc}{1-combing}
\newcommand{\cay}{X}
\newcommand{\cdd}{coarse distance}
\newcommand{\tff}{tame filling function}
\newcommand{\tffs}{tame filling functions}
\newcommand{\dff}{diameter filling function}
\newcommand{\dffs}{diameter filling functions}
\newcommand{\tcf}{radial tame combing function}
\newcommand{\tcfs}{radial tame combing functions}
\newcommand{\idmap}{id_{[0,1]}}

\begin{document}
\title[Tame filling invariants for groups]
{Tame filling invariants for groups}

\author[M.~Brittenham]{Mark Brittenham}
\address{Department of Mathematics\\
        University of Nebraska\\
         Lincoln NE 68588-0130, USA}
\email{mbrittenham2@math.unl.edu}

\author[S.~Hermiller]{Susan Hermiller}
\address{Department of Mathematics\\
        University of Nebraska\\
         Lincoln NE 68588-0130, USA}
\email{smh@math.unl.edu}
%\date{December 23, 2011}

\begin{abstract}

A new pair of asymptotic invariants 
for finitely presented groups, called intrinsic and
extrinsic \tffs, are introduced.
These 
filling functions are quasi-isometry invariants
that strengthen the
notions of intrinsic and extrinsic diameter functions for
finitely presented groups. 
% These intrinsic and extrinsic \tfs\ 
%are quasi-isometry invariants.
%, up to
%Lipschitz equivalence of functions (and, in the case of
%the intrinsic \tfi, up to choice of a sufficiently large
%set of defining relators).  
%We show that  
%tame combing invariants developed by Mihalik
%and Tschantz and further by Hermiller and Meier are 
%equivalent to extrinsic \tfs,
%and so tame combing inequalities are a refinement
%of extrinsic diameter inequalities.
%the intrinsic analog of tame combing conditions.
We show that the existence of a (finite-valued)
\tff\ implies that the group is
tame combable.
Bounds on both intrinsic and extrinsic
\tffs\ are discussed for stackable groups,
including groups with a finite complete
rewriting system,
Thompson's group $F$, and almost convex groups.
%These results yield evidence for the conjectured
%existence of a finitely presented group that
%is not stackable.
%Moreover, we prove that
%{\tcf}s are equivalent 
%to extrinsic tame 
%filling functions, showing
%that tame combing functions are a refinement
%of extrinsic diameter functions.
\end{abstract}

\maketitle

%\tableofcontents

%%%%%%%%%%%%%%%%%%%%%%%%%%%%%%%%%%%%%%%%%%%%%%%%%%%%%%%%%%%%%%%%%%%%%%%%%%%%
%%%%%%%%%%%%%%%%%%%%%%%%%%%%%%%%%%%%%%%%%%%%%%%%%%%%%%%%%%%%%%%%%%%%%%%%%%%%
%%%%%%%%%%%%%%%%%%%%%%%%%%%%%%%%%%%%%%%%%%%%%%%%%%%%%%%%%%%%%%%%%%%%%%%%%%%%

\section{Introduction}\label{sec:intro}

%%%%%%%%%%%%%%%%%%%%%%%%%%%%%%%%%%%%%%%%%%%%%%%%%%%%%%%%%%%%%%%%%%%%%%%%%%%%
%%%%%%%%%%%%%%%%%%%%%%%%%%%%%%%%%%%%%%%%%%%%%%%%%%%%%%%%%%%%%%%%%%%%%%%%%%%%

In geometric group theory, many asymptotic
invariants associated to a
group $G$ with a finite presentation
$\pp=\langle A \mid R\rangle$ have been defined
using properties of van Kampen diagrams over
this presentation.  Collectively, these are
referred to as filling invariants; an
exposition of many of these is given by
Riley in~\cite[Chapter II]{riley}.  
One of the most well-studied
filling functions is the isodiametric, or intrinsic
diameter, function for $G$.  
The adjective ``intrinsic''
refers to the fact that the distances are 
measured using the path metric
$d_\dd$ in 
%the 1-skeleton $\dd\sko$  of 
van Kampen diagrams $\dd$;
%for this property.  
%it is natural to consider
measuring distance using the path metric
$d_\cay$ in the 1-skeleton $\cay\sko$
of the Cayley complex $\cay=\cay(G,\pp)$, instead, gives
an ``extrinsic'' property, and
in~\cite{bridsonriley}, 
Bridson and Riley defined and studied 
properties of extrinsic
diameter functions.
In this paper we define two new filling invariants
that refine these diameter filling functions.

In order to accomplish this, in a 
2-dimensional
van Kampen diagram $\dd$ or 
Cayley complex $X$,
we consider ``distance'' to 2-cells 
as well as within 1-skeleta.  Since the
path metric may not extend to a metric on 2-complex,
given a combinatorial 2-complex $Y$ with a
basepoint vertex $y$, and any point $p \in Y$, we use
the {\em coarse distance}~\label{def:coarsedist}
$\td_Y(y,p)$, defined as follows.  
Let $\td_Y(y,p):=d_Y(y,p)$ be the path metric distance
from $y$ to $p$ in $Y\sko$ if $p$ is a vertex; 
if $p$ is 
in the interior %$Int(e)$ 
of an edge $e$ of $Y$ let
$\td_Y(y,p) := \min\{\td_Y(y,v) \mid v \in \bo(e)\} + \frac{1}{2}$
(the path metric distance from $y$ to the midpoint of $e$);
and if $p$ is in the interior of a 2-cell $\sigma$ of $Y$, then 
let $\td_Y(y,p) := 
\max\{\td_Y(y,q) \mid q \in \bo(\sigma) \setminus Y\skz \}-\frac{1}{4}$.
In order to measure extrinsic coarse
distance in any van Kampen diagram $\dd$ over $\pp$
with basepoint $*$, we apply the unique cellular
map $\pi_\dd:\dd \ra X$ such that $\pi_\dd(*)=\ep$
is the vertex of $X$ labeled by the identity of $G$
and $\pi_\dd$ maps $n$-cells to $n$-cells preserving
edge and 2-cell boundary labels and orientations,
and use the coarse distance $\td_X$ in $X$.

We also use 1-combings
(developed in~\cite{mihaliktschantz}) of these 2-complexes; that is,
given a subcomplex $Z$ of $Y\sko$, a {\em \oc}\label{def:1comb}
of the pair $(Y,Z)$ at a basepoint $y_0 \in Y$ is a continuous function
$\ps:Z \times [0,1] \ra Y$ satisfying:
\begin{itemize} 
\item[(C1)] $\ps(p,0)=y_0$ and 
$\ps(p,1)=p$ for all $p \in Z$,
\item[(C2)] if $y_0 \in Z$ then $\ps(y_0,t)=y_0$
for all $t \in [0,1]$, 
and
\item[(C3)] whenever $p \in Z\skz$, 
then $\ps(p,t) \in Y\sko$ for all $t \in [0,1]$.
\end{itemize}
(Several {\oc}s for van Kampen diagrams are illustrated
in Figure~\ref{fig:seashellcomb}.)
Thus a \oc\ is a continuous choice of paths
in $Y$ from $y_0$ to the points of $Z$, along which
we will measure the tameness of the complex $Y$.

A continuous function $\psi:Z' \times [0,1] \ra Y$, where
$Z'$ is any 1-complex, is {\em $f$-tame}~\label{def:ftame}
with respect to a nondecreasing function $f:\nn \ra \nn$
if for all $p \in Z'$ and $0 \le s < t \le 1$,
we have 
\[ 
\td_Y(y,\psi(p,s)) \le f(\td_Y(y,\psi(p,t)))~.
\]

\begin{definition}\label{def:tfi}
A nondecreasing function $f:\nn \ra \nn$ is an
{\em intrinsic}
[respectively, {\em extrinsic}] {\em \tff} for a
group $G$ over a finite
presentation $\pp=\langle A \mid R\rangle$ if 
for all words $w \in A^*$ that represent the identity in $G$,
there exists a van Kampen diagram $\dd$ for $w$ over $\pp$, 
with basepoint $*$, 
and a \oc\ 
$\ps:\bo \dd \times [0,1] \ra \dd$ at $*$ such that 
the function
$\ps$ [respectively, $\pi_\dd \circ \ps$]
is $f$-tame.
\end{definition} 

Viewing the unit interval as a unit of time
and using $\psi$ to denote either $\ps$ or $\pi_\dd \circ \ps$,
this property says that 
if at a time $s$ the 
path $\psi(p,\cdot)$
has reached a \cdd\ greater
than $f(n)$
from the basepoint, then at all later times $t>s$
the path must remain further than $n$ from the basepoint.
Essentially, the \tff\  bounds the extent
to which a 1-combing of $\dd$ 
from the basepoint to the boundary must 
go outward steadily, rather than returning to 
$n$-cells in $\dd$ (with $n \le 2$) 
that are significantly closer to the basepoint.

In Section~\ref{sec:relationships}, we discuss 
relationships between tame and diameter filling
functions, and
show in Proposition~\ref{prop:itimpliesid}
that an intrinsic or
extrinsic \tff\ with respect to a function $f$ 
implies that the function $n \mapsto \lceil f(n) \rceil$
is an upper bound for the
intrinsic or extrinsic diameter function (respectively).
%\dmi\ for the function .

Another motivation for the definition 
of \tffs\ is to elucidate the close relationship of
the concept of tame combability, as defined by Mihalik
and Tschantz~\cite{mihaliktschantz} (see p.~\pageref{def:tamecomb}
for details),
as well as associated {\tcf}s
advanced by Hermiller and Meier~\cite{hmeiermeastame}
(defined on p.~\pageref{def:tcf}),
with more well-studied diameter filling functions.
%Tame combings are homotopies in the Cayley complex,
%in contrast to homotopies in van Kampen diagrams for 
%the case of combed fillings.
In Section~\ref{sec:relax}, we show 
an equivalence between
 tame combing functions and extrinsic filling
invariants.
% of van Kampen diagrams
%in the following.
% portion of Corollary~\ref{cor:etistc}.

\smallskip

\noindent{\bf Corollary~\ref{cor:etistc}.} {\em
Let $G$ be a finitely presented group.
Up to Lipschitz equivalence of nondecreasing functions,
the function $f$ is an extrinsic \tff\ for $G$
 if and only if $f$ is a \tcf\ for $G$.
}

\smallskip

\noindent That is, Corollary~\ref{cor:etistc}
and Proposition~\ref{prop:itimpliesid} together 
show that a \tcf\ is an upper bound for
the extrinsic diameter function.

In contrast to the definition
of \dffs\ (see Definition~\ref{def:diamfcn} for details),
\tffs\ do not depend on
the length $l(w)$ of the word $w$.  Indeed, 
in Definition~\ref{def:tfi} 
the property that
a \oc\ path $\ps(p,\cdot)$ cannot return to a distance
less than $n$ from the basepoint 
after it has reached a distance greater than $f(n)$ is 
uniform for all reduced words over $A$ representing $\ep$.
As a consequence, 
although every finitely presented
group admits well-defined intrinsic and extrinsic 
diameter functions,
it is not clear whether every pair
$(G,\pp)$ admits a well-defined
(i.e.~finite-valued) intrinsic or extrinsic \tff.
In Section~\ref{sec:relax}, we show that
the existence of a well-defined \tff\ 
implies tame combability.

\smallskip

\noindent{\bf Corollary~\ref{cor:tamecomb}.} {\em
If $G$ has a well-defined extrinsic
\tff\ over some finite
presentation, then $G$ is tame combable.
}

\smallskip

A long-standing conjecture of
Tschantz~\cite{tschantz} states
that there is a finitely presented group that is
not tame combable, and as a result, 
that the converse of Proposition~\ref{prop:itimpliesid} fails
in the extrinsic case and the extrinsic \tff\ is
a strict upper bound for
(or a stronger invariant than) the extrinsic
diameter function.
That is, Tschantz's conjecture implies the existence of
a finitely presented group which admits
a finite-valued extrinsic diameter function $f$,
%If this conjecture holds, then the converse of 
%Proposition~\ref{prop:itimpliesid} would not hold in the
%extrinsic case.
but which does not
have a finite-valued extrinsic \tff, and 
hence does not have
an extrinsic \tff\ Lipschitz equivalent to $f$.

While a \tcf\ is an extrinsic property,
Corollary~\ref{cor:etistc} also shows 
that an intrinsic \tff\ can be interpreted as the
intrinsic analog of a \tcf.
The fundamental differences between intrinsic and
extrinsic properties 
arising in Section~\ref{sec:relax}
all stem from the fact that 
gluing van Kampen diagrams along their boundaries
preserves extrinsic distances, but not necessarily
intrinsic distances.  

In Section~\ref{sec:examples} we discuss \tffs\ for
several  large classes of groups.
We begin in
 Section~\ref{subsec:stkbl}
by considering stackable groups, 
defined by
the present authors in~\cite{britherm}.  Stackability is a
topological property of the Cayley graph that 
holds for almost convex groups and
groups with finite complete rewriting systems
(and hence holds for all fundamental groups
of 3-manifolds with a uniform geometry), and that gives
a uniform model for the inductive procedures
%that gives an inductive procedure 
to build van Kampen diagrams 
in these groups.  The procedure
is an algorithm in the case that the group
is algorithmically \stkbl. 
(See p.~\pageref{def:stkbl} for these definitions.)
%The class  of 
This procedure naturally leads to a method
of constructing 1-combings in these van Kampen
diagrams, which we use to obtain the following.

\smallskip

\noindent{\bf Theorem~\ref{thm:stkblhastff}.}  {\em
If $G$ is a \fstkbl\ group,  
then $G$ admits well-defined intrinsic and extrinsic \tffs.
}

\smallskip

\noindent{\bf Theorem~\ref{thm:astktf}.}  {\em
If $G$ is an \afstkbl\ group, then $G$ has
recursive intrinsic and extrinsic \tffs.
}

\smallskip

This leads us to another motivation for 
studying \tfs, namely to give information leading toward 
answering the open
question of whether there exists a finitely presented
group which is not \stkbl.
An immediate consequence of Theorem~\ref{thm:stkblhastff}
and Corollary~\ref{cor:tamecomb} is that every \fstkbl\ group
satisfies the quasi-isometry invariant 
tame combable property.
If Tschantz's
conjecture~\cite{tschantz} above 
of the existence of a finitely presented group that
is not tame combable is true, 
such a group $G$ would also
not be \stkbl\ with
respect to any finite generating set.

In the next three subsections
of Section~\ref{sec:examples} we compute 
more detailed bounds on 
tame filling invariants for several classes of
\fstkbl\ groups.
In Section~\ref{subsec:rs} we consider groups
that can be presented by rewriting systems.  
In this case we use the results of Section~\ref{subsec:stkbl}
to obtain \tffs\ in terms of the string growth
complexity function $\gamma:\N \ra \N$, where
$\gamma(n)$ is the maximal word length that can
be achieved after finitely many rewritings of
a word of length up to $n$.

\smallskip

\noindent{\bf Proposition~\ref{prop:rsgrowth}.} {\em
Let $G$ be a group with a finite complete
rewriting system. Let
$\gamma$ be the
string growth complexity function and let $\maxr$ denote the length of 
the longest rewriting rule.
Then the 
function $n \mapsto \gamma(\lceil n \rceil +\maxr+2)+1$
is both an intrinsic and an extrinsic \tff\ for $G$.
}

\smallskip

\noindent This result has 
%considerable 
potential for application in 
%reducing the amount of work in 
%the process of 
searching for
finite complete rewriting systems for groups.
A choice of partial ordering used to determine
the termination property of a rewriting system implies an
upper bound on the string growth complexity function.
Then given a lower bound on the intrinsic or extrinsic
\tfs\ or \dms, this corollary
can be used to eliminate partial orderings before
attempting to use them (e.g., via the
Knuth-Bendix algorithm) to construct a rewriting system.
%A further paper by the present authors will 
%address this application more fully.

We note in Section~\ref{subsec:rs}
that the iterated Baumslag-Solitar groups $G_k$
are examples of groups admitting recursive intrinsic
and extrinsic \tffs.  However, applying the lower bound of
Gersten~\cite{gerstenexpid} on their diameter functions,
for each natural number $k>2$
the group $G_k$ does not admit intrinsic or extrinsic
\tffs\ with respect to
a $(k-2)$-fold tower of exponentials.

Section~\ref{subsec:finite} contains a proof
that all finite groups, with respect to
all finite presentations, have both
intrinsic and extrinsic \tffs\ that are constant
functions.

In Section~\ref{subsec:linear} we consider 
three examples of \stkbl\ groups for which a linear
\tcf\ was known, and analyze the \stkbl\ structure
to obtain bounds on intrinsic \tffs\ for these examples.
The first of these is
Thompson's group $F$; i.e.,
the group of orientation-preserving piecewise linear
automorphisms of the unit interval for which all linear
slopes are powers of 2, and all breakpoints lie in the
the 2-adic numbers. 
Thompson's group $F$ has been the focus of considerable
research in recent years, and yet the questions of
whether $F$ is automatic or has a finite complete rewriting system  
are open (see the problem list at~\cite{thompsonpbms}).
In~\cite{chst} Cleary, Hermiller, Stein, and Taback
show that $F$ is \stkbl, and also
(after combining their result
with Corollary~\ref{cor:etistc}) that $F$ admits a
linear extrinsic \tff.  In Section~\ref{subsec:linear}
we show that this group also 
admits a linear intrinsic \tff, thus strengthening
the result of Guba~\cite[Corollary 1]{guba} that $F$ has a linear
intrinsic diameter function.

We also show in Section~\ref{subsec:linear} 
that the Baumslag-Solitar group 
$BS(1,p)$ with $p \ge 3$ admits an intrinsic \tff\ Lipschitz
equivalent to the exponential function $n \mapsto p^n$, 
using the linear extrinsic \tff\ for these groups
shown in~\cite{chst}.
Next we consider Cannon's almost convex groups~\cite{cannon} 
(see Definition~\ref{def:ac}), which include
all word hyperbolic groups and  
cocompact discrete groups of isometries of Euclidean
space (with respect to every generating set)~\cite{cannon},
as well as all shortlex automatic groups.
Building upon the characterization
of almost convexity by \tcfs\ 
in~\cite{hmeiermeastame}, 
we show that
almost convexity is equivalent to
conditions on \tffs.

\smallskip

\noindent{\bf Theorem~\ref{thm:ac}.}  {\em
Let $G$ be a group with finite generating set $A$, and
let $\iota: \nn \ra \nn$ denote the identity
function.  The following
are equivalent:
\begin{enumerate}
\item The pair $(G,A)$ is almost convex.
\item $\iota$ is an intrinsic \tff\ for $G$ over 
a finite
presentation $\pp=\langle A \mid R \rangle$.
\item $\iota$ is an extrinsic \tff\ for $G$ over 
a finite
presentation $\pp=\langle A \mid R \rangle$.
\end{enumerate}
}

\smallskip

Our last example, in Section~\ref{subsec:combable},
is a class of combable groups.  
%The normal forms
%for the combable structures in this class
%must 

\smallskip

\noindent{\bf Corollary~\ref{cor:combable}.}  {\em 
If a finitely generated group $G$ admits a 
quasi-geodesic language of normal forms
that label simple paths in the Cayley graph and that
satisfy a $K$-fellow traveler property, then $G$ admits linear
intrinsic and extrinsic \tffs.
}

\smallskip

\noindent In particular,
all automatic groups
over a prefix-closed language of normal forms
satisfy the hypotheses of Corollary~\ref{cor:combable}.
This result strengthens that of Gersten~\cite{gersten},
that combable groups have a linear intrinsic
diameter function.

Finally, in Section~\ref{sec:qiinv}, we prove that \tffs\ are
quasi-isometry invariants up to Lipschitz equivalence, 
in Theorem~\rfthmitisqi.

%%%%%%%%%%%%%%%%%%%%%%%%%%%%%%%%%%%%%%%%%%%%%%%%%%%%%%%%%%%%%%%%%%%%%%%%%%%%
%%%%%%%%%%%%%%%%%%%%%%%%%%%%%%%%%%%%%%%%%%%%%%%%%%%%%%%%%%%%%%%%%%%%%%%%%%%%
%%%%%%%%%%%%%%%%%%%%%%%%%%%%%%%%%%%%%%%%%%%%%%%%%%%%%%%%%%%%%%%%%%%%%%%%%%%%

\section{Notation} \label{sec:notation}

%%%%%%%%%%%%%%%%%%%%%%%%%%%%%%%%%%%%%%%%%%%%%%%%%%%%%%%%%%%%%%%%%%%%%%%%%%%%
%%%%%%%%%%%%%%%%%%%%%%%%%%%%%%%%%%%%%%%%%%%%%%%%%%%%%%%%%%%%%%%%%%%%%%%%%%%%

Throughout this paper, let $G$ be a group
with a finite {\em symmetric} presentation
$\pp = \langle A \mid R \rangle$; that is, 
such that the generating set $A$ is closed under inversion,
and the set $R$ of defining relations is closed under
inversion and cyclic conjugation.
We will also assume that for each $a \in A$,
the element of $G$ represented by $a$ is not 
the identity $\ep$ of $G$.
%(Note that we may not assume that every defining relator
%is freely reduced, but the freely reduced representative
%of every defining relator, except the
%empty word, must also be a defining relator.)

Let $\rep:A^* \ra G$ be the monoid homomorphism
mapping generators to their representatives in $G$.
A set of {\em normal forms} is a subset $\cc \subset A^*$
for which the restriction of $\rep$ to $\cc$ is
a bijection. 
We will also assume that every set of normal forms in this 
paper contains the empty word.
Write $y_g$ for the normal form of the element $g$ of $G$,
and write $y_w$ for the normal form of $\rep(w)$.

For a word $w \in A^*$, we write $w^{-1}$ for the 
formal inverse of $w$ in $A^*$, and let
$l(w)$ denote the length of the word $w$. 
%(obtained by replacing each letter by its 
%inverse in $A$, and reversing the order of the 
%letters in the word).
Let 1 denote the empty word in $A^*$.
For words $v,w \in A^*$, we write $v=w$ if $v$
and $w$ are the same word in $A^*$, and write $v=_G w$ if
$\rep(v)=\rep(w)$ are the same element of $G$.
%We will always assume that the generating set $A$ is finite.
%, and if $b$ is an element of $A$ distinct from $a$, then
%$a \neq_G b$.

Let $X$ be the Cayley 2-complex corresponding to this presentation,
whose 1-skeleton $\ga=X\sko$ is the Cayley graph of $G$ with
respect to $A$.  Denote the path metric on $\xx$ by $d_X$;
for $w \in A^*$,
$d_X(\ep,w)$ then denotes path distance in $\xx$ from
the identity $\ep$ of $G$ to the element of $G$
represented by the word $w$.
For all $g \in G$ and $a \in A$, let $\ega$
denote the directed edge in $\xx$ labeled $a$ from $g$ to $ga$.
By usual convention, both directed edges $\ega$ and $e_{ga,a^{-1}}$
have the same underlying undirected CW complex
edge in $\xx$ between the vertices labeled $g$ and $ga$.

For an arbitrary word $w$ in $A^*$
that represents the
trivial element $\ep$ of $G$, there is a {\em van Kampen
diagram} $\dd$ for $w$ with respect to $\pp$.  
That is, $\dd$ is a finite,
planar, contractible combinatorial 2-complex with 
edges directed and
labeled by elements of $A$, satisfying the
properties that the boundary of 
%the infinite region outside of 
$\dd$ is an edge path labeled by the
word $w$ starting at a basepoint 
vertex $*$ and
reading counterclockwise, and every 2-cell in $\dd$
has boundary labeled by an element of $R$.

Note that although the definition in the previous
paragraph is standard, it involves a slight abuse of notation,
in that the 2-cells of a van Kampen diagram
are polygons whose boundaries are labeled by words in $A^*$,
rather than elements of a (free) group.  We will
also consider the set $R$ of defining relators
as a finite subset of $A^* \setminus \{1\}$, where 1 is
the empty word.
We do not assume that every defining relator
is freely reduced, but the freely reduced representative
of every defining relator, except 1, must also be in $R$.
%(In fact, we only use (potentially) unreduced
%relators in the quasi-isometry proof of Section~\ref{sec:qiinv}.)
%%% Note to me:  I need this for the QI proof.  %%%%
%We also define the finite set 
%$\hr:=\{x \in A^* \mid \exists a \in A$ with $xa \in R\}$.

Recall that  $\pi_\dd:\dd \ra X$
denotes the cellular map such that $\pi_\dd(*)=\ep$ and
$\pi_\dd$ maps $n$-cells to $n$-cells preserving edge
(and 2-cell boundary) labels and orientations.
A word $w \in A^*$ 
is called a {\em simple} word if $w$ labels 
a simple path in the corresponding Cayley graph $\xx$; that is,
the path does not repeat any vertices or edges.
Since a path labeled by a simple word $w$
in a van Kampen diagram $\dd$ maps via $\pi_\dd$ to 
a simple path in $X$,
the path in $\dd$ must also be simple.
Simple words are a useful ingredient for gluing 
van Kampen diagrams together; given two 
planar diagrams with simple boundary subpaths
sharing a common label but
in reversed directions, the diagrams can be
glued along the subpaths to construct another
planar diagram.

In general, there may be many different van 
Kampen diagrams for the word $w$.  Also,
we do not assume that van Kampen diagrams
in this paper are reduced; that is, we allow adjacent
2-cells in $\dd$ to be labeled by the same relator with
opposite orientations.

If $Y$ is any 2-complex, 
let $E_Y$, $\vece_Y$, and $\vec P_Y$ denote
the sets of  undirected edges, directed edges, and directed paths
in $Y$, respectively.  
Let $\init=\init_Y,\term=\term_Y:\vec P_Y \ra Y\skz$ 
map paths to their initial
and terminal vertices, respectively. 
If $Y$ is either the Cayley complex $X$ 
or the Cayley graph $\ga=X\sko$, define

$\path_Y:G \times A^* \ra \vec P_Y$ by $\path_Y(g,w):=$ the
  path in $Y$ starting at $g$ and labeled~by~$w$.

\noindent If $Y=\dd$ is a van Kampen diagram, define

$\path_\dd:A^* \times A^* \ra \vec P$ by $\path_\dd(v,w):=$ the
  counterclockwise path in $\bo\dd$ labeled by  $w$, 

\hspace{.1in}  that starts at the end of
  the counterclockwise path along $\bo\dd$ from $*$ labeled by $v$.

\noindent In both cases define

$\lbl=\lbl_Y:\vec P_Y \ra A^*$ by $\lbl_Y(p):=$ 
  the word labeling the path $p$.

\noindent For any function $\rho:Y \times [0,1] \ra Z$, the
notation $\rho(p,\cdot)$ denotes the function
$[0,1] \ra Z$ given by $t \mapsto \rho(p,t)$,
and $\rho(\cdot,t)$ denotes the function $Y \ra Z$
defined by $y \mapsto \rho(y,t)$.

Two functions  
$f,g:\nn \ra \nn$ are called {\em Lipschitz
equivalent} if there is a constant $C$
such that $f(n) \le Cg(Cn+C)+C$
and $g(n) \le Cf(Cn+C)+C$ for all $n \in \nn$.

We refer to a \oc\ $\ps:\bo\dd \times [0,1] \ra \dd$ at the basepoint
of a van Kampen diagram as a {\em \vky} of $\dd$ (see 
in Figure~\ref{fig:seashellcomb}).
A collection $\{\dd_w \mid w \in A^*, w=_G\ep\}$ of
van Kampen diagrams for all words representing the trivial element,
where each 
diagram $\dd_w$ has boundary label $w$,
is called a {\em filling} for the
group $G$ over the presentation $\pp$.
A {\em \cfl} for $G$ over $\pp$ is a
collection $\clf=\{(\dd_w,\ps_w) \mid w \in A^*, w=_G \ep\}$
such that each $\dd_w$ is a van Kampen diagram
with boundary word $w$, and $\ps_w:\bo \dd_w \times [0,1] \ra \dd_w$ 
is a 1-combing.

See for example~\cite{bridson} or~\cite{lyndonschupp} 
for expositions of the theory of van Kampen diagrams.

%\vspace{.1in}

%%%%%%%%%%%%%%%%%%%%%%%%%%%%%%%%%%%%%%%%%%%%%%%%%%%%%%%%%%%%%%%%%%%%%%%%%%%%
%%%%%%%%%%%%%%%%%%%%%%%%%%%%%%%%%%%%%%%%%%%%%%%%%%%%%%%%%%%%%%%%%%%%%%%%%%%%
%%%%%%%%%%%%%%%%%%%%%%%%%%%%%%%%%%%%%%%%%%%%%%%%%%%%%%%%%%%%%%%%%%%%%%%%%%%%

\section{Relationships among filling invariants}\label{sec:relationships}

%%%%%%%%%%%%%%%%%%%%%%%%%%%%%%%%%%%%%%%%%%%%%%%%%%%%%%%%%%%%%%%%%%%%%%%%%%%%
%%%%%%%%%%%%%%%%%%%%%%%%%%%%%%%%%%%%%%%%%%%%%%%%%%%%%%%%%%%%%%%%%%%%%%%%%%%%

In this section we show that for finitely
presented groups, \tffs\ give
upper bounds for \dffs.
We begin with a description of the diameter functions
and their motivation of the tame filling invariants
introduced in this paper. 

\begin{definition} \label{def:diamfcn}
The {\em intrinsic}
[respectively, {\em extrinsic}] {\em \dff}
for a group $G$ with finite presentation $\pp$
is the minimal nondecreasing function 
$f:\N \ra \N$ satisfying the property that
for all $w \in A^*$ with $w=_G \ep$,
there exists a van Kampen diagram $\dd$ for $w$ over $\pp$ such that
for all vertices $v$ in $\dd\skz$
we have 
$d_\dd(*,v) \le f(l(w))$
[respectively, $d_X(\ep,\pi_\dd(v)) \le f(l(w))$].
%The pair $(G,\pp)$ satisfies an {\em extrinsic \dmi}
%for a nondecreasing function $f:\N \ra \N$ if 
%for all $w \in A^*$ with $w=_G \ep$,
%there exists a van Kampen diagram $\dd$ for $w$ over $\pp$ such that
%for all vertices $v$ in $\dd\skz $
%we have 
%$d_X(\ep,\pi_\dd(v)) \le f(l(w))$.
\end{definition}

Since for each $n \in \N$ there are only
finitely many words of length up to $n$,
there is a minimal value for $f(n)$, and so
both diameter functions are well-defined.
See, for example, the exposition
in \cite[Chapter II]{riley} for more details on these diameter 
inequalities and functions.

These \dmc\  functions are rather weak, in that 
although they guarantee that the 
maximum intrinsic (resp. extrinsic)
distance from a vertex
to the basepoint in the diagram $\dd$ 
is at most $f(l(w))$, they
do not measure 
the extent to which vertices at this maximum
distance can occur.
%whether this maximum distance
%occurs for only localized sets of vertices.
For example, in the extrinsic case it may be possible
to have a chain of contiguous vertices lying at the
maximum distance,  
surrounding a region containing vertices
much closer to the basepoint.
In other words, the \dmc\  functions
do not distinguish how wildly or tamely these
maxima occur in van Kampen diagrams.  
The \tffs\ of Definition~\ref{def:tfi} were
designed to measure this tameness.

\begin{proposition} \label{prop:itimpliesid}
If a nondecreasing function $f:\nn \ra \nn$ is
an intrinsic [resp. extrinsic] \tff\ for a
group $G$ with finite presentation $\pp$,
then the function $\hat f:\N \ra \N$ defined
by $\hat f(n) = \lceil f(n) \rceil$
is an upper bound for the 
intrinsic [resp. extrinsic] diameter function
of the
pair $(G,\pp)$.
\end{proposition}

\begin{proof}
We prove this for intrinsic tameness; the extrinsic proof is similar.
Let $w$ be any word over the generating set $A$ from 
%the presentation 
$\pp$ representing the trivial element $\ep$
of $G$, and let $\dd$ be a van Kampen diagram  for $w$ with an
$f$-tame 1-combing $\ps$ of $(\dd,\bo\dd)$.
Since the function $\ps$ is continuous,
each vertex $v \in \dd\skz $ satisfies $v=\ps(p,s)$ for some $p \in \bo \dd$
and $s \in [0,1]$.  There is an edge
path along $\bo \dd$ from $*$ to $p$ labeled by at most
half of the word $w$, and so 
$\td_\dd(*,p) \le \frac{l(w)}{2}$.  Using the facts that 
$p=\ps(p,1)$ and $s \le 1$, the $f$-tame 
condition 
implies that $\td_\dd(*,v) \le f(\td_\dd(*,p))$.
Since $f$ is nondecreasing, $\td_\dd(*,v) \le f(\frac{l(w)}{2})$,
as required.
%and the coarse distance
%$\td_\dd(*,v)$ from $*$ to a vertex $v$ is a natural number,
%then $\td_\dd(*,v) \le \ceil 
\end{proof}

In \cite{bridsonriley}, Bridson and Riley give an example
of a finitely presented group $G$ whose intrinsic and
extrinsic diameter functions are not Lipschitz
equivalent. 
While the relationship
between \tffs\ remains unresolved in general, in
the following (somewhat technical) lemma 
we give bounds
on their interconnections; 
Lemma~\ref{lem:itversuset} 
will be applied in several examples later in this paper. 

\begin{lemma}\label{lem:itversuset}
Let $G$ be a finitely presented group with 
Cayley complex $X$ 
and \cfl\ $\clf=\{(\dd_w,\ps_w) \mid w \in A^*, w=_G \ep\}$.
Suppose that $j:\N \ra \N$ is a nondecreasing function such that
for %$w=_G \ep$ and
every vertex $v$ of 
a van Kampen diagram $\dd_w$ in $\clf$,
$d_{\dd_w}(*,v) \le j(d_X(\ep,\pi_{\dd_w}(v)))$, and
let $\tj:\nn \ra \nn$ be
defined by $\tj(n):=j(\lceil n \rceil)+1$.  
\begin{enumerate}
\item If $\pi_{\dd_w} \circ \ps_w$ is $f$-tame for all $w$,
then $\tj \circ f$ is an intrinsic \tff~for~$G$.
%If $G$ satisfies an extrinsic \tfi\ for the function 
%$f$ with respect to $\clf$, then $G$ 
%satisfies an intrinsic \tfi\ for the function $\tj \circ f$.
\item If $\ps_w$ is $f$-tame for all $w$,
then $f \circ \tj$ is an extrinsic \tff\ for $G$.
%If $G$ satisfies an intrinsic \tfi\ for the function 
%$f$ with respect to $\clf$, then $G$ 
%satisfies an extrinsic \tfi\ for the function $f \circ \tj$.
\end{enumerate}
\end{lemma}

\begin{proof}
We begin by showing that the inequality restriction
for $j$ on vertices holds for the function $\tj$ on all
points in the van Kampen diagrams in $\clf$, using the fact that
coarse distances on edges and 2-cells are closely
linked to those of vertices.
Let $(\dd,\ps)
%=(\dd_w,\ps_w)
 \in \clf$
and let $p$ be any point in $\dd$.
Among the vertices
in the boundary of the open cell of $\dd$ containing $p$, 
let $v$ be the vertex whose coarse
distance to the basepoint $*$ is maximal.
Then $\td_{\dd}(*,p) < \td_{\dd}(*,v)+1$.
Moreover, $\pi_{\dd}(v)$ is again a vertex
in the boundary of the open cell of $X$ 
containing $\pi_{\dd}(p)$, and so
$\td_X(\ep,\pi_{\dd}(v)) \le \lceil \td_X(\ep,\pi_{\dd}(p)) \rceil$.
Applying the fact that $j$ is nondecreasing, then 
$\td_{\dd}(*,p) \le \td_{\dd}(*,v)+1
\le j(\td_{X}(\ep,\pi_{\dd}(v)))+1 \le 
j(\lceil \td_X(\ep,\pi_{\dd_w}(p)) \rceil) +1~.$
Hence the second inequality in 
\[ \td_X(\ep,\pi_{\dd}(p)) \le \td_{\dd}(*,p) 
\le \tj(\td_X(\ep,\pi_{\dd}(p))) \hspace{0.8in} (a)
\]
follows. The first inequality is a consequence
of the fact that coarse distance can only be
preserved or decreased by the map $\pi_\dd:\dd \ra X$.
%from any van Kampen diagram to the Cayley complex.

Now suppose that the composition 
$\pi_\dd \circ \ps:\bo\dd \times [0,1] \ra X$
is $f$-tame.
%that $G$ (with its finite presentation) satisfies an 
%extrinsic \tfi\ for the function 
%$f:\nn \ra \nn$ with respect to $\clf$.
Then for all $p \in \dd$ and for all $0 \le s < t \le 1$, we have
%\begin{align*}
$
\td_{\dd}(*,\ps(p,s)) \le \tj(\td_X(\ep,\pi_{\dd}(\ps(p,s))))
\le \tj(f(\td_X(\ep,\pi_{\dd}(\ps(p,t)))))
\le \tj(f(\td_{\dd}(*,\ps(p,t))))
$
%\end{align*}
where the first and third inequalities follow from $(a)$ and
the fact that $f$ and $\tj$ are nondecreasing, and
the second also uses $f$-tameness of $\ps$. 
%and the hypothesis that $\tj$ is nondecreasing.  
Hence $\ps$ is $(\tj \circ f)$-tame,
completing the proof of (1).
The proof of (2) is similar.
\end{proof}

%\begin{remark}\label{rmk:qgitisqget}
%{\em
%In Theorem~\ref{thm:itversuset}, if the function
%$g$ is linear, then the 
%}
%\end{remark} 

%%%%%%%%%%%%%%%%%%%%%%%%%%%%%%%%%%%%%%%%%%%%%%%%%%%%%%%%%%%%%%%%%%%%%%%%%%%%
%%%%%%%%%%%%%%%%%%%%%%%%%%%%%%%%%%%%%%%%%%%%%%%%%%%%%%%%%%%%%%%%%%%%%%%%%%%%
%%%%%%%%%%%%%%%%%%%%%%%%%%%%%%%%%%%%%%%%%%%%%%%%%%%%%%%%%%%%%%%%%%%%%%%%%%%%

\section{Tame filling functions and tame combings}\label{sec:relax}

%%%%%%%%%%%%%%%%%%%%%%%%%%%%%%%%%%%%%%%%%%%%%%%%%%%%%%%%%%%%%%%%%%%%%%%%%%%%
%%%%%%%%%%%%%%%%%%%%%%%%%%%%%%%%%%%%%%%%%%%%%%%%%%%%%%%%%%%%%%%%%%%%%%%%%%%%

The purpose of this section is twofold.
The main goal is to prove Corollaries~\ref{cor:etistc}
and~\ref{cor:tamecomb},
connecting the concepts of tame combings and \tfs.
Before doing this, we require a pair of results that
provide more tractable methods for constructing combed fillings.
Lemma~\ref{lem:edgeimpliesbdry} will also be applied 
to stackable groups in Section~\ref{subsec:stkbl}, as well as in 
the proof of Theorem~\ref{thm:ac}.  This lemma also
is part of the stronger
Proposition~\ref{prop:htpydomain}, which 
will be essential to the proofs of Corollary~\ref{cor:etistc}
and Theorem~\rfthmitisqi.

The objective of Lemma~\ref{lem:edgeimpliesbdry} is to reduce the 
set of diagrams required to construct a combed filling.
Let $\cc=\{y_g \mid g \in G\} \subset A^*$ be a set of normal
forms (including the empty word) for $G=\langle A \mid R \rangle$ 
that label simple paths in the Cayley complex $X$.
An {\em \edg}\label{defedg} is a van Kampen diagram $\dd$
with boundary label $y_gay_{ga}^{-1}$ for some $g \in G$ and $a \in A$.
An {\em \ehy}\label{defehy} of $\dd$
is a 1-combing 
%$\tht:\he \times [0,1] \ra \dd$ 
$\tht$
of the pair $(\dd,\he)$ 
at the basepoint $*$,
where $\he$ is the 1-cell underlying the edge
$\tilde e:=\path_\dd(y_g,a)$
in $\bo \dd$ corresponding to $a$, % in the boundary label,
such that the paths 
$\tht(\init(\tilde e),\cdot),\tht(\term(\tilde e),\cdot):[0,1] \ra \dd$
to the endpoints of $\he$ follow the paths
labeled $y_g,y_{ga}$ in $\bo\dd$ from $*$.
(See Figure~\ref{fig:seashellcomb}).
A {\em \cnf}~\label{defcnf} %for the pair $(G,\pp)$ 
is a collection
$\cle=\{(\dd_e,\tht_e) \mid e \in E_X\}$
such that for each $e \in E_X$,  
$\dd_e$ is an \edg\ with  \ehy\ 
$\tht_e$ associated to the elements $g \in G$ and $a \in A$
for one of the directed edges $\ega$ corresponding to $e$,
 and 
the {\oc}s satisfy
the following {\em gluing condition}:
For every pair of edges $e,e' \in E_X$
with a common endpoint $g$, we require that  
$\pi_{\dd_e} \circ \tht_e(\hat g_e,t)=
\pi_{\dd_{e'}} \circ \tht_{e'}(\hat g_{e'},t)$ for all 
$t$ in $[0,1]$, where $\hat g_e$ and $\hat g_{e'}$ are
the vertices of $\he$ in $\dd_e$ and $\hat{e'}$ in $\dd_{e'}$
mapping to the vertex $g$ in $X$, respectively; that is, at these vertices
$\tht_e$ and $\tht_{e'}$
project to the same path, with the same
parametrization, in the Cayley complex $X$.
A \cnf\ is 
{\em geodesic} if all of the words in the
normal form set $\cc$ 
label geodesics in the associated Cayley graph.

In the following we extend  the ``seashell'' 
(``cockleshell'' in \cite[Section~1.3]{riley})
method of constructing
a filling from {\edg}s, to build 
a combed filling from a {\cnf}.

\begin{lemma}\label{lem:edgeimpliesbdry}
A \cnf\ $\cle=\{(\dd_e,\tht_e) \mid e \in E_X\}$ 
for a group $G$ over a finite presentation $\pp$
induces a \cfl\ $\{\dd_w,\ps_w) \mid w=_G \ep\}$, 
such that
if $\pi_{\dd_e} \circ \tht_e$ is $f$-tame for all $e$,
then each $\pi_{\dd_w} \circ \ps_w$ is also $f$-tame.
Moreover in the case
that the \cnf\ is geodesic,
if every \ehy\ $\tht_e$ is $f$-tame, then
every \vky\ $\ps_w$ is also $f$-tame.
\end{lemma}

\begin{proof}
%Let $\cc=\{y_g \mid g \in G\} \subset A^*$ be a set of 
%simple word normal
%forms for $G$ over $A$. 
Given a word $w=a_1 \cdots a_n$ with $w=_G \ep$ and
each $a_i \in A$, let $(\dd_i,\tht_i)$ be the
element of $\cle$ corresponding to the 
%directed edge  $e_{a_1 \cdots a_{i-1},a_i}$ 
edge of $X$
with endpoints $a_1 \cdots a_{i-1}$ and
$a_1 \cdots a_i$ and label $a_i$.  
By replacing $\dd_i$ by its mirror image if necessary,
we may assume that $\dd_i$ has boundary labeled by the word
$y_{i-1} a_i y_i$, where $y_i$ is the normal
form in $\cc$ of $a_1 \cdots a_i$. 

We next iteratively build a van Kampen diagram 
$\dd_i'$ for the word 
$y_{\ep}a_1 \cdots a_i y_{i}^{-1}$,
beginning with $\dd_1':=\dd_1$.  For $1<i \le n$, 
the planar diagrams $\dd_{i-1}'$ and $\dd_{i}$
have boundary subpaths
sharing a common label by the simple word $y_{i}$.
The fact that $y_i$  labels
a simple path in $X$
implies that 
%any path in a van Kampen
%diagram labeled by $y_i$ must also be simple, and hence
each of these boundary paths in $\dd_{i-1}'$,$\dd_{i}$
is an embedding.
These paths are also oriented in the same direction,
and so the diagrams $\dd_{i-1}'$ and $\dd_{i}$ can be
glued, starting at their basepoints and
folding along these subpaths,
to construct the 
planar diagram $\dd_i'$.  
Performing these gluings consecutively for each
$i$ results in a van Kampen diagram $\dd_n'$ with
boundary label $y_{\ep}wy_{w}^{-1}$.
Note that we have allowed the
possibility that some of the boundary edges of $\dd_n'$
may not lie on the boundary of a 2-cell in $\dd_n'$;
some of the words $y_{i-1} a_i y_i$ may freely
reduce to the empty word, and the corresponding
van Kampen diagrams $\dd_i$ may have no 2-cells.
Note also that
the only simple word representing the identity of $G$
is the empty word; that is, $y_\ep =y_{w}=1$.  
Hence $\dd_n'$ is a van Kampen diagram for $w$.

Let $\dd_w:=\dd_n'$, and let
%This procedure yields a quotient map 
$\alpha:\coprod \dd_i \ra \dd_w$ be the quotient map.
Then each
restriction $\alpha|:\dd_i \ra \dd_w$ is an embedding.
%Hence we may consider each $\dd_i$
%as a subset of $\dd_w$. 
%Let $\he_i$ be the edge of $\dd_i$ corresponding
%to the letter $a_i$ in the boundary path; then
%$\he_i$ also lies in $\bo \dd_w$.
Let $\he_i:=\path_{\dd_i}(y_{i-1},a_i)$ be 
the $a_i$ edge in the boundary path
of $\dd_i$ (and by slight abuse of notation also in the
boundary of $\dd_w$).
%Note that we have allowed the
%possibility that some of these edges
%may not lie on the boundary of a 2-cell in $\dd_w$;
%some of the words $x_i$ may freely
%reduce to the empty word, and the corresponding
%van Kampen diagrams $\dd_i$ may have no 2-cells.
In order to build a \vky\ on
$\dd_w$, we note that the \ehs\ 
$\tht_i$ give a continuous function 
$\alpha \circ \coprod \tht_i: 
\coprod \he_i \times [0,1] \ra \dd_w$.
The gluing condition in the definition of \cnf\ implies that 
%%$\pi_{\dd_i} \circ \tht_i(v_i'',t)=
%%\pi_{\dd_{i+1}} \circ \tht_{i+1}(v_{i+1}',t)$
%%for all $t \in [0,1]$; that is,
on the common endpoint 
%%$v_i:=\alpha(v_i'')=\alpha(v_{i+1}')$ 
$v_i$ of the edges $\he_i$
and $\he_{i+1}$ of $\dd_w$, the paths
$\pi_{\dd_i} \circ \tht_i(v_i,\cdot)$ and
$\pi_{\dd_{i+1}} \circ \tht_{i+1}(v_i,\cdot)$
follow the edge path in $X$ labeled $y_i$ with the same parametrization.
Hence the same is true for the functions
$\tht_i(v_i,\cdot)$ and $\tht_{i+1}(v_{i+1},\cdot)$
following the edge paths labeled $y_i$ that are glued by
$\alpha$.
Moreover, if an
$\he_i$ edge and (the reverse of) an $\he_j$ edge are glued
via $\alpha$, the maps $\tht_i$ and $\tht_j$ have been
chosen to be consistent.
Hence the collection of maps $\tht_i$
are consistent on points identified by the
gluing map $\alpha$,
and we obtain
an induced function
$\ps_w:\bo \dd_w \times [0,1] \ra \dd_w$.
Moreover, the function $\ps_w$ satisfies all
of the properties needed for the required \vky\ 
on the diagram $\dd_w$.
Let $\clf=\{(\dd_w,\ps_w) \mid w \in A^*, w =_G \ep\}$ 
be the induced \cfl\ 
from this ``seashell'' procedure.
(See Figure~\ref{fig:seashellcomb}.)
\begin{figure}
\begin{center}
%\includegraphics[width=2.2in,height=1.1in]{tf_fig2_diskvkhtpy.eps}%old disk.eps
%\hspace{.3in} 
\includegraphics[width=1.6in,height=1.1in]{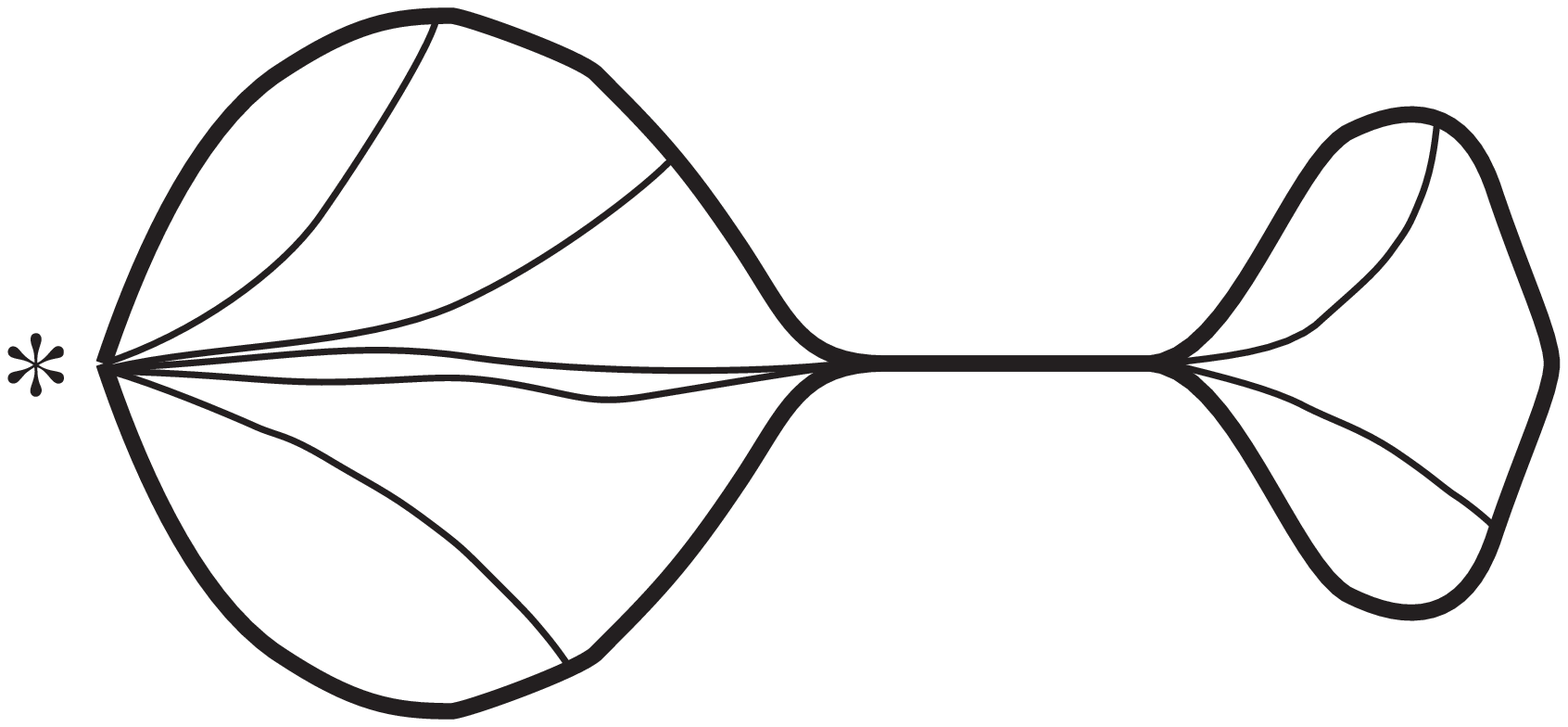}%old vkd.eps
\hspace{.2in}
\includegraphics[width=1.6in,height=1.1in]{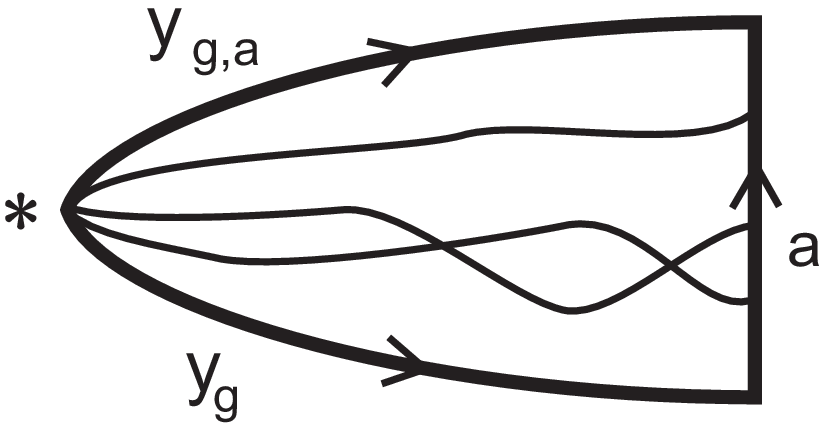}%old vkd.eps
\hspace{.2in}
\includegraphics[width=2.3in,height=1.5in]{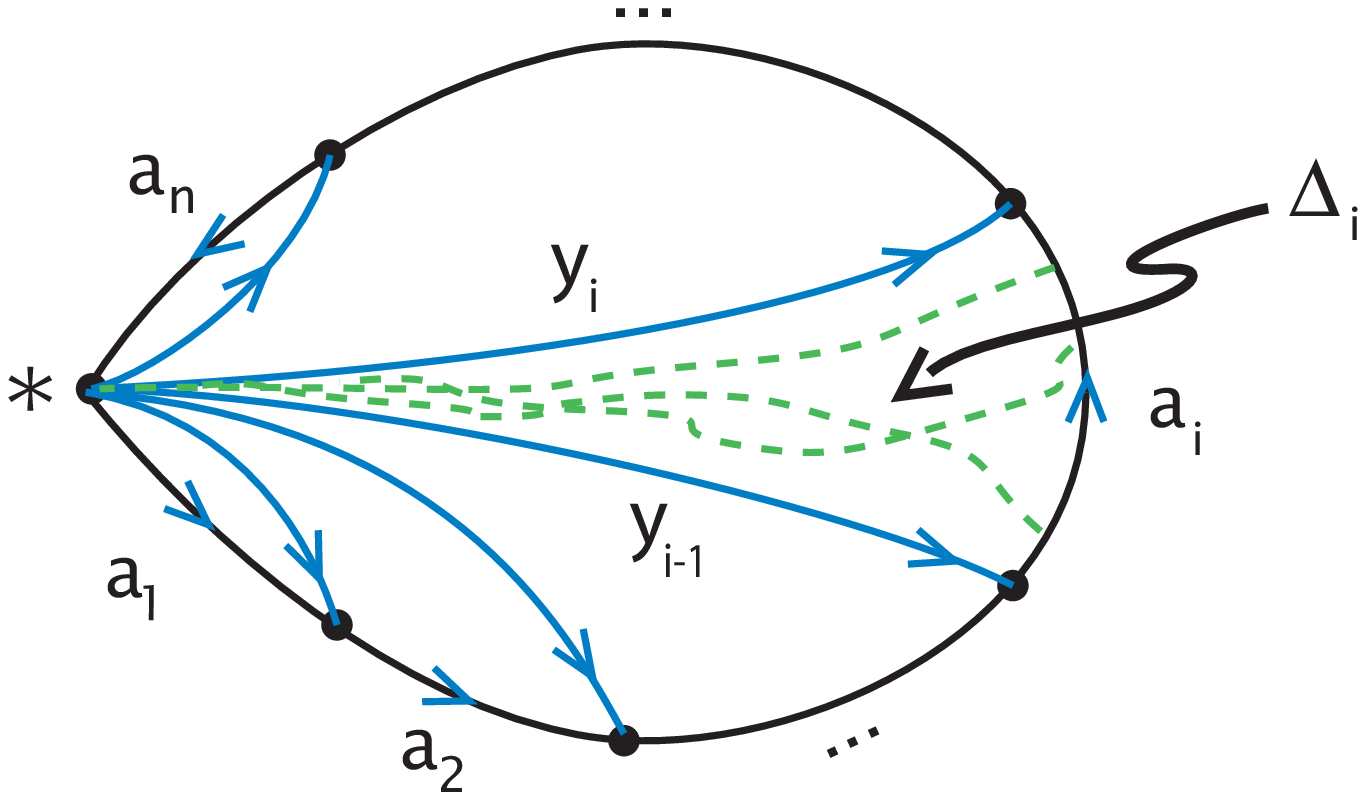}
\caption{Boundary 1-combing, \ehy, and seashell process}\label{fig:seashellcomb}
\end{center}
\end{figure}

In the extrinsic setting,
the seashell quotient map $\alpha$ preserves 
extrinsic distances (irrespective of whether or
not the normal forms are geodesics).
That is, for any point $q$ in $\dd_i$, 
we have $\pi_{\dd_i}(q)=\pi_{\dd_w}(\alpha(q))$, and 
hence $\td_X(\ep,\pi_{\dd_i}(q))=\td_X(\ep,\pi_{\dd_w}(\alpha(q)))$.
Now for each $p \in \bo \dd_w$, 
we have $p=\alpha(p')$ for a point 
$p'$ in $\he_i \subseteq \dd_i$ for some $i$, and 
$\ps_w(p,t)=\alpha(\tht_i(p',t))$ for all
$t \in [0,1]$.  Thus if each $\pi_{\dd_i} \circ \tht_i$ is an
$f$-tame map, then
$\pi_{\dd_w} \circ \ps_w$ is also $f$-tame.
%Thus extrinsic $f$-tameness of homotopies is preserved
%by the seashell construction.

Next suppose that each of the words in the set
$\cc$ of  normal
forms are geodesic.
%, and that each \ehy\  
%$\tht_e$ intrinsically $f$-tame.
To show that 
%each induced van Kampen homotopy
%$\ps_w$ is also intrinsically $f$-tame,
$f$-tameness of {\oc}s is
preserved by the seashell construction in this case,
it similarly suffices to show that the map $\alpha$
preserves (intrinsic) coarse distance.

In order to analyze coarse distances in
the van Kampen diagram $\dd_w$, 
we begin by supposing that $p$ is any vertex in $\dd_w$.  
Then $p=\alpha(q)$ for some
vertex $q \in \dd_i$ (for some $i$).  The identification
map $\alpha$ cannot increase distances to the 
basepoint, so we have $d_{\dd_w}(*,p) \le d_{\dd_i}(*,q)$.
Suppose that $\beta:[0,1] \ra \dd_w$ is an edge path 
in $\dd_w$ from $\beta(0)=*$
to $\beta(1)=p$ of length strictly less than $d_{\dd_i}(*,q)$.
This path cannot stay in the (closed)
subcomplex $\alpha(\dd_i)$ of $\dd_w$,
and so there is a minimum time $0 < s \le 1$ such
that $\beta(t) \in \alpha(\dd_i)$ for all $t \in [s,1]$.
Then the point $\beta(s)$ must lie on the image
of the boundary of $\dd_i$ in $\dd_w$.  Since the
words $y_{i-1}$ and $y_i$ label geodesics in the
Cayley complex of the presentation $\pp$, 
these words must also label geodesics in
$\dd_i$ and $\dd_w$.  Hence we can replace the
portion of the path $\beta$ on the interval $[0,s]$
with the geodesic path along one of these words
from $*$ to $\beta(s)$, to obtain a new edge path
in $\alpha(\dd_i)$ from $*$ to $p$ of length 
strictly less than $d_{\dd_i}(*,q)$.  Since
$\alpha$ embeds $\dd_i$ in $\dd_w$, this
results in a contradiction.  
Thus for each vertex $p=\alpha(q)$ in $\dd_w$,
we have $d_{\dd_w}(*,p)=d_{\dd_i}(*,q)$.
Since the coarse distance from the 
basepoint to 
any point in the interior of an
edge or 2-cell in $\dd_w$ is computed
from path metric  distances of vertices,
this also shows that for any
point $p$ in $\dd_w$ with $p=\alpha(q)$ for
some point $q \in \dd_i$, we have
$\td_{\dd_w}(*,p)=\td_{\dd_i}(*,q)$, as required.
%For any point in the interior of an
%edge or 2-cell in $\dd_w$, the coarse
%distance in $\dd_w$ to the basepoint is computed
%from the path metric  distances of the vertices
%in the boundary of the cell.  Hence the result
%of the previous paragraph shows that for any
%point $p$ in $\dd_w$ with $p=\alpha(q)$ for
%some point $q \in \dd_i$, we have
%$\td_{\dd_w}(*,p)=\td_{\dd_i}(*,q)$.  That is,
%the map $\alpha$ preserves coarse distance.
\end{proof}

The heart the proof of Corollary~\ref{cor:etistc}
consists of showing that the converse of 
Lemma~\ref{lem:edgeimpliesbdry} holds; that is,
that every \cfl\ induces a \cnf\ such that 
intrinsic and extrinsic tameness is preserved
(up to Lipschitz equivalence).  However,
this converse requires more work, because the
domain of the \oc\ essentially needs to be rerouted
from the entire boundary of a van Kampen diagram
to a single edge. 

In order to accomplish this, we first need to
``lift" the domain from the boundary of the van Kampen
diagram up to a circle.
%, in a sense ``relaxing" constraints
%on how points on the boundary may have been identified.
%In Definition~\ref{def:1comb},
Our definition of a \vky\  
$\ps:\bo \dd \times [0,1] \ra \dd$
is ``natural'', in the sense that the
first factor in the domain of this function 
is a subcomplex of the associated van Kampen
diagram $\dd$; however, this
requires that for each point $p$ on an
edge $e$ of $\bo \dd$,
there is a unique choice of path from
the basepoint $*$ to $p$ via this \oc.
When traveling along the boundary $\bo \dd$ 
counterclockwise, a
point $p$ (and undirected edge $e$) may be traversed more
than once, and we use our lift to relax this constraint and allow
different combings of this point corresponding
to the different traversals.  
%In this section, we also show that 
%a \tfi\ with respect to this
%more relaxed condition is equivalent to
%the \tfi\ defined in Section~\ref{intro}.
%Replacing the homotopy domain with a circle.

%Next we turn to relaxing the boundary condition.
%Let $S^1$ denote the unit circle in the
%${\mathbb R}^2$ plane.
For any natural number $n$, let 
$C_n$ be the Euclidean circle $S^1$
with a 1-complex structure consisting of
$n$ vertices (one of which is the basepoint
$(-1,0)$) and $n$ edges.
%Given any van Kampen diagram $\dd$ over $\pp$ for a word
%$w$ of length $n$, let $\vt_\dd:C_n \ra \bo \dd$
%be a continuous function that maps $(-1,0)$ to $*$
%and, going counterclockwise once around $C_n$,
%maps each subsequent edge of $C_n$ homeomorphically onto
%the next edge in the counterclockwise path labeled $w$
%along the boundary of $\dd$.  
%Since  no element of the generating
%set $A$ represents the identity of $G$, 
%the edges of $\dd$ have distinct endpoints, and so
%for each edge $e$ of $C_n$, the map $\vt_\dd|_e$ is
%a homeomorphism.
%
A {\em \dhy} of a van Kampen diagram $\dd$
over $\pp$ for a word $w$ of length $n$ is a continuous 
function $\ph:C_n \times [0,1] \ra \dd$ such that
\begin{itemize}
\item[(d1)] $\ph(p,0)=*$ for all $p \in C_n$, and
the function $\ph(\cdot,1):C_n \ra \dd$ 
satisfies $\ph((-1,0),1)=*$ and, going counterclockwise
around $C_n$, maps each subsequent edge of $C_n$ 
homeomorphically onto the next edge in $\path_\dd(1,w)$,
% the counterclockwise path along $\bo\dd$.
%
%$\ph(p,1)=\vt_\dd(p)$ ,
\item[(d2)] $\ph((-1,0),t)=*$ for all $t \in [0,1]$,  
and
\item[(d3)] if $p \in C_n\skz$, 
then $\ph(p,t) \in \dd\sko$ for all $t \in [0,1]$.
\end{itemize}
That is, 
for each edge $e$ of $C_n$, the set $\he:=\ph(e \times \{1\})$
is an edge of $\bo\dd$ and 
$\ph \circ (\ph(\cdot,1)|_e^{-1} \times \idmap):
\he \times [0,1] \ra \dd$
is a \oc\ 
%of the pair $(\dd,\vt_\dd(e))$ 
based at  $*$; although strictly
speaking $\ph$ is not a \oc\ since $C_n$ is not
a subcomplex of $\dd$, we use this terminology
to express the connection to these {\oc}s for edges.
A {\em \scf} is a collection 
$\cld=\{(\dd_w,\ph_w)\mid w \in A^*, w=_G \ep\}$
such that for each $w$, $\dd_w$ is a van Kampen
diagram for $w$, and $\ph_w$ is a \dhy\ of $\dd_w$.

\begin{remark}~\label{rmk:infinite}
{\em
If the group $G$ is finite, then in Section~\ref{subsec:finite}
we show that there is a filling 
$\{\dd_w\mid w \in A^*, w=_G \ep\}$ 
and a constant $C$ such that 
the intrinsic and extrinsic diameters of 
each van Kampen diagram $\dd_w$ are bounded above
by $C$, and so the \dhs\ associated to
any \scf\ $\cld=\{(\dd_w,\ph_w)\mid w \in A^*, w=_G \ep\}$
built from this filling satisfy the property that
$\ph_w$ and $\pi_{\dd_w} \circ \ph_w$ are both
tame with respect to the constant function
$n \mapsto C$.  In contrast, suppose that $G$ is infinite,
$f:\nn \ra \nn$ is a nondecreasing function,
and $\cld=\{(\dd_w,\ph_w)\mid w \in A^*, w=_G \ep\}$
is a \scf\ for $G$ such that each 
\dhy\ $\ph_w$ 
is $f$-tame.
Then for any natural number $n$ there is
a word $w_n$ labeling a geodesic in the
Cayley graph for $G$; let
$\ph := \ph_{w_nw_n^{-1}}$. 
% : C_{2n} \times [0,1] \ra \dd_{w_nw_n^{-1}}$.  
Now the path $\ph(1,\cdot)$
from $\ph(1,0)=*$ ends at the vertex
$p := \ph(1,1)=\term(\path_{\dd_{w_nw_n^{-1}}}(1,w_n))$
%on $\bo \dd_{w_nw_n^{-1}}$, and $p$
which lies at a coarse distance $n$ from $*$
in $\dd_{w_nw_n^{-1}}$.  The path $\ph(1,\cdot)$
must traverse another cell $\sigma$ of $\dd_{w_nw_n^{-1}}$
immediately before reaching $p$, and 
since the vertex $p$ lies in $\bo \sigma$,
the coarse distance from $*$ to any point
of $\sigma$ must be at least $n-\frac{3}{4}$.
Thus $f$ satisfies
$f(n) \ge n-\frac{3}{4}$ for $n \in \N$, 
and so $f(n) \ge f(\lfloor n \rfloor) \ge \lfloor n \rfloor - \frac{3}{4}
> n-2$ for all $n \in \nn$.
A similar argument shows that if instead we assume
that each $\pi_{\dd_w} \circ \ph_w$ is $f$-tame,
then again $f(n) > n-2$ for all $n \in \nn$.
}%em
\end{remark}

In Proposition~\ref{prop:htpydomain} below, the
extrinsic result $(1) \Leftrightarrow (4)$ (or $(3)$) is used in the proof 
of Corollary~\ref{cor:etistc}, and the equivalence
$(1) \Leftrightarrow (2)$ in both the intrinsic
and extrinsic cases is used in the
quasi-isometry invariance proof for Theorem~\rfthmitisqi.

\begin{proposition}\label{prop:htpydomain}
Let $G$ be a group with a 
finite symmetric presentation $\pp$,
and let $f:\nn \ra \nn$ be a nondecreasing function.
The following are equivalent, up to Lipschitz equivalence
of the function $f$:
\begin{enumerate}
\item $f$ is an intrinsic [respectively, extrinsic] 
\tff\ for $(G,\pp)$.
\item $(G,\pp)$ has an \scf\ 
$\cld=\{(\dd_w,\ph_w)\mid w=_G \ep\}$ 
such that each \dhy\ $\ph_w$ [respectively, $\pi_{\dd_w} \circ \ph_w$]
is $f$-tame.
\item $(G,\pp)$ has a geodesic \cnf\ 
$\cle=\{(\dd_e,\tht_e) \mid e \in E_X\}$
such that each \ehy\ $\tht_e$ [respectively, $\pi_{\dd_e} \circ \tht_e$]
is $f$-tame.
\end{enumerate}
In addition, in the extrinsic case:
\begin{itemize}
\item[(4)] $(G,\pp)$ has a  \cnf\ 
$\cle=\{(\dd_e,\tht_e) \mid e \in E_X\}$
such that each $\pi_{\dd_e} \circ \tht_e$ 
is $f$-tame.
\end{itemize}
\end{proposition}

\begin{proof}
%Write the presentation for $G$ as $\pp=\langle A \mid R \rangle$.
In the extrinsic case, the result $(3) \Rightarrow (4)$ is
immediate.  The implications $(4) \Rightarrow (1)$ in the
extrinsic case and $(3) \Rightarrow (1)$ in the intrinsic case
are Lemma~\ref{lem:edgeimpliesbdry}. \\
\noindent $(1) \Rightarrow(2)$: \\
Given a van Kampen diagram
$\dd$ for a word $w$ of length $n$
and a \vky\ 
$\ps:\bo \dd \times [0,1] \ra \dd$,
let $\vt:C_n \ra \dd_n$ be any cellular map
that is a homeomorphism on closed 1-cells
taking $(-1,0)$ to $*$ and mapping $C_n$
(counterclockwise) along $\path_\dd(1,w)$.
Then the composition $\ph=\ps \circ (\vt_\dd \times \idmap) :
C_n \times [0,1] \ra \dd$ is a \dhy\ for this diagram.
The fact that the identity function is used on the
[0,1] factor implies that tameness of the {\oc}s
is preserved.

\noindent $(2) \Rightarrow(3)$:  

%Suppose that
%$\clf=\{(\dd_w,\ph_w) \mid w \in A^*, w=_G \ep\}$
%is a \scf.
%%collection of van Kampen diagrams and disk homotopies
%%such that 
%%each $\ph_w$ is intrinsically $f$-tame.

Let $\cc:=\{y_g \mid g \in G\}$ be the set of 
shortlex normal forms with respect to some 
total ordering of $A$.  
%A set of geodesic normal forms that we will 
%use several times in this paper is the set of
%{\em shortlex} normal forms.  
%Choose a (lexicographic) total ordering on the
%finite set $A$.   
That is, for any two words $z,z'$ over
$A$, define $z \slex z'$ if 
the word lengths satisfy $l(z)<l(z')$,
or else $l(z)=l(z')$ and $z$ is less than 
$z'$ in the
corresponding lexicographic ordering on $A^*$.
%, and
%for each element $g \in G$, let $z_g$ denote the
%shortlex least word over $A$ that represents $g$
%(i.e., the shortlex normal form for $g$).
%The set $\{y_g \mid g \in G\}$ is the set of 
% shortlex normal forms.

For any edge $e \in E_X$ with endpoints $g$ and $ga$, 
we orient the edge
$e$ from vertex $g$ to $ga$ if $y_g \slex y_{ga}$. 
There is a pair $(\dd_{w_e},\ph_{w_e})$ in $\cld$
associated to the word $w_e:=y_gay_{ga}^{-1}$.
Define $\dd_e:=\dd_{w_e}$, and
let $\he:=\path_{\dd_e}(y_g,a)$ 
be the edge in 
%the boundary path labeled $w_e$ of $\dd_{e}$ 
$\bo\dd_e$ associated to the
letter $a$ in $w_e$.

We construct an \ehy\  
$\tht_{e}: \he \times [0,1] \ra \dd_e$ as follows
(and depicted in Figure~\ref{fig:vktoedgehtpy}).
Let $\gamma: \he \ra C_{l({w_e})}$ be a continuous map that
wraps the edge $\he$ once (at constant speed)
in the counterclockwise direction
along the circle, with the endpoints of $\he$
mapped to $(-1,0)$.  
For each point $p$ in $\he$, and for all
$t \in [0,\frac{1}{2}]$, define 
$\tht_e(p,t):=\ph_{w_e}(\gamma(p),2t)$.
%Then $\tht_e(p,\frac{1}{2})=\vt_{\dd_e}(\gamma(p)) \in \bo \dd_e$. 

Let $\hhe$ be the (directed) edge of $C_{l({w_e})}$ corresponding
to the edge $\he$ of the boundary path in $\bo \dd_e$, 
with endpoint $v_1$ of $\hhe$ occurring earlier than endpoint $v_2$
in the counterclockwise path beginning at $(-1,0)$.
Let $\vt_e:C_{l({w_e})} \ra \bo \dd_e$ be the map 
$\vt_e(q):=\ph_{w_e}(q,1)$
wrapping $C_{l({w_e})}$ around $\bo \dd_e$.
Also let $\hhr_1$, $\hhr_2$ be the arcs of $C_{l({w_e})}$ mapping
via $\vt_e$ to
$\path_\dd(1,y_g)$ and $\path_\dd(y_ga,y_{ga}^{-1})$, respectively, 
%the paths labeled by the subwords 
%$y_g$, $y_h^{-1}$,
%respectively, of ${w_e}$ 
in $\bo \dd_e$.
For each point $p$ in the interior %$Int(\he)$ 
of the edge $\he$,
there is a unique point $\hhp$ in $\hhe$ with
$\vt_e(\hhp)=p$.
There is an arc (possibly a single point) in $C_{l({w_e})}$
from $\gamma(p)$ to $\hhp$ that is disjoint from the point
$(-1,0)$; let $\delta_p:[\frac{1}{2},1] \ra C_{l({w_e})}$
be the constant speed path following this arc.
That is, $\vt_e \circ \delta_p$ is a path in
$\bo \dd_e$ from $\vt_e(\gamma(p))$ to $p$.
In particular, if $\gamma(p)$ lies in $\hhr_1$, then
the path $\vt_e \circ \delta_p$ follows the end portion
of the boundary path labeled by $y_g$ from 
$\vt_e(\gamma(p))$ to the endpoint $\vt_e(v_1)$ 
of $\he$
and then follows a portion of $\he$ to $p$.  If $\gamma(p)$
lies in $\hhr_2$, the path $\vt_e \circ \delta_p$
follows a portion of the boundary path $y_h$
and $\he$  clockwise
from  $\vt_e \circ \delta_p$ via $\vt_e(v_2)$ to $p$,
and if $\gamma(p)$ is in $\hhe$, then
the path $\vt_e \circ \delta_p$ remains in $\he$.
Finally, for either endpoint $p=\vt_e(v_i)$ (with $i=1,2$),
let $\delta_p:[\frac{1}{2},1] \ra C_{l({w_e})}$
be the constant speed path along the arc $\hhr_i$ in $C_{l({w_e})}$
from $(-1,0)$ to $v_i$.  
%(Note that if the edge $\he$
%is a loop with a single endpoint $\vt_{\dd_e}(v_1)=\vt(v_2)$,
%then we have $y_g=y_h$, and although the path
%$\delta_p$ is not well-defined, the path 
%$\vt_{\dd_e} \circ \delta_p$ is well-defined.)
Now for all $p$ in $\he$ and $t \in [\frac{1}{2},1]$,
define
$\tht_e(p,t):=\ph_{w_e}(\delta_p(t),1)$.

Combining the last sentences of the previous
two paragraphs, we have constructed a continuous function
$\tht_e:\he \times [0,1] \ra \dd_e$.
See Figure~\ref{fig:vktoedgehtpy} for an illustration 
of this map.
\begin{figure}
\begin{center}
\includegraphics[width=3.4in,height=1.2in]{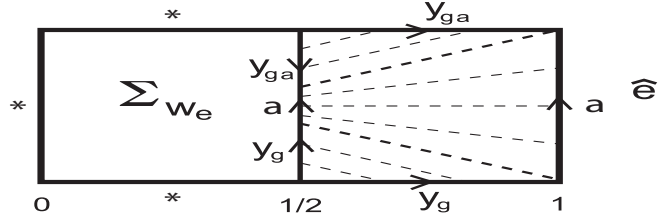}
\caption{Edge \oc\ $\tht_e$ in Proposition~\ref{prop:htpydomain} 
proof (of (2) $\Rightarrow$ (3))}\label{fig:vktoedgehtpy}
\end{center}
\end{figure}
The \dhy\ conditions satisfied by $\ph_{w_e}$
imply that $\tht_e$ is an \ehy.  
%We define this
%\ehy\ $\tht=\tht_{(\dd,\ph,e)}$ to be
%the {\em canonical \ehy} associated to
%the tuple $(\dd,\ph,e)$.
Let $\cle$ be the collection
$\cle=\{(\dd_e,\tht_e) \mid 
e \in E_X \}$.
Then $\cc$ together with $\cle$ define a geodesic \cnf\ of the
pair $(G,\pp)$.

Now we turn to analyzing the tameness of the
\ehy\  $\tht_e:\he \times [0,1] \ra \dd_e$.
We will give the proof for the extrinsic case;
the intrinsic proof is nearly identical.
Suppose that $\pi_{\dd_w} \circ \ph_w$
is $f$-tame for each \dhy\ $\ph_w$ of the \scf\ $\cld$.

Suppose that $p$ is any point in 
$\he$.  Then $f$-tameness of $\pi_{\dd_w} \circ \ph_{w_e}$
%The fact
%that the disk homotopy $\ph_{w_e}$ is intrinsically
%$f$-tame 
implies that for all
$0 \le s<t \le \frac{1}{2}$, we have
$\td_{X}(\ep,\pi_{\dd_e}(\tht_e(p,s))) \le f(\td_{X}(\ep,\pi_{\dd_e}(\tht_e(p,t))))$.

The path
$\pi_{\dd_e}(\tht_e(p,\cdot))$ 
% =\pi_{\dd_e}(\vt_e(\delta_p(\cdot))):[\frac{1}{2},1] \ra X$ 
on the second half of the interval $[0,1]$ follows a
portion of a geodesic in $X$ (labeled $y_g$ or $y_{ga}$)
going steadily away from the basepoint
$\ep$, with the possible exception
of the end portion of this path that
is completely contained in the edge $e$.
Hence for all $\frac{1}{2} \le s<t \le 1$, we have
$\td_{X}(\ep,\pi_{\dd_e}(\tht_e(p,s))) \le \td_{X}(\ep,\pi_{\dd_e}(\tht_e(p,t)))+1$.

Finally, whenever $0 \le s<\frac{1}{2}<t \le 1$,
we have

\hspace{.4in} $\td_{X}(\ep,\pi_{\dd_e}(\tht_e(p,s))) \le 
  f(\td_{X}(\ep,\pi_{\dd_e}(\tht_e(p,\frac{1}{2}))))
\le f(\td_{X}(\ep,\pi_{\dd_e}(\tht_e(p,t)))+1),$

\noindent where the latter inequality uses the nondecreasing property
of $f$.

Putting these three cases together, the \oc\ 
 $\pi_{\dd_e} \circ \tht_e$ is  $g$-tame
with respect to the nondecreasing function
$g:\nn \ra \nn$ given by $g(n)=f(n+1)+n+1$ for all $n \in \nn$.
%Now $f(n) \le f(n+1) \le g(n)$, and
In the case that $G$ is an infinite group,
Remark~\ref{rmk:infinite} shows that 
$g(n) < 2f(n+1)+2$, and so
$g$ is Lipschitz
equivalent to $f$.

If instead $G$ is a finite group, then 
Remark~\ref{rmk:infinite} says  
there is a \scf\ 
$\cld'=\{(\dd_w',\ph_w')\mid w \in A^*, w=_G \ep\}$
such that each \dhy\ $\ph_w'$ 
[respectively, $\pi_{\dd_w'} \circ \ph_w'$]
is $h$-tame with respect to the constant
function $h(n) \equiv C$.  
Applying the procedure above to construct a
geodesic \cnf\ $\cle'=\{(\dd_e',\tht_e') \mid e \in E_X\}$
from $\cld'$ instead,
%rather than from $\cld$, 
then since each van Kampen diagram
$\dd_e'$ has diameter bounded by $C$, we have
that each \ehy\ $\tht_e'$ [respectively, 
each function $\pi_{\dd_e'} \circ \tht_e'$]
is $h$-tame.
Then each $\tht_e'$ [respectively, $\pi_{\dd_e'} \circ \tht_e'$]
is also $g$-tame for the function $g:=f+h$ that is Lipschitz
equivalent to~$f$.
\end{proof}

Now we are ready to turn to the concept of tame combability.
The {\oc}s considered in~\cite{hmeiermeastame} and here 
satisfy more restrictions than those of
Mihalik and Tschantz~\cite{mihaliktschantz}, 
in that the \oc\ lifts to
van Kampen diagrams.  
More precisely, a {\em \dia\  \oc} of a
Cayley complex $X$ is a \oc\ 
%$\up:\xx \times [0,1] \ra X$ 
$\up$ of the pair $(X,\xx)$ based at $\ep$
that satisfies:
\begin{itemize}
\item for all $v \in X\skz$,
the path $\up(v,\cdot)$ follows a simple path
%(i.e., no repeated vertices or edges),   
%from $\ep$ to $v$ 
labeled by a word $w_v$, 
%with $w_\ep=1$, 
and
\item whenever $e$ is a directed edge from 
vertex $u$ to vertex $v$ in $X$ labeled by $a$, then
there is a van Kampen diagram $\dd$ with respect to
$\pp$ for the word $w_uaw_v^{-1}$, 
together with
an \ehy\  $\tht:\hat e \times [0,1] \ra \dd$ 
associated to the edge $\hat e$ of $\bo \dd$ corresponding
to the letter $a$ in this boundary word,
such that 
$\up \circ (\pi_\dd \times \idmap)|_{\hat e \times [0,1]}
=\pi_{\dd} \circ \tht$.
\end{itemize}

A nondecreasing function $f:\nn \ra \nn$ is called
a {\em \tcf}\label{def:tcf} for a group $G$ with respect to a finite
symmetric presentation $\pp$ if there is an $f$-tame 
\dia\ \oc\ of the Cayley complex.
We note that this concept is Lipschitz equivalent to the
``radial tameness function'' defined in~\cite{hmeiermeastame}; 
there, coarse distance in $X$
is described in terms of ``levels'', and the 
definition of coarse distance for a 2-cell 
is defined slightly differently from that on p.~\ref{def:coarsedist}
in Section~\ref{sec:intro}.

%\coment{Mark and I believe that tame combings can always 
%be converted to diagrammatic ones.  We started
%a proof using barycentric subdivision/simplicial
%approximation.  Is it worth it to try to finish this
%proof and put it in?  Is there a reference we can
%make to the literature that shows that every 
%map $D^2 \ra X$ can be quasigeodesically deformed
%to a map $\dd \ra X$?}

The equivalence of extrinsic \tffs\ with 
\tcfs\  in Corollary~\ref{cor:etistc}
now follows from the observation that
\dia\ 1-combings are projections of 
{\cnf}s in the Cayley complex, along with
Proposition~\ref{prop:htpydomain}.

\begin{corollary}\label{cor:etistc}
Let $G$ be a finitely presented group.
Up to Lipschitz equivalence of nondecreasing functions,
the function $f$ is an extrinsic \tff\ for $G$
 if and only if $f$ is a \tcf\ for $G$.
\end{corollary}

\begin{proof}
Let $\pp$ be a finite presentation for $G$.
In the case that $f$ is a \tff\ for $(G,\pp)$,
by applying Proposition~\ref{prop:htpydomain}, 
we also have that $(G,\pp)$
has a \cnf\ given 
by a set $\cc$ of geodesic normal forms together with a collection 
$\cle=\{(\dd_e,\tht_e) \mid e \in E_X\}$
of van Kampen diagrams and \ehs,
such that each $\pi_{\dd_e} \circ \tht_e$ is $g$-tame
for a nondecreasing function $g$ that
is Lipschitz equivalent to $f$.
Construct a \dia\ \oc\ of $(X,\xx)$ at $\ep$ as follows.  For 
any point $p \in \xx$, let $u$ be
an edge in $X$ containing $p$.  Then for all
$t \in [0,1]$, let $\up(p,t):=\pi_{\dd_u}(\tht_u(p,t))$.
The gluing condition of the definition of
a \cnf~(p.~\ref{defcnf}) ensures that $\up$ is well-defined.
Moreover $g$-tameness of $\pi_{\dd_u} \circ \tht_u$ then
implies that $\up$ is also $g$-tame.

On the other hand,
if $f$ is a \tcf\ for $(G,\pp)$,
with $f$-tame \dia\ \oc\ $\up$, the definition of
\dia\ implies that there is an associated
\cnf\ through which $\up$ factors.
Again tameness of $\up$ implies that  each of
the \ehs\ of this \cnf\ is $f$-tame, with
respect to the same function $f$, as an immediate
consequence, and Proposition~\ref{prop:htpydomain}
completes the proof. 
\end{proof}

We note  that each of the properties in 
Proposition~\ref{prop:htpydomain} and Corollary~\ref{cor:etistc}
must also have the same quasi-isometry invariance
as the respective \tff, from Theorem~\rfthmitisqi.
Combining this corollary with
Proposition~\ref{prop:itimpliesid}
shows that a \tcf\ is a
strengthening of the concept of,
and an upper bound for, the extrinsic diameter function.

The \tcf\  is 
fundamentally
an extrinsic object,
using (coarse) distances measured in the Cayley
complex;  
Corollary~\ref{cor:etistc}
above shows that 
the logical intrinsic analog of a \tcf\ 
is the concept of an intrinsic \tff.
Another consequence of this corollary
together with the proof of Proposition~\ref{prop:htpydomain}
is that  
%admits an extrinsic \cni\  
%for $f$ if and only if $(G,\pp)$ admits a Lipschitz
%equivalent bound on an extrinsic \cnf\ whose normal form set $\cc$ is
%the set of shortlex normal forms.
%Similarly, 
if $f$ is a \tcf\ for a pair $(G,\pp)$, then  
up to replacing $f$ with a Lipschitz equivalent function
we can restrict
the associated \dia\ 1-combing $\up$ so that the
paths $\up:X\skz  \times [0,1] \ra X$ to
vertices in $X$ follow the shortlex normal forms.

We conclude this section by showing that
a well-defined \tff\ implies the group is tame combable.
A group $G$ is 
{\em tame combable}\label{def:tamecomb}~\cite{mihaliktschantz}
if there is a Cayley complex $X$ for $G$ with respect to
some finite presentation with a 1-combing 
%$\ps:X\sko \times [0,1] \ra X$ 
$\ps$ of $(X,\xx)$ at $\ep$
satisfying the property
that whenever $\tau$ is a 0- or 1-cell in $X$
and $C$ is a compact subset of $X$, then there is a
compact set $D$ in $X$ such that
$\ps^{-1}(C) \cap (\tau \times [0,1])$ is contained in
a single path component of 
$\ps^{-1}(D) \cap (\tau \times [0,1])$.

\begin{corollary}\label{cor:tamecomb}
If $G$ has a well-defined extrinsic
\tff\ over some finite
presentation, then $G$ is tame combable.
\end{corollary}

\begin{proof}
Using Corollary~\ref{cor:etistc}, there is
a \dia\ \oc\ $\up:\xx \times [0,1] \ra X$
of the Cayley complex $X$ for the presentation
that is $f$-tame with respect to a
nondecreasing function $f:\nn \ra \nn$.
Let $C$ be any compact subset of $X$.
Then there is a natural number $N$ such
that for all $p \in C$ the coarse distance
satisfies $\td_X(\ep,p) \le N$.
Let $D$ be the subset of $X$ consisting of
all cells $\sigma$ satisfying the property
that for all vertices $v \in \sigma\skz$,
$\td_X(\ep,p) \le \lceil f(N) \rceil$.
Then $D$ is a finite subcomplex of $X$, and hence
is also compact.  

Let $\tau$ be any 0- or 1-cell in $X$.
For any point $p$ in $\tau$, 
let $s_p \in [0,1]$ be the largest real
number such that $\td_X(\ep,\up(p,t)) \le f(N)$
for all $t \in [0,s_p]$.  Since $N \le f(N)$, we have
$\up^{-1}(C) \cap (\{p\} \times [0,1])
\subseteq \{p\} \times [0,s_p]
\subseteq
\up^{-1}(D) \cap (\{p\} \times [0,1])$.
The path-connected set 
$\{ \{p\} \times [0,s_p] \mid p \in \tau\}$
is contained in a path component
of $\up^{-1}(D) \cap (\tau \times [0,1])$.
\end{proof}

%Another consequence of the proof of Proposition~\ref{cor:etistc}
%is that the pair $(G,\pp)$ admits an extrinsic \cni\  
%for $f$ if and only if $(G,\pp)$ admits a Lipschitz
%equivalent bound on an extrinsic \cnf\ whose normal form set $\cc$ is
%the set of shortlex normal forms.
%Similarly, in any radial \tci, we can restrict
%the \dia\ 1-combing $\up$ so that the
%set of paths $\up:X\skz  \times [0,1] \ra X$ to
%vertices in $X$ follow the shortlex normal forms.

%\coment{See ITF notes p. 101 for a remark on exactly how the
%IT,IRT,IGNF  related, and similarly for extrinsic case, in 
%the proofs above.  Should this be spelled out in the paper?}

%%%%%%%%%%%%%%%%%%%%%%%%%%%%%%%%%%%%%%%%%%%%%%%%%%%%%%%%%%%%%%%%%%%%%%%%%%%%
%%%%%%%%%%%%%%%%%%%%%%%%%%%%%%%%%%%%%%%%%%%%%%%%%%%%%%%%%%%%%%%%%%%%%%%%%%%%
%%%%%%%%%%%%%%%%%%%%%%%%%%%%%%%%%%%%%%%%%%%%%%%%%%%%%%%%%%%%%%%%%%%%%%%%%%%%

\section{Examples of tame filling functions}\label{sec:examples}

%%%%%%%%%%%%%%%%%%%%%%%%%%%%%%%%%%%%%%%%%%%%%%%%%%%%%%%%%%%%%%%%%%%%%%%%%%%%
%%%%%%%%%%%%%%%%%%%%%%%%%%%%%%%%%%%%%%%%%%%%%%%%%%%%%%%%%%%%%%%%%%%%%%%%%%%%

%%%%%%%%%%%%%%%%%%%%%%%%%%%%%%%%%%%%%%%%%%%%%%%%%%%%%%%%%%%%%%%%%%%%%%%%%%%%
%%%%%%%%%%%%%%%%%%%%%%%%%%%%%%%%%%%%%%%%%%%%%%%%%%%%%%%%%%%%%%%%%%%%%%%%%%%%
%%%%%%%%%%%%%%%%%%%%%%%%%%%%%%%%%%%%%%%%%%%%%%%%%%%%%%%%%%%%%%%%%%%%%%%%%%%%

\subsection{Stackable groups}\label{subsec:stkbl}

%%%%%%%%%%%%%%%%%%%%%%%%%%%%%%%%%%%%%%%%%%%%%%%%%%%%%%%%%%%%%%%%%%%%%%%%%%%%
%%%%%%%%%%%%%%%%%%%%%%%%%%%%%%%%%%%%%%%%%%%%%%%%%%%%%%%%%%%%%%%%%%%%%%%%%%%%

$~$

\vspace{.1in}

In this section we give an inductive procedure for
constructing a \cnf, and give
bounds for intrinsic and extrinsic \tffs, for any \fstkbl\ group.
%The definition of \stkbl\ groups requires only
%a finite generating set $A$ for a group $G$;
%following notation conventions in Section~\ref{sec:notation},
%we let $\vec E_\ga$ and $\vec P_\ga$ be the sets 
%of directed edges and paths, respectively, in the 
%Cayley graph $\ga$.

\begin{definition}\label{def:stkbl}~\cite{britherm}
A group $G$ is {\em stackable} over a symmetric generating set $A$
if there is a maximal tree $\ttt$ in the Cayley graph
$\ga=\ga(G,A)$ and a function $\ff:\vece_\ga \ra \vec P_\ga$
such that 
\begin{itemize}
\item[(F1)] For each edge $e \in \vece_\ga$,
the path $\ff(e)$ has the same initial and terminal
vertices as $e$.
\item[(F2d)] If the undirected edge underlying $e$ 
lies in the tree $\ttt$, then $\ff(e)=e$.
%The equality $\ff(e)=e$ holds if and only
%if the undirected edge underlying $e$ lies in the tree $T$.
\item[(F2r)]  The transitive closure 
$<_\ff$ of the relation $<$ on
$\vec E_\ga$, defined by 
\begin{itemize}
\item[]
$e' < e$ whenever $e'$ lies on the path $\ff(e)$
and the undirected edges underlying both
$e$ and $e'$ do not lie in $T$,
\end{itemize}
is a well-founded strict
partial ordering.
\item[(F3)] There is a constant $k$ such that
for all $e \in \vece_\ga$ the path $\ff(e)$ has length at most $k$.
\end{itemize}
Further, $G$ is {\em \astkbl} if the subset
\[
S_\ff:=\{(w,a,\lbl_\ga(\ff(e_{\rep(w),a}))) \mid w \in A^*, a \in A\}
\]
of $A^* \times A \times A^*$ is recursive.
\end{definition}

A function $\ff$ satisfying (F1), (F2d), and (F2r) is
a {\em flow function}, and if $\ff$ also satisfies (F3)
it is {\em bounded}.

In~\cite{britherm}, the authors show that a group $G$
that is \stkbl\ over $A$ 
has a finite presentation $\pp_\ff=\langle A \mid R_\ff\rangle$
where $R_\ff:=\{\lbl(\ff(e))\lbl(e)^{-1} \mid e \in \vece_\ga\}$;
note that finiteness of the set $A$ of edge labels
and (F3) imply that $R_\ff$ is finite.
Let $X$ be the corresponding Cayley complex.
The directed edge $\ega$ from $g$ to $ga$ labeled $a$
will be called {\em degenerate} if (the
undirected edge underlying) $\ega$ lies in the tree
$\ttt$, and {\em recursive} otherwise.  We let $\dgd$ and
$\ves$ denote the sets of degenerate and recursive edges
of $X$, respectively.

In the following proof
that every \fstkbl\ group admits finite-valued
intrinsic and extrinsic \tffs, 
%of Theorem~\ref{thm:stkblhastff},
we modify the construction in~\cite{britherm}
of a collection of {\edg}s for a \stkbl\ group, in order
to give an inductive construction of a \cnf\ for
any \stkbl\ group.
The tameness step in the proof of Theorem~\ref{thm:stkblhastff} hinges
on the fact that a portion of
an \ehy\ path on a subset of $[0,t'] \subset [0,1]$
is also a \oc\ path for any subdiagram containing it;
this is essentially
a notion of ``prefix closure'' for the \ehs.
The recursive nature of the van Kampen diagrams
for \stkbl\ groups is key
to making this prefix closure possible.

\begin{theorem}\label{thm:stkblhastff}
If $G$ is a \fstkbl\ group,  
then $G$ admits well-defined intrinsic and extrinsic \tffs.
\end{theorem}

\begin{proof}
Let $\ttt,\ff$ be a maximal tree and
bounded flow function for the \stkbl\ group
$G$ over $A$, and let $X$ be the Cayley
complex for the finite presentation 
$\pp_\ff=\langle A \mid R_\ff\rangle$.
Define the subset $\cc=\cc_\ttt$ of $A^*$ to be the
set of all words that label a geodesic path
(i.e., with no backtracking) in the tree $\ttt$ starting
at the vertex $\ep$.  (As usual, let $y_g$ be
the label of the path ending at the vertex 
labeled by $g \in G$.)
Then $\cc$ is a prefix-closed
set of normal forms for $G$ over $A$, and all of the
words in $\cc$ label simple paths in $X$.

We proceed in two steps, constructing a \cnf\ for
$G$ over $\pp$, and analyzing the tameness of the
associated \ehs.

\smallskip

\noindent {\em Step 1:  Inductive construction of the \cnf.}

\smallskip

In this step we will construct an \edg\ $\dd_e$ and
\ehy\ $\tht_e$ for every directed edge $e$, and at the end,
for each undirected edge we will choose one pair $(\dd_e,\tht_e)$
to include in the \cnf.
Let $e=\ega$ be a directed edge in $X$, and let 
$w_e:=y_g a y_{ga}^{-1}$.  

{\em Degenerate case.}
Suppose that $e \in \dgd$.  Then the
word $w_e$ freely reduces to the empty
word.  Let $\dd_e$ be the van Kampen diagram for 
$w_e$ consisting of a line segment of edges,
with no 2-cells.
Let $\he=\path_{\dd_e}(y_g,a)$ be the edge of $\bo \dd_e$
corresponding to $a$ in the factorization of $w_e$, and
define the \ehy\  $\tht_e:\he \times [0,1] \ra \dd_e$
by taking $\tht_e(p,\cdot):[0,1] \ra \dd_e$ 
to follow the
shortest length (i.e. geodesic with respect to
the path metric) path from the basepoint $*$ to $p$
at a constant speed,
for each point $p$ in $\he$.
See Figure~\ref{fig:eintree}.
\begin{figure}
\begin{center}
\includegraphics[width=2.3in,height=0.7in]{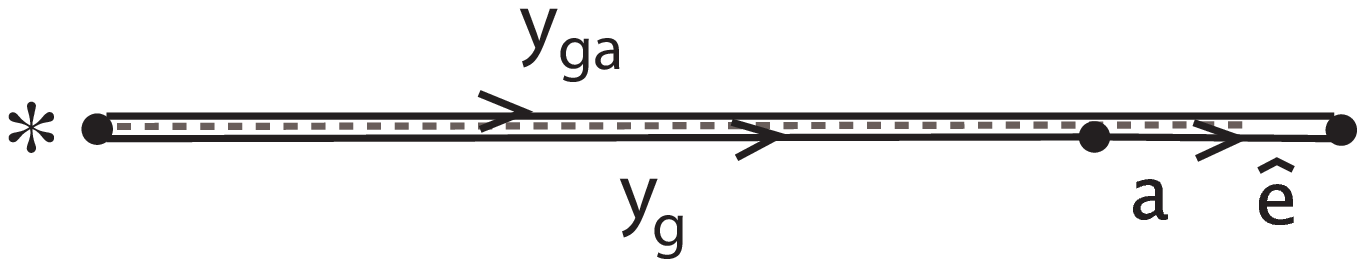}
\hspace{.2in}
\includegraphics[width=2.3in,height=0.7in]{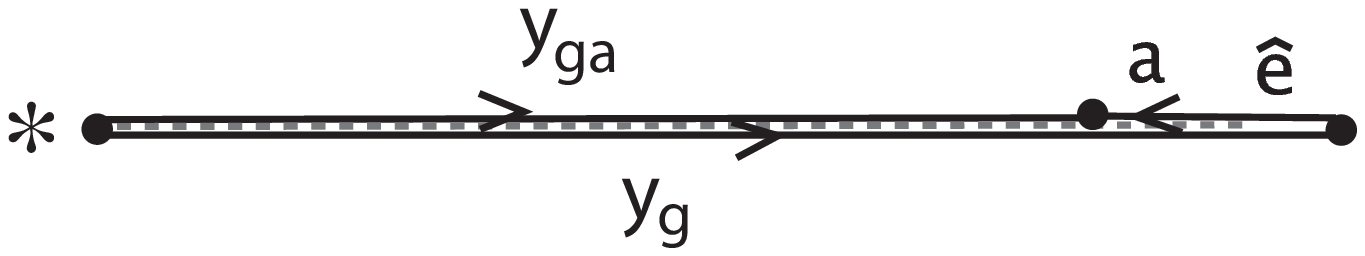}
\caption{$(\dd_e,\tht_e)$ for 
degenerate edge $e$}\label{fig:eintree}
\end{center}
\end{figure}

{\em Recursive case.}
Now suppose that $e \in \ves$.  In this case
will use Noetherian induction
to construct the \edg, using the well-founded
strict partial ordering $<_\ff$ from Definition~\ref{def:stkbl}.
We note that since there are only finitely many
recursive edges in each path $\ff(e'')$, 
it follows that there are at most finitely many
edges $e'$ satisfying $e'<_\ff e$.
Our induction assumption
will be that for all recursive edges $e'<_\ff e$,
we have a \edg\ $\dd_{e'}$ with an \ehy\ $\tht_{e'}$
satisfying the
property that the paths $\tht_{e'}(p,\cdot)$ following
$\bo\dd_{e'}$ from $*$ to each endpoint $p$ of 
$\hat{e'}$
are parametrized to follow these paths at constant speed.

Let $h:=ga$ and
factor the word $\lbl(\ff(e))=x_g^{-1} \tc_e x_h$ such that 
$x_g$ is a suffix of $y_g$ and $x_h$ is a suffix of $y_h$.
Then $y_g=y_q x_g$ and $y_h=y_r x_h$
for some elements $q:=_G gx_g^{-1}$ and $r:=_G hx_h^{-1}$ of $G$,
and the directed edges in the paths $\path_X(g,x_g^{-1})$
and $\path_X(gx_g^{-1}\tc_e,x_h)$ in $X$ are all degenerate.
If the word $\tc_e$ is
nonempty, then write
$\tc_e:=a_1 \cdots a_k$ with each $a_i$ in $A$.
%For each index $i$,
%$1 \le i \le k$,  
Let $e_i$ be the directed edge in $X$ from $qa_1 \cdots a_{i-1}$
to $qa_1 \cdots a_i$ labeled by $a_i$.
Either  $e_i$ is
in $\dgd$, or else $e_i \in \ves$ and 
$e_i <_\ff e$; in both cases we
have, by above or by Noetherian induction,
a van Kampen diagram 
$\dd_i:=
\dd_{e_i}$ with boundary label
$y_{qa_1 \cdots a_{i-1}} a_i y_{qa_1 \cdots a_i}^{-1}$.
By using the ``seashell'' method from the proof of 
Lemma~\ref{lem:edgeimpliesbdry},
we successively 
glue the diagrams  
%$\dd_{e_{i-1}}$, $\dd_{e_i}$
$\dd_{i-1}$, $\dd_i$ 
along their
common boundary words $y_{qa_1 \cdots a_{i-1}}$.
Since all of these gluings are along simple paths,
and the parametrizations of $\tht_{e_{i-1}}$
and $\tht_i:=\tht_{e_i}$ are 
consistent on their common endpoint,
this results in a planar van Kampen diagram $\dd_e'$ with boundary
word $y_q \tc_e y_{r}^{-1}$, and
a \oc\ $\tht_e'$ of the pair $(\dd_e',Z_e')$ based at $*$
where the subcomplex $Z_e'$ is the set of cells
underlying the path $p_e':=\path_{\dd_e'}(y_q,\tc_e)$
in $\bo\dd_e'$ labeled by $\tc_e$.  Moreover, the
\oc\ paths to the endpoints of $p_e'$ again travel along
$\bo\dd_e'$ following the words $y_q,y_r$, respectively,
at constant speed.
If the path $\tc_e$ is empty, then $q=_G r$ and we
define $\dd_e'$ to be a simple edge
path from a basepoint labeled by the word
$y_{q}$ (i.e., the van Kampen diagram for the
word $y_qy_q^{-1}$ with no 2-cells), 
and let $\tht_e'$ be a \oc\ of the pair 
$(\dd_e',\{v_e'\})$ where $v_e'$ is the terminal vertex
$\term(\path_{\dd_e'}(1,y_q))$, such that
$\tht_e'(v_e',\cdot)$ follows this path
at constant speed.

Next we glue
a single (polygonal) 2-cell $\hc_e$ with boundary label 
$x_g^{-1} \tc_e x_h a^{-1}$
onto $\dd_e'$, along the $\tc_e$ subpath in $\bo \dd_e'$,
to produce $\dd_e$.
If the word $\tc_e$ is empty,
then we glue the vertex $v_e'$ to the vertex of $\bo \hc_e$
separating the $x_g^{-1}$ and $x_h$ subpaths.
It follows from this construction that the diagram $\dd_e'$ 
and the
cell $\hc_e$ can be considered to be subsets of $\dd_e$.
Since  we have glued a disk onto $\dd_e'$
along an arc, the diagram $\dd_e$ is again planar, and is a \edg\ 
corresponding to $e$,  with boundary word $w_e$.  
See Figure~\ref{fig:rnfebad}.

Let $\he:=\path_{\dd_e}(y_g,a)$ be the directed edge 
in $\bo \dd_e$ from 
vertex $\hat g$ to vertex $\hat h$ corresponding
to $a$ in the factorization of $w_e$.
Let $\hat q$ and $\hat r$ be the initial and terminal
vertices, respectively, of the
2-cell $\hc_e$ at the start and end, respectively,
of the path in $\bo \hc_e$
labeled by $\tc_e$.  
%These vertices also lie in $\bo \dd_e$.
Let $J:\he \ra [0,1]$ be a homeomorphism, with
$J(\hat g)=0$ and $J(\hat h)=1$.
%, and let
%$\alpha_e$ denote the set of edges in $\bo \hc_e$
%labeled $x_e$.  
Since $\hc_e$ is a disk, there is
a continuous function $\Xi_e:\he \times [0,1] \ra \hc_e$
such that: (i) For each $p$ in the interior $Int(\he)$, we have
$\Xi_e(p,(0,1)) \subseteq Int(\hc_e)$ and $\Xi_e(p,1)=p$.
(ii) $\Xi_e(J^{-1}(\cdot),0):[0,1] \ra \hc_e$ follows the path
in $\bo \hc_e$ labeled $\tc_e$ from $\hat q$ to $\hat r$
at constant speed.
(iii) $\Xi_e(\hat g,\cdot)$ follows the path
in $\bo \hc_e$ labeled $x_g$ from $\hat q$ to $\hat g$
at constant speed.  (iv) $\Xi_e(\hat h,\cdot)$ follows the path
in $\bo \hc_e$ labeled $x_h$ from $\hat r$ to $\hat h$
at constant speed.  
Let $l_g$, $m_g$, $l_h$, and $m_h$ be the lengths 
of the words $y_q$, $x_g$, $y_r$, and $x_h$ in $A^*$,
respectively.
Now define $\tht_e:\he \times [0,1] \ra \dd_e$ by 
%For any point $p$ in $\he$, 
%if $\tc_e$ is a nonempty word, then there is an index
%$j \le i \le k$ such that the point $\Xi_e(p,0)$ lies in $\dd_i$.
%In the case that $\tc_e=1$, let $\tht_i(\hat q,\cdot)$ 
%in the following formula denote the
%constant speed path following the geodesic
%in $\dd_e$ from $*$ to $\hat q=\hat r$.
%Define 
\[ \tht_e(p,t):= \left\{
  \begin{array}{ll}
  \tht_e'(\Xi_e(p,0),\frac{1}{a_p}t)          
        & \mbox{if $t \in [0,a_p]$} \\
  \Xi_e(p,\frac{1}{1-a_p}(t-a_p))          
        & \mbox{if $t \in [a_p,1]$} 
  \end{array} \right.
\]
where
\[ a_p:= \left\{
  \begin{array}{ll}
  \frac{2l_g}{l_g+m_g}(\frac{1}{2}-J(p))+J(p)          
        & \mbox{if $J(p) \in [0,\frac{1}{2}]$} \\
  (1-J(p))+\frac{2l_h}{l_h+m_h}(J(p)-\frac{1}{2})          
        & \mbox{if $J(p) \in [\frac{1}{2},1]$} 
  \end{array} \right.
\]
%See Figure~\ref{fig:rnfebad} for an illustration of
%this \ehy.
Note that if $a_p=0$
and $J(p) \in [0,\frac{1}{2}]$, then we must also
have $J(p)=0$ and $l_g=0$.  In this case $p=\hat g$
and $y_q$ is the empty word, and so 
$\tht_i(\Xi_e(p,0),\cdot)=\tht_i(\hat q,\cdot)$ is
a constant path at the basepoint $*$ of $\dd_e$;
hence $\tht_e$ is well-defined in this case.  
The other instances in which $a_p$
can equal 0 or 1 are similar.
(The complication in this definition of $\tht_e$ stems from
the need to have consistent parametrizations,
in order to satisfy the gluing condition of
the definition of \cnf.)
The map $\tht_e$ is an \ehy\ with the required 
parametrization property.

The van Kampen diagram $\dd_e$ and \ehy\ 
$\tht_e$ are illustrated in Figure~\ref{fig:rnfebad}.
\begin{figure}
\begin{center}
\includegraphics[width=3.8in,height=1.3in]{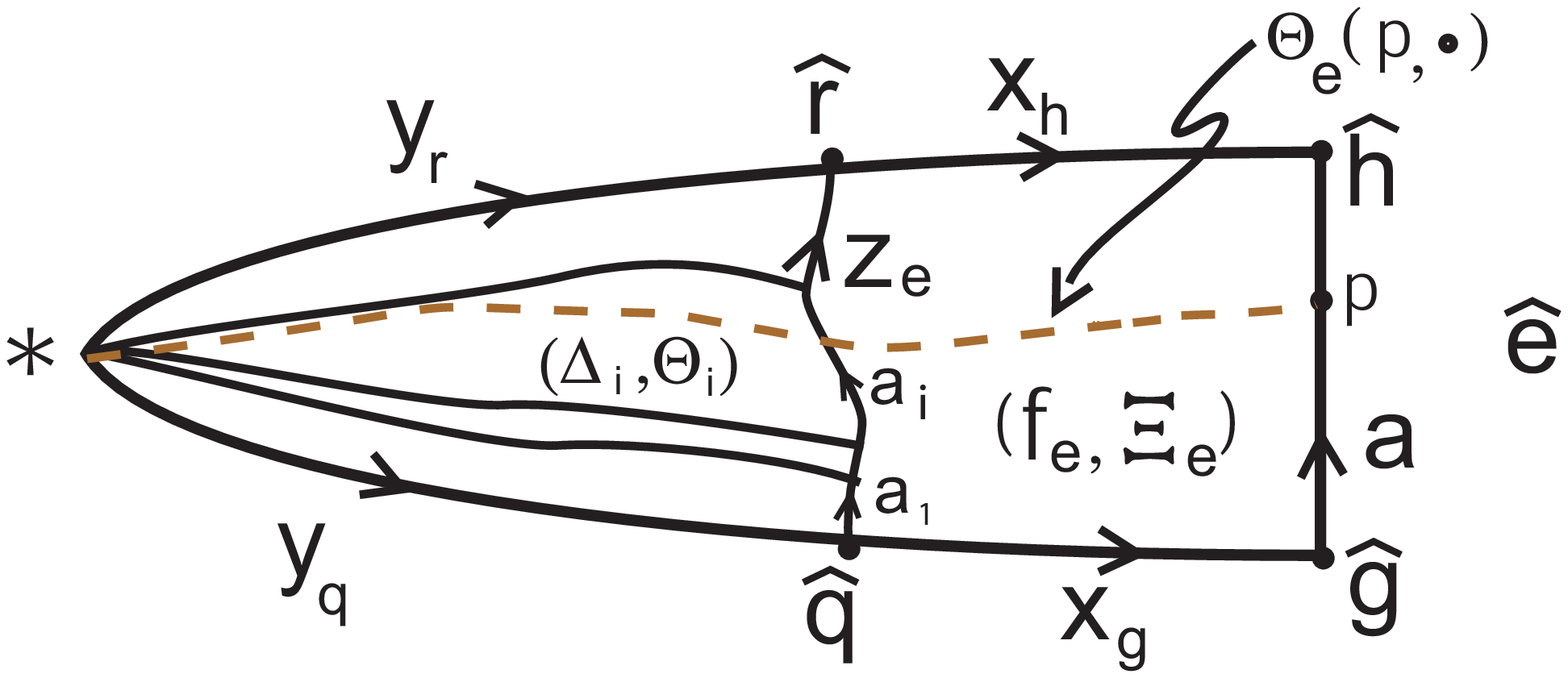}
\caption{$(\dd_e,\tht_e)$ for 
recursive edge $e$}\label{fig:rnfebad}
\end{center}
\end{figure}
We now have a collection 
$\cle:=\{(\dd_e,\tht_e) \mid e \in \vec E_X\}$
of van Kampen diagrams
and \ehs\  for the elements of $\vec E_X$.
To obtain the \cnf\ associated to
the flow function $\ff$, 
the final step again is to eliminate repetitions;
given any undirected edge $e$ in $E_X$, 
let $(\dd_,\tht_e)$ be a \edg\ and \ehy\ constructed
above for one of the orientations of $e$.
Then the collection $\cc$ of prefix-closed normal forms
from the tree $\ttt$,
together with this collection 
$\cle':=\{(\dd_e,\tht_e) \mid e \in E_X\}$
of van Kampen diagrams and \ehs, is a \cnf\ for $G$.
%From Lemma~\ref{lem:edgeimpliesbdry}, this
%\cnf\ induces a \cfl\ using the seashell
%method.

We remark that the {\edg}s of $\cle$ 
satisfy a further property which
we will find useful to refer back to:
 
\begin{itemize}
\item[($\dagger$)]\label{rmk:vkdhasnf}
{\em For every van Kampen diagram $\dd_e$ of $\cle$ and 
every vertex $v$ in $\dd_e$,
there is an edge path in $\dd_e$ from the basepoint $*$ to $v$
labeled by the normal form in $\cc$ for the element
$\pi_{\dd_e}(v)$ in $G$. } 
\end{itemize}

\noindent {\em Step 2:  Analyzing the tameness of the \cnf.}

\smallskip

In this step we analyze the tameness of all of the \ehs\ 
in the set
%%the \cnf\ 
$\cle$ containing the \cnf\ from Step 1.

For each \edg\ $\dd_e$ of $\cle$, 
define the 
{\em intrinsic diameter} 
$\idi(\dd_e) := \max\{d_{\dd_e}(\ep,v) \mid v \in \dd_e\skz \}$ and
{\em extrinsic diameter} 
$\edi(\dd_e) := \max\{d_X(\ep,\pi_{\dd_e}(v)) \mid v \in \dd_e\skz \}$.
%(Note that the van Kampen diagrams in the 
%\cfl\ induced by the \cnf\ $(\cc,\cle)$ may
%not realize the minimal possible intrinsic or 
%extrinsic diameter among all van Kampen diagrams
%for the same boundary words.)
% where
%$X$ is the Cayley complex of the \stkg\ presentation.
Let $B(n)$ be the ball of radius $n$ (with respect to
path metric) in the Cayley
graph $\xx$ centered at $\ep$.
Define the functions 
$\kti, \kte, \kxi, \kxe:\N \ra \N$ by
\begin{eqnarray*}
\kti(n) &:=& \max\{l(y_g) \mid g \in B(n)\}~,\\
\kte(n) &:=& \max\{d_X(\ep,x) \mid \exists~g \in B(n) \text{ such that } 
      x \text{ is a prefix of } y_g \}~,\\
\kxi(n) &:=& \max\{\idi(\dd_e) \mid e \in \ves
              \text{ and the initial vertex } \init(e) \text{ is in } B(n)\}~,\\
\kxe (n) &:=& \max\{\edi(\dd_e) \mid e \in \ves 
              \text{ and the initial vertex } \init(e) \text{ is in } B(n)\}~. 
\end{eqnarray*}
(We do not assume that prefixes are proper.)
%We will need to consider coarse distances throughout
%the Cayley complex $X$.  
%To that end, define the functions 
Finally, define $\mui,\mue:\nn \ra \nn$ by
\begin{eqnarray*}
\mui(n) &:=& \max \{\kti(\lceil n \rceil +  1) + 1, n+1, 
         \kxi(\lceil n \rceil + \maxr+1)+1\} \text{ and} \\
\mue(n) &:=& \max \{\kte(\lceil n \rceil +  1) + 1, n+1, 
         \kxe(\lceil n \rceil + \maxr+1)+1\}~,
\end{eqnarray*}
where $\maxr$ is the length of the longest relator
in the  presentation $\pp_\ff$.  
It follows directly from the definitions that $\kti,\kte,\kxi,\kxe$ 
are well-defined nondecreasing functions, 
and therefore so are $\mui$ and $\mue$. 

In this Step 2 we will show that $\mui$ and $\mue$ are
intrinsic and extrinsic \tffs, respectively, for $G$ 
with respect to $\pp_\ff$.  From Lemma~\ref{lem:edgeimpliesbdry},
it suffices to show that each \ehy\ $\tht_e$ 
from $\cle$ is $\mui$-tame,
and each $\pi_{\dd_e} \circ \tht_e$ is $\mue$-tame.

Let $(\dd_e,\tht_e)$ be any element of $\cle$, 
with $\tht_e:\he \times [0,1] \ra \dd_e$.
Let $p$ be any point in the edge $\he$, and let
$0 \le s<t \le 1$.
To simplify notation later, we also 
let $\sigma:=\tht_e(p,s)$ and $\tau:=\tht_e(p,t)$.
If $\tau$ is in the 1-skeleton $\dd_e\sko$, then
let $\tau':=\tau$ and $t':=t$.  Otherwise, $\tau$
is in the interior of a 2-cell, and there is a
$t \le t' \le 1$ such that $\tht_e(p,[t,t'))$ is
contained in that open 2-cell, and 
$\tau':=\tht_e(p,t') \in \dd_e\sko$; moreover,
from the construction of the {\oc}s in Step 1,
in this case we also have $\tau' \notin \dd_e\skz$.
%Either way, we have 
%$\td_{\dd_w}(*,\tau') < \td_{\dd_w}(*,\tau)+\mxr$, where
%$\mxr$ is the length of the longest relator in the finite
%presentation.

\smallskip

{\em Case I.  $\tau' \in \dd_e\skz $ is a vertex.}
In this case $\tau'=\tau$ and 
the path $\tht_e(p,\cdot):[0,t'] \ra \dd_e$
follows an edge path labeled by the normal
form $y_{\pi_{\dd_e}(\tau')}$ from $*$, 
through $\sigma$, to $\tau$ (at constant speed).  
%In this case $\tau=\tau'$.  
There is a vertex $\sigma'$ on this path lying on the
same edge as $\sigma$ (with $\sigma'=\sigma$ if $\sigma$ is
a vertex) satisfying
$\td_{\dd_e}(*,\sigma) < d_{\dd_e}(*,\sigma')+1$
and $\td_{X}(\ep,\pi_{\dd_e}(\sigma)) < d_{X}(\ep,\pi_{\dd_e}(\sigma'))+1$.
The subpath from $*$ to $\sigma'$ is labeled by a prefix $x$ of 
the word $y_{\pi_{\dd_e}(\tau)}$.
Then 
\[
\td_{\dd_e}(*,\sigma) < d_{\dd_e}(*,\sigma')+1 \le
l(y_{\pi_{\dd_e}(\tau)})+1 \le \kti(d_X(\ep,\pi_{\dd_e}(\tau)))+1
\le \kti(d_{\dd_e}(*,\tau))+1
\]
\[
\text{ and   }~~
 \td_{X}(\ep,\pi_{\dd_e}(\sigma)) 
< d_{X}(\ep,\pi_{\dd_e}(\sigma'))+1 \le
\kte(d_X(\ep,\pi_{\dd_e}(\tau)))+1~.
\]

\smallskip

{\em Case II. $\tau'$ is in the interior of an edge $\te$ of $\dd_e$.} 
%%From the seashell construction, the path $\tht_e(p,\cdot):[0,1] \ra \dd_e$
%%lies in a subdiagram $\dd'$ of $\dd_w$ such that
%%$\dd'$ is a \edg\ in $\cle$.
%Let $\pe:=\pi_{\dd_e}(\te)$ be the image of $\te$ in $X$. 
From the construction of $\cle$ in Step 1,
the subpath $\tht_e(p,\cdot):[0,t'] \ra \dd_e$
lies in a subdiagram $\dd_\pe$ of $\dd_e$ for the
pair  $(\dd_\pe,\tht_\pe) \in \cle$ associated to
a directed edge $\pe \in \vec E_X$ such that
%%$\pi_{\dd_w}(\tau')$ lies in the interior of this edge in $X$.
%%Then $\tau'$ lies in the interior of
%%the corresponding edge $\he$ in $\dd_e$,
%%Moreover 
$\te$ is the edge of $\bo \dd_\pe$ corresponding to $\pe$. 
Then the path $\tht_e(p,\cdot):[0,t'] \ra \dd_e$ is
a bijective (orientation preserving)
reparametrization of the path $\tht_\pe(\tau',\cdot):[0,1] \ra \dd_\pe$.

\smallskip

{\em Case IIA.  $\pe$ is degenerate.} 
The van Kampen diagram $\dd_\pe$ contains no 2-cells, and
the path $\tht_\pe(\tau',\cdot):[0,1] \ra \dd_\pe$
follows the edge path labeled by a normal form
$y_g \in \cc$ from $*$ to 
$\hat g$ (at constant speed), 
and then follows the portion of $\te$ from $\hat g$ to $\tau'$,
where $\hat g$ is the endpoint of $\te$
closest to $*$ in the diagram $\dd_\pe$.
In this case, $\tau$ must also lie in $\dd_e\sko$, and
so again we have $\tau=\tau'$.
Since $\hat g$ and $\tau$ lie in the same closed 1-cell,
we have $d_{\dd_e}(*,\hat g)< \lceil \td_{\dd_e}(*,\tau) \rceil +1$,
and similarly for their images (via $\pi_{\dd_e}$) lying in the same
closed edge of $X$.

If $\sigma$ lies in the $y_{g}$ path, then 
Case I applies to that path, with $\tau$ replaced by the 
vertex $\hat g$.  Combining
this with the inequality above and applying the nondecreasing
property of the functions $\kti$ and $\kte$ yields
\[
\td_{\dd_e}(*,\sigma) <  \kti(d_{\dd_e}(*,\hat g))+1
\le \kti(\lceil \td_{\dd_e}(*,\tau) \rceil+1)+1
\text{ and }
\]
\[
 \td_{X}(\ep,\pi_{\dd_e}(\sigma)) < \kte(d_X(\ep,\pi_{\dd_e}(\hat g)))+1
\le \kte(\lceil \td_X(\ep,\pi_{\dd_e}(\tau)) \rceil +1)+1~.
\]

On the other hand, if $\sigma$ lies in $\te$, 
then $\sigma$ and $\tau$ are contained in a common edge.
Hence 
\[
\td_{\dd_e}(*,\sigma) <  \td_{\dd_e}(*,\tau) +1
\text{  and  } \td_{X}(\ep,\pi_{\dd_e}(\sigma)) <  
\td_X(\ep,\pi_{\dd_e}(\tau)) +1~.
\]

\smallskip

{\em Case IIB. $\pe$ is recursive.}  
%We prove this case by Noetherian
%induction using the well-founded ordering $<$.
In this case either $\tau=\tau'$, or $\tau$ is
in the interior of the cell $\hc_\pe$ of the
diagram $\dd_\pe$ from the construction in Step 1.  
%Let $\hat g$ be the vertex of $\he$ corresponding
%to the initial vertex $g=\pi_{\dd_w}(\hat g)$
%of the directed edge $e$.
Let $g$ be the initial vertex 
of the directed edge $\pe \in \vec E_X$.
Then $g$ and $\pi_{\dd_e}(\tau)$ lie in a common
edge or 2-cell of $X$, and so
%$d_{X}(*,\hat g) < \lceil \td_{X}(*,\pi_{\dd_w}(\tau)) \rceil +\maxr+1$ and 
$d_{X}(\ep,g) < \lceil \td_{X}(\ep,\pi_{\dd_e}(\tau)) \rceil +\maxr+1$.
%where $\maxr$ is the length of the longest relator in 
%the presentation of $G$.

Note that distances in the  subdiagram $\dd_\pe$
are bounded below by distances in $\dd_e$.
Also, the map $\pi_{\dd_e}$ cannot increase distances.
In this case, combining these inequalities and the
nondecreasing properties of $\kxi$ and $\kxe$ yields
\begin{eqnarray*}
\td_{\dd_e}(*,\sigma) &\le&  \td_{\dd_\pe}(*,\sigma)
\le \idi(\dd_\pe)+1 \le \kxi(d_{X}(\ep,g))+1 \\
&\le&  \kxi(\lceil d_{X}(\ep,\pi_{\dd_e}(\tau)) \rceil +\maxr+1)+1
\le \kxi(\lceil d_{\dd_e}(*,\tau)) \rceil +\maxr+1)+1
\text{ and } \\
 \td_{X}(\ep,\pi_{\dd_e}(\sigma)) 
&=&\td_{X}(\ep,\pi_{\dd_\pe}(\sigma)) 
\le \edi(\dd_\pe)+1 \le \kxe(d_{X}(\ep,g))+1 \\
&\le& \kxe(\lceil \td_{X}(*,\pi_{\dd_e}(\tau)) \rceil +\maxr+1)+1~.
\end{eqnarray*}

%The point $\sigma=\tht_e(\tau',s')$
%for some $s' \in [0,1]$.
%From the construction above of the \ehy\  $\tht_e$, 
%recall that on the interval
%$[0,a_p]$ the path
%$\tht_e(\tau',\cdot)$ follows the path
%$\tht_i(q,\cdot):[0,1] \ra dd_e$ from $*$ to
%$q:=\Xi_e(\tau',0)$, and on $[a_p,1]$ traverses the
%cell $\hc_e$ from $q$ to $\tau'$.
%Suppose first that $\sigma$ lies on the $\tht_i$ path.  If $\tx_e$ is
%the empty word, then Case I applies, and if $\tx_e$ is not empty,
%then either Case II or induction on $<$ applies, 

Therefore in all cases, we have
$\td_{\dd_e}(*,\sigma) \le \mui(\td_{\dd_e}(*,\tau))$
and $\td_{X}(\ep,\pi_{\dd_e}(\sigma)) \le \mue(\td_{X}(\ep,\pi_{\dd_e}(\tau)))$,
as required.
\end{proof}

The \tff\ 
bounds in Theorem~\ref{thm:stkblhastff} are not sharp in general.
In particular, we will improve upon these bounds for the example of
almost convex groups in Section~\ref{subsec:linear}.

\begin{definition}\label{def:rcnf}
A {\em \rcnf} is a \cnf\ that can be
constructed from a bounded flow function  by the
procedure in Step 1 of the proof of Theorem~\ref{thm:stkblhastff}. 
A {\em \rcf} is a \cfl\ induced
by a \rcnf\ using the seashell procedure.
\end{definition}

%In our analysis of intrinsic tameness, we 
%will make use of the function $j:\N \ra \N$ defined by
%$$
%j(n) := \max\{l(y_g) \mid g \in G \text{ with } d_X(\ep,g) \le n\}~,
%$$
%where $l$ denotes word length in $A^*$, and
%the function $\tj:\nn \ra \nn$ defined by
%$$
%\tj(n) := j(\lceil n \rceil)+1~.
%$$

The procedure 
described in Step 1 of the proof of
Theorem~\ref{thm:stkblhastff} for building 
van Kampen diagrams for a stackable group
using the bounded flow function
may not be an algorithm.  However, in~\cite{britherm},
we show that for \afstkbl\ groups this process is
algorithmic, and hence the word problem is
solvable for \afstkbl\ groups.
%Applying this, we obtain a computable bound on \tffs\ for
%\afstkbl\  groups.  

\begin{theorem}\label{thm:astktf}
If $G$ is an \afstkbl\ group, then $G$ has recursive
intrinsic and extrinsic \tffs.
\end{theorem}

\begin{proof}
Note that
whenever a function $f:\nn \ra \nn$ is a \tff\ for
a group $G$, and $g:\nn \ra \nn$
satisfies the property that $f(n) \le g(n)$
for all $n \in \nn$, then $g$ is also a \tff\ for $G$
(with respect to the same intrinsic or extrinsic
property).
From Step 2 of the proof of 
Theorem~\ref{thm:stkblhastff}, it suffices to show that the functions
$\kti$, $\kte$, $\kxi$, and $\kxe$ are 
bounded above by recursive functions.
Moreover, since distances in the Cayley complex $X$ are
bounded above by distances in van Kampen diagrams,
we have $\kte(n) \le \max\{l(x) \mid x \text{ is a prefix of }
y_g$ for some $g \in B(n)\} = \kti(n)$ and $\kxe(n) \le \kxi(n)$
for all $n$, so it suffices to find recursive upper bounds
for $\kti$ and $\kxi$.
%Since distance in a van Kampen diagram $\dd$ always
%gives an upper bound for distance, via the map $\pi_\dd$,
%in the Cayley complex $X$,
%then for all $n \in \N$,
%we have $\kte(n) \le \kti(n)$.  
%Moreover, $\idi(\dd)$ must always be an
%upper bound for $\edi(\dd)$, and so
%$\kxe(n) \le \kxi(n)$ for all $n$.  Thus it
%suffices to find recursive upper bounds for $\kti$ and $\kxi$.

The set $\cc=\cc_\ttt$ of normal forms can be
written 
\[ \cc=\{a_1 \cdots a_n \mid \forall~i,~a_i \in A
\text{ and }
(a_1 \cdots a_{i-1},a_i,a_i) \in S_\ff\},
\]
and hence $\cc$ is recursive.
Given a word $w \in A^*$, by enumerating the set
of words $z \in \cc$ and applying the word problem
solution to determine whether or not $zw^{-1}=_G \ep$,
we can compute the normal form $y_w$ of $w$.
By enumerating the finite set of words over $A$ of length
up to $n$,
computing their normal forms in $\cc$
% with the reduction algorithm, 
and taking the maximum word length that
occurs, we obtain $\kti(n)$.  Hence the function
$\kti$ is computable.

Given $w \in A^*$ and $a \in A$, we compute 
the two words $y_w$ and $y_{wa}$ and store them
in a set $L_e$.
Next we follow the construction of 
the \edg\ $\dd_e$
%and \ehy\ $(\dd_e,\tht_e)$
associated to the edge $e=e_{w,a} \in \vec E_X$
from Step 1 of the proof of Theorem~\ref{thm:stkblhastff}.
%If the word $y_w a y_{wa}^{-1}$ freely reduces to the empty word, 
If $(w,a,a) \in \alg$,
then $e \in \dgd$ and we add no
other words to $L_e$.  On the other hand, if 
%$y_w a y_{wa}^{-1}$ does not freely reduce to 1,
$(w,a,a) \notin \alg$,
then $e \in \ves$.  In the latter case,
by enumerating the finitely many words $x$ of length up to $k$,
where $k$ is the bound on the flow function,
and checking whether or not $(w,a,x)$ lies in the computable
set $\alg$, 
we can determine the word $\lbl(\ff(e))=x$.
Write $x=a_1 \cdots a_n$ with each $a_i \in A$.
For $1 \le i \le n$, we compute the normal forms $y_i$ in $\cc$
for the words $wa_1 \cdots a_i$, and add these words to
the set $L_e$.  For each pair $(y_{i-1},a_i)$, we determine 
the word $x_i$ such that
$(y_i,a_i,x_i) \in \alg$. 
If $x_i \neq a_i$, we write $x_i=b_1 \cdots b_m$ 
with each $b_j \in A$, and add the normal forms
for the words $y_{i-1}b_1 \cdots b_j$ to $L_e$ for each $j$.
Repeating this process through all of the
steps in the construction of $\dd_e$, we must, after finitely many
steps, have no more words to add to $L_e$. 
The set $L_e$ now contains the normal form
$y_{\pi_{\dd_e}(v)}$ for each vertex $v$ of the
diagram $\dd_e$.
Then $k(w,a):=\max\{l(y) \mid y \in L_e\}$ is computable.

Now 
property~($\dagger$) says that for 
each vertex $v$ of the \edg\ $\dd_e$ there is
a path in $\dd_e$ from the basepoint to $v$
labeled by a word in the set $L_e$.
Then $\idi(\dd_e) \le k(w,a)$.
%, and we have
%an algorithm to compute $k(w,a)$. 
%
%
%Now we can write 
That is, $\kxi(n) \le k_r'(n)$ for all $n \in \N$, where
\begin{eqnarray*}
k_r'(n) &:=& \max \{k(w,a) \mid w \in \cup_{i=0}^n A^i, a \in A
%, e_{w,a} \in \ves
\}.
\end{eqnarray*}
Repeating the computation of  $k(w,a)$
above for all words $w$ of length at most $n$
and all $a \in A$, we can compute
this upper bound $k_r'$ for $\kxi$, as required.
\end{proof}

\begin{remark}
{\em In~\cite{britherm}, we show that the fundamental group
of every three manifold with uniform geometry is \astkbl.
Theorem~\ref{thm:astktf} then gives bounds on the filling
functions for these groups, and Corollary~\ref{cor:tamecomb}
gives another proof that these groups are tame combable.}
\end{remark}

%\begin{remark}\label{rmk:quasiastk}
%{\em Although the proof of Theorem~\ref{thm:astktf}
%shows in the abstract that an algorithm must exist
%to compute $k_r'(n)$, this proof does not give a
%method to find this algorithm starting from the
%computable set $\alg$.  In particular, 
%give $\alg$ there may not be an algorithm to find
%the bound $k$ on $\ff$.
%although
%every finite set is recursively enumerable,
%it is not clear how
%to enumerate the finite set $\hr$.  In practice, however,
%for every example we will discuss, we start with both
%a finite presentation $\langle A \mid R \rangle$
%for the group $G$ and a \stkg\ that
%(re)produces that presentation.  In that case, the set
%$\hr$ must be contained in the finite set 
%$R':=\{x \in A^* \mid \exists a \in A$ with $xa \in R\}$.  
%Then we can replace
%the enumeration of $\hr$ with an enumeration of $R'$, 
%which can be computed from $R$.}
%\end{remark}

%\eject

%%%%%%%%%%%%%%%%%%%%%%%%%%%%%%%%%%%%%%%%%%%%%%%%%%%%%%%%%%%%%%%%%%%%%%%%%%%%
%%%%%%%%%%%%%%%%%%%%%%%%%%%%%%%%%%%%%%%%%%%%%%%%%%%%%%%%%%%%%%%%%%%%%%%%%%%%
%%%%%%%%%%%%%%%%%%%%%%%%%%%%%%%%%%%%%%%%%%%%%%%%%%%%%%%%%%%%%%%%%%%%%%%%%%%%

\subsection{Groups admitting complete
     rewriting systems}\label{subsec:rs}

%%%%%%%%%%%%%%%%%%%%%%%%%%%%%%%%%%%%%%%%%%%%%%%%%%%%%%%%%%%%%%%%%%%%%%%%%%%%
%%%%%%%%%%%%%%%%%%%%%%%%%%%%%%%%%%%%%%%%%%%%%%%%%%%%%%%%%%%%%%%%%%%%%%%%%%%%

$~$

\vspace{.1in}

In this section we discuss implications of Theorem~\ref{thm:astktf}
for a special class of
\stkbl\ groups, namely
groups that admit a finite complete rewriting system. 
A {\em finite complete rewriting system}
(finite {\em CRS}) for a group $G$ consists of a finite set $A$
and a finite set of rules $R \subseteq A^* \times A^*$
%with $u,v \in A^*$ 
(with each $(u,v) \in R$ written $u \ra v$)
such that 
as a monoid, $G$ is presented by 
$G = Mon\langle A \mid u=v$ whenever $u \ra v \in R \rangle,$
and the rewritings
$xuy \ra xvy$ for all $x,y \in A^*$ and $u \ra v$ in $R$ satisfy:
(1) Each $g \in G$ is 
represented by exactly one {\em irreducible} word $y_g$
(i.e. word that cannot be rewritten)
over $A$, and
(2) the (strict) partial ordering
$x>y$ if $x \rightarrow
x_1 \rightarrow ... \rightarrow x_n \rightarrow y$ is 
well-founded.

Property (1) says that the set $\cc$ of irreducible
words is a set of normal forms for $G$ over $A$.
Given a finite CRS for $G$, there is a {\em minimal}
finite CRS for $G$ with the same set of irreducible
normal forms; that is, for every rule $u \ra v$,
the word $v$ and every proper subword of $u$
are irreducible.  Every finite CRS in this paper
will be assume to be minimal.

Since any prefix of an irreducible word is irreducible,
the set $\cc$ is prefix closed.  Let $\ttt$ be
the set of edges of the Cayley complex $X$
(of the presentation from the CRS) that lie
on paths $\path_X(\ep,y_g)$ for all $y_g \in \cc$.
Prefix closure of $\cc$ implies that $\ttt$ is a 
maximal~tree~in~$X$.

Define the ``rewriting flow function'' 
$\ff:\vec E_X \ra \vec P_X$ 
as follows.  Let  $\ega$ be any edge in $\vece_X$.
If either word $y_ga$ or $y_{ga}a^{-1}$
is irreducible, then let
$\ff(\ega):=\ega$.  On the other hand, if
neither word is irreducible, then there is a unique
decomposition $y_ga=y'ua$ for some $y',u \in A^*$
and $ua \ra v$ in $R$.  In that case, define
$\ff(\ega):=\path_X(g,u^{-1}v)$. 

In~\cite{britherm}, we show that 
$\ff$ is a bounded flow function and the set
$\alg$ is computable, and so the
group $G$ is \astkbl.
Theorem~\ref{thm:astktf} shows that any group
with a finite CRS
%complete rewriting system 
admits recursive
intrinsic and extrinsic \tffs.
In Proposition~\ref{prop:rsgrowth}
 we relax the bounds on 
the \tffs\  further, in order to write 
bounds on 
%both \tfs\ and diametric 
filling functions in terms
of another important function in the study of rewriting systems.

Given any word $w \in A^*$, we write
$w \ras w'$ if there is any sequence of rewritings
$w=w_0 \ra w_1 \ra \cdots \ra w_n=w'$ (including
the possibility that $n=0$ and $w'=w$).
%\begin{definition}
The {\em string growth complexity} function $\gamma:\N \ra \N$
associated to a finite complete rewriting system $(A,R)$
is defined by 
\[ 
\gamma(n) := \max \{l(x) \mid \exists~w \in A^* \text{ with }
l(w) \le n \text { and } w \ras x\}. 
\]
%\end{definition}
This function $\gamma$ is an upper bound for the
intrinsic (and hence also extrinsic) diameter function
of the group $G$ presented by the rewriting system.  
In the following, we show that 
$G$ admits \tffs\ that are
Lipschitz equivalent to $\gamma$.

\begin{proposition}\label{prop:rsgrowth}
Let $G$ be a group with a finite complete
rewriting system. Let
$\gamma$ be the
string growth complexity function and let $\maxr$ denote the length of 
the longest rewriting rule. Then the
function $n \mapsto \gamma(\lceil n \rceil +\maxr+2)+1$
is both an intrinsic and an extrinsic \tff\ for $G$.
\end{proposition}

\begin{proof}
We use the notation and results of the proofs
of Theorems~\ref{thm:stkblhastff} and~\ref{thm:astktf}
throughout this proof.
Let $(A,R)$ be a minimal finite complete rewriting system
for $G$, with associated Cayley complex $X$, and 
let $\ff$ be the rewriting flow function on
the tree from the set $\cc$ of normal forms of this system.
% such that $A$ is inverse-closed.
%Let $\maxr$ be the length of the longest rule; that is,
%$\maxr=\max\{l(u)+l(v) \mid u \ra v \text{ is in }R$
%
%The {\em length} of a rewriting rule $u \ra v$ in $R$
%is the sum of the lengths of the words $u$ and $v$.
%, and let
%$\pp$ be the symmetric presentation of $G$ associated to
%this system, with 
%Cayley complex $X$.
%Let $(\cc,c)$ be the \stkg\ for $G$ constructed
%in the proof of Theorem~\ref{thm:crsrecit}, and
%let $X$ be the Cayley complex of the \stkg\ presentation $\pp$.
%
%the (minimal) finite complete rewriting system.
%Let  for the presentation
%from this system.
For any recursive edge $e:=\ega$, the word $y_ga$
is reducible, with rewriting $y_ga=y'ua \ra y'v$ for 
a rule $ua \ra v$ in $R$.
Following the notation of Step 1 of
Theorem~\ref{thm:stkblhastff}, we factor
$\lbl(\ff(e))=u^{-1}v=x_g^{-1} \tc_e x_{ga}$ where
$x_g=u$, $\tc_e=v$, and $x_{ga}=1$, since
$u$ is a suffix of the normal form $y_g$.
Let $\cle=\{\dd_e,\tht_e) \mid e \in E_X\}$
be the associated \rcnf\ from Step 1.
%Let $\clf=\{(\dd_w,\ps_w)\}$ be the \cfl\ 
%over $\pp$ induced by this \rcnf.
%For the rest of this proof, we rely heavily
%on the result and
%notation developed in 
%to obtain the tameness bounds for these \ehs.

From the proof of Theorem~\ref{thm:astktf}, we have 
$\kte(n) \le \kti(n) = \max\{l(y_g) \mid g \in B(n)\}=
\max\{l(y_w) \mid w \in A^*,~l(w) \le n\}$ for all $n$.
%, where $\cc_n$ is the set of irreducible normal forms obtained by
%rewriting words over $A$ of length at most $n$.  
Since each normal form $y_w$ is obtained
from $w$ by rewritings, $\kte(n) \le \kti(n) \le \gamma(n)$.

Also from Theorem~\ref{thm:astktf}, we have $\kxe(n) \le \kxi(n) \le k_r'(n)$
for all $n \in \N$.  
Suppose that $w \in A^*$ is a word of length at most $n$, $a \in A$,
and $e=e_{w,a}$ is the edge in $X$ from $w$ to $wa$
labeled by $a$.
%We now analyze the van Kampen diagram $\dd_e$ more carefully.
If $e$ is degenerate, then all of the vertices
of $\dd_e$ lie along a path from $*$ labeled $y_w$ or $y_{wa}$ in $\dd_e$.
On the other hand, suppose $e$ is recursive, and 
let $\hat v$ be a vertex on the boundary of the 2-cell $f_e$
of $\bo \dd_e$ from Step 1 of 
%the proof of 
Theorem~\ref{thm:stkblhastff}.
Writing (as above) 
$y_wa=y'ua \ra y'v$ where $ua \ra v \in R$, then 
$\hat v$ lies along a path in $\dd_e$
starting at the vertex $\hat g=\term(\path_{\dd_e}(1,y_w))$
and labeled by $\lbl(\ff(e))=u^{-1}v$.
Hence there is a path in $\dd_e$ from
$*$ to $\hat v$ 
that is labeled by a prefix of $y_w$ or $y'v$; 
in either case, the label is a subword of a word $x$ satisfying
$y_wa \ra x$. 
%(and since $w \ras y_w$, then $w \ras x$ as well).  
The diagram $\dd_e$ is built 
by iterating the flow function,  successively
applying rewritings to the word $y_w a$
and/or by applying free reductions (which must also result
from rewritings), and so 
%the property in the
%previous sentence holds for all vertices in $\dd_e$.
%Applying remark $\dagger$ (see p.~\pageref{rmk:vkdhasnf}),  
we have that for every vertex $\tilde v$ in $\dd_e$,
there is a path from the basepoint $*$ to $\tilde v$ labeled by 
a word $s \in A^*$
%, and
%$y$ is a prefix of a word $x \in A^*$ satisfying
%$y_w a \prs x$.  
such that $y_wa \ras st$ for some $t \in A^*$.
Moreover,
the normal form $y_s$ of $s$
is the element of the set $L_e$ 
(defined in the proof of Theorem~\ref{thm:astktf})
corresponding to the vertex $\tilde v$.
Note that since $s \ras y_s$, then $y_wa \ras y_st$ as well.
That is, $L_e \subseteq L_e'$
where 
\[
L_e':=\{y \in \cc \mid \exists~x \in A^* \text{ such that }
y \text{ is a 
prefix of } x \text{ and } y_w a \ras x\}.
\]

Then the maximum $k(w,a)$ of the lengths of
the elements of $L_e$ is bounded above
by $\max \{l(y) \mid y \in L_e'\}$.
%by $\max \{l(y) \mid y$ is a prefix of $x$ and $y_w a \ras x\}$.
Since the length of a prefix of a word $x$ is at most
$l(x)$, we have
$k(w,a) \le \max \{l(x) \mid y_w a \ras x\}$.
Using the fact that $w \ras y_w$, we also obtain
$k(w,a) \le \max \{l(x) \mid w a \ras x\} \le \gamma(l(w)+1)$.
Plugging this into the formula for $k_r'$, we obtain
$\kxi(n) \le k_r'(n) = \max\{k(w,a) \mid w \in A^*,~a \in A,~l(w) \le n\}
\le \gamma(n+1)$. 

Putting these inequalities together, we obtain $\mue(n) \le \mui(n)$ and
%\begin{eqnarray*}
%\mui(n) &= &\max \{\kti(\lceil n \rceil +  1) + 1, 
%         n+1, \kxi(\lceil n \rceil + \maxr+1)\} \\
%  &\le& \gamma(\lceil n \rceil +\maxr+2)+1.
%\end{eqnarray*}

\hspace{.2in} $\mui(n) = \max \{\kti(\lceil n \rceil +  1) + 1, 
         n+1, \kxi(\lceil n \rceil + \maxr+1)\} 
%\hspace{.3in} 
 \le \gamma(\lceil n \rceil +\maxr+2)+1.$
\end{proof}

%\begin{remark}  {\em 
We remark that every instance of rewritings $\ras$ in the proof
of Proposition~\ref{prop:rsgrowth}
can be replaced by prefix rewritings; that is, 
a sequence of rewritings $w=w_0 \ra \cdots \ra w_n=w'$,
written $w \prs w'$,
such that at each $w_i$,
the shortest possible reducible prefix is rewritten to obtain
$w_{i+1}$.  
Hence $G$ also admits \tffs\ given by the potentially
smaller {\em prefix rewriting string growth complexity}
function $\gamma_p(n):= \max \{l(x) \mid \exists~w \in A^* \text{ with }
l(w) \le n \text { and } w \prs x\} \le \gamma(n)$.
%}  \end{remark}

Next we use rewriting systems to discuss
the breadth of the range of \tffs\ for groups.
The iterated
Baumslag-Solitar group
$$G_k=\langle a_0,a_1,...,a_k \mid a_i^{a_{i+1}}=a_i^2; 0 \le i \le k-1\rangle$$
admits a finite complete rewriting
system for each $k \ge 1$ (first described by Gersten;
see \cite{hmeiermeastame} for details).
Hence $G_k$ is also \astkbl, 
with the recursive \tffs\ 
%(bounded) rewriting flow function~\cite{britherm}
described above.

Gersten~\cite[Section~6]{gerstenexpid} showed that $G_k$
has an isoperimetric function that grows at least as fast as a tower 
of exponentials
$$
E_k(n) := \underbrace{2^{2^{.^{.^{.^{2^n}}}}}}_{\hbox{k times}}~.
$$
It follows from his proof that the (minimal) extrinsic diameter function 
for this group is at least $O(E_{k-1}(n))$.
%~\cite{hmeiermeastame}.  
Hence this is also a lower bound for the (minimal) intrinsic diameter
function for this group.  
Then, by Proposition~\ref{prop:itimpliesid}, 
the function $E_{k-2}$ cannot be an intrinsic or extrinsic \tff\ 
for $G_k$.  In the extrinsic case, this was shown in the
context of tame combings in~\cite{hmeiermeastame}.
Combining this with Proposition~\ref{prop:rsgrowth}
yields the following.

\begin{corollary}
For each $k \ge 2$, the group $G_k$ 
%is an example of a (\astkbl) group  which 
admits recursive intrinsic
and extrinsic \tffs, 
but all \tffs\ for $G$ must grow faster than $E_{k-2}$.
\end{corollary}
  
%Moreover, for every minimal finite complete
%rewriting system for the group $G_k$ over any generating set,
%the associated string growth growth
%complexity function  must also grow faster than any function
%Lipschitz equivalent to $E_{k-2}$.

%\coment{Want to add to this example that the Gersten
%group $G_k$ has even better yet an iterated exponential \tfi.  Need to write up
%the applications and see how to do this first.}

%\vspace{.1in}

%%%%%%%%%%%%%%%%%%%%%%%%%%%%%%%%%%%%%%%%%%%%%%%%%%%%%%%%%%%%%%%%%%%%%%%%%%%%
%%%%%%%%%%%%%%%%%%%%%%%%%%%%%%%%%%%%%%%%%%%%%%%%%%%%%%%%%%%%%%%%%%%%%%%%%%%%
%%%%%%%%%%%%%%%%%%%%%%%%%%%%%%%%%%%%%%%%%%%%%%%%%%%%%%%%%%%%%%%%%%%%%%%%%%%%

\subsection{Finite groups}\label{subsec:finite}

%%%%%%%%%%%%%%%%%%%%%%%%%%%%%%%%%%%%%%%%%%%%%%%%%%%%%%%%%%%%%%%%%%%%%%%%%%%%
%%%%%%%%%%%%%%%%%%%%%%%%%%%%%%%%%%%%%%%%%%%%%%%%%%%%%%%%%%%%%%%%%%%%%%%%%%%%

$~$

\vspace{.1in}

Suppose that $G$ is a finite group with 
finite symmetric presentation $\pp = \langle A \mid R \rangle$.
Let ${\mathcal F}$ be a finite collection 
of van Kampen diagrams over $\pp$, one for each
word over $A$ of length at most $|G|$ that
represents the identity $\ep$ of $G$.

Given any word $u$ over $A$ with $u=_G \ep$,
we will construct a van Kampen diagram for $u$
with intrinsic diameter bounded above by the constant
$|G|+\max\{idiam(\dd) \mid \dd \in {\mathcal F}\}$,
as follows. 
Start with a planar 1-complex that
is a line segment consisting of an edge path
labeled by the word $u$ starting at a basepoint $*$; 
that is, we start with a van Kampen diagram for
the word $uu^{-1}$.
Write $u=u_1'u_1''u_1'''$ where $u_1''=_G \ep$
and no proper prefix of $u_1'u_1''$ contains a subword 
that represents $\ep$.
Note that $l(u_1'u_1'') \le |G|$.
We identify the vertices in the
van Kampen diagram at the start and end of the
boundary path labeled $u_1''$,
and fill in this loop with the van Kampen
diagram from ${\mathcal F}$ for this word.
We now have a van Kampen diagram for the
word $uu_1^{-1}$ where 
$u_1:=u_1'u_1'''$.
We then begin again, and write $u_1=u_2'u_2''u_2'''$
where $u_2''=_G \ep$ and no proper prefix of
$u_2'u_2''$ contains a subword representing the
identity.  Again we identify the vertices at
the start and end of the word $u_2''$ in the
boundary of the diagram, and fill in this
loop with the diagram from ${\mathcal F}$ for
this word, to obtain a van Kampen diagram 
for the word $uu_2^{-1}$ where $u_2:=u_2'u_2'''$.
Repeating this process, since at each step
the length of $u_i$ strictly decreases, we
eventually obtain a word $u_k=u_k''$.  Identifying
the endpoints of this word and filling in
the resulting loop with the van Kampen diagram
in ${\mathcal F}$ yields a van Kampen diagram $\dd_u$
for $u$.  
See Figure~\ref{fig:finitegp} for an illustration of
this procedure.
\begin{figure}
\begin{center}
\includegraphics[width=3.4in,height=0.7in]{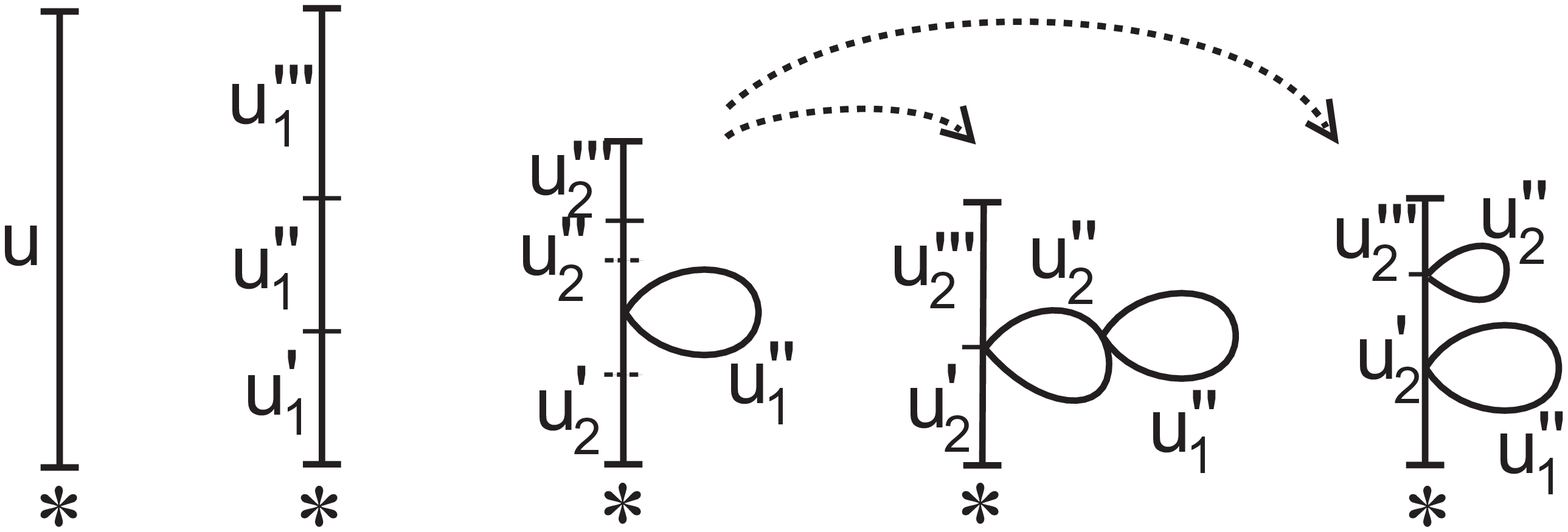}
\caption{Building $\dd_u$ in the finite group case}\label{fig:finitegp}
\end{center}
\end{figure}
At each step, the maximum distance
from the basepoint $*$ to any vertex in
a van Kampen diagram included from ${\mathcal F}$
is at most $|G|+\max\{idiam(\dd) \mid \dd \in {\mathcal F}\}$,
because this subdiagram is attached at 
the endpoint of a path starting at $*$ and
labeled by the word $u_i'$ of length
less than $|G|$.
%At the end of this process, every vertex of the
%final diagram $\dd_u$ lies on one of these subdiagrams.
Hence we obtain the required intrinsic diameter
bound.

Let $\ps:\bo \dd_u \times [0,1] \ra \dd$ 
be any \oc\ of $(\dd_u,\bo \dd_u)$
at $*$.  Then  $\ps$ 
%and $\pi_\dd \circ \ps$ are both
is $f$-tame for the constant function 
$f(n) := |G|+\max\{idiam(\dd) \mid \dd \in {\mathcal F}\}+\frac{1}{2}$,
since this constant is an upper bound for the
coarse distance from the basepoint to every point of $\dd_u$.
%Then 
%%$H$ admits a \scf\ whose \dhs\ are  
%each \dhy\ $\ph_u'$ is $(f^i+f)$-tame,
%and $f_i+f$
%is Lipschitz equivalent to $f^i$.  
%
Similarly, since the extrinsic diameter of
$\dd_u$ (or, 
indeed, of any other van Kampen diagram for $u$) is at
most $|G|$, the function $\pi_\dd \circ \ps$ is
$g$-tame for the constant function
$g(n) := |G|+\frac{1}{2}$.
That is, we have shown the following.

\begin{proposition}\label{prop:finite}
If $G$ is a finite group, then
over any finite presentation $G$
admits constant intrinsic and
extrinsic \tffs.
\end{proposition}

%\vspace{.1in}

%%%%%%%%%%%%%%%%%%%%%%%%%%%%%%%%%%%%%%%%%%%%%%%%%%%%%%%%%%%%%%%%%%%%%%%%%%%%
%%%%%%%%%%%%%%%%%%%%%%%%%%%%%%%%%%%%%%%%%%%%%%%%%%%%%%%%%%%%%%%%%%%%%%%%%%%%
%%%%%%%%%%%%%%%%%%%%%%%%%%%%%%%%%%%%%%%%%%%%%%%%%%%%%%%%%%%%%%%%%%%%%%%%%%%%

\subsection{Groups with linear extrinsic \tffs}\label{subsec:linear}

%%%%%%%%%%%%%%%%%%%%%%%%%%%%%%%%%%%%%%%%%%%%%%%%%%%%%%%%%%%%%%%%%%%%%%%%%%%%
%%%%%%%%%%%%%%%%%%%%%%%%%%%%%%%%%%%%%%%%%%%%%%%%%%%%%%%%%%%%%%%%%%%%%%%%%%%%

$~$

\vspace{.1in}

For each of the groups considered in this section,
a \dia\ \oc\ $\up:\xx \times [0,1] \ra X$ of the
Cayley complex $X$ for a finite presentation
is constructed, along with
a proof that $\up$ is $f$-tame for a linear function $f$,
elsewhere in the literature.
Consequently, applying Corollary~\ref{cor:etistc}, 
each of these groups admits a linear extrinsic \tff.
Moreover, in each case the group is known to be \stkbl,
and the \rcnf\ constructed in Step 1 of the proof
of Theorem~\ref{thm:stkblhastff} is the same as
the \cnf\ obtained from the \oc\ $\up$ using
the proof of Corollary~\ref{cor:etistc}.  In this
section, we discuss the normal forms from the
bounded flow function in order to find sharper
bounds for their intrinsic \tffs.

\noindent{\bf Thompson's group $F$:}
Thompson's group 
\[F=\langle x_0,x_1 \mid [x_0x_1^{-1},x_0^{-1}x_1x_0],
   [x_0x_1^{-1},x_0^{-2}x_1x_0^2] \rangle\]
is the group of orientation-preserving piecewise linear
homeomorphisms of the unit interval [0,1], satisfying that
each linear piece has a slope of the form $2^i$ for some
$i \in \Z$, and all breakpoints occur in the 2-adics.
In \cite{chst}, Cleary, Hermiller, Stein, and Taback
show that Thompson's group with the generating set
$A=\{x_0^{\pm 1},x_1^{\pm 1}\}$ is (algorithmically)
\fstkbl, 
%with \stkg\ presentation given by 
%the symmetrization of the presentation above. 
%(Moreover, in~\cite[Definition~4.3]{chst} they give an 
%algorithm for computing the \stkg\ map, which 
%can be shown to yield an \astkg\ for $F$.)
as a stepping stone to showing
that $F$ 
%with this presentation also
has a linear \tcf.
We note that although the definition of \dia\ 1-combing
is not included in that paper, and the coarse distance
definition differs slightly, the constructions
of 1-combings in the proofs are \dia\  and
admit Lipschitz equivalent  {\tcf}s.
Hence by Corollary~\ref{cor:etistc}, $F$
has a linear extrinsic \tff.

\begin{corollary}\label{cor:f}
Thompson's group $F$ also has a linear intrinsic \tff.
\end{corollary}

\begin{proof}
Let $\cle=\{(\dd_e,\tht_e)\}$ 
be the \rcnf\ 
 associated 
to the bounded flow function for $F$ from~\cite{chst},
that in turn induces the $f$-tame \dia\ \oc\ in that paper,
for a linear function $f$.
%Then the proofs of Corollary~\ref{cor:etistc} and 
Then Lemma~\ref{lem:edgeimpliesbdry} shows that the
\cfl\  ${\mathcal D}=\{(\dd_w,\ph_w) \mid w \in A^*, w=_F \ep\}$ 
%set of van Kampen diagrams and van Kampen homotopies
induced by $\cle$ by the seashell procedure
satisfies the property that each $\pi_{\dd_w} \circ \ph_w$
is $f$-tame for the same function $f$.

The set of normal forms over the generating set 
$A=\{x_0^{\pm 1},x_1^{\pm 1}\}$
associated to the bounded flow function 
for $F$, in~\cite[Observation~3.6(1)]{chst}), is 

\smallskip

$\cc:=\{w \in A^* \mid \forall \eta \in \{\pm 1\}$, the words
$x_0^\eta x_0^{-\eta}$, $ x_1^\eta x_1^{-\eta}$, and $x_0^2x_1^\eta$ 
are  not subwords of $w$, 

\hspace{1in} and $\forall $ prefixes $ w' $ of $ w,
expsum_{x_0}(w') \le 0\}$,

\smallskip

\noindent  where $expsum_{x_0}(w')$ denotes the 
number of occurrences in $w'$ of the letter
$x_0$ minus the number of occurrences in $w'$ 
of the letter $x_0^{-1}$;
that is, the exponent sum for $x_0$.
Moreover, each of the words in $\cc$
labels a (6,0)-quasi-geodesic path in the
Cayley complex $X$~\cite[Theorem~3.7]{chst}.

A consequence of property~($\dagger$) (p.~\pageref{rmk:vkdhasnf}) and the
seashell construction (Lemma~\ref{lem:edgeimpliesbdry})
is that for 
each word $w \in A^*$ with $w=_F \ep$ and for each vertex
$v$ in $\dd_w$, there is a path in $\dd_w$ from the
basepoint $*$ to the vertex $v$ labeled by the (6,0)-quasi-geodesic
normal form in $\cc$ representing $\pi_{\dd_w}(v)$. 
Then we have 
$d_{\dd_w}(*,v)
\le 6d_X(\ep,\pi_{\dd_w}(v))$.
Let $\tj:\nn \ra \nn$ be the (linear) function
defined by $\tj(n)=6 \lceil n \rceil + 1$.
Lemma~\ref{lem:itversuset} then shows that 
the linear function $\tj \circ f$ is a \tff\ for
Thompson's group $F$.
\end{proof}

On the other hand, we note that Cleary and Taback~\cite{clearytaback}
have shown that Thompson's group $F$ is not almost convex
(in fact, Belk and Bux~\cite{belkbux}
have shown that $F$ is
not even minimally almost convex).  Combining this
with Theorem~\ref{thm:ac} below, 
the identity function cannot be an intrinsic or
extrinsic \tff\ for 
Thompson's group $F$.

%%%%%%%%%%%%%%%%%%%%%%%%%%%%%%%%%%%%%%%%%%%%%%%%%%%%%%%%%%%%%%%%%%%%%%%%%%%%
%%%%%%%%%%%%%%%%%%%%%%%%%%%%%%%%%%%%%%%%%%%%%%%%%%%%%%%%%%%%%%%%%%%%%%%%%%%%
%%%%%%%%%%%%%%%%%%%%%%%%%%%%%%%%%%%%%%%%%%%%%%%%%%%%%%%%%%%%%%%%%%%%%%%%%%%%

%\subsection{Solvable Baumslag-Solitar groups}\label{subsec:bs}

%%%%%%%%%%%%%%%%%%%%%%%%%%%%%%%%%%%%%%%%%%%%%%%%%%%%%%%%%%%%%%%%%%%%%%%%%%%%
%%%%%%%%%%%%%%%%%%%%%%%%%%%%%%%%%%%%%%%%%%%%%%%%%%%%%%%%%%%%%%%%%%%%%%%%%%%%

%$~$

%\vspace{.1in}

\medskip

\noindent{\bf Solvable Baumslag-Solitar groups:}
The solvable Baumslag-Solitar groups are presented by  
$G=BS(1,p)=\langle a,t \mid tat^{-1}=a^p \rangle$ with $p \in \Z$.
In~\cite{chst} Cleary, Hermiller, Stein, and Taback 
show that for $p \ge 3$, the groups $BS(1,p)$ admit
a linear \tcf, 
and hence (from Corollary~\ref{cor:etistc})
a linear extrinsic \tff.

\begin{corollary}\label{cor:bs1p} 
The Baumslag-Solitar group
$BS(1,p)$ has an exponential intrinsic \tff.
\end{corollary}

\begin{proof}
For the generating set $A=\{a,a^{-1},t,t^{-1}\}$,
the infinite complete set of rewriting rules
$ta^\eta \ra a^{\eta p}t$ and
$a^\eta t^{-1} \ra t^{-1}a^{\eta p}$ for $\eta=\pm 1$
together with $t^{-1}a^{pm}t \ra a^m$ for $m \in \Z$
and free reductions yields the set of normal forms
$$
\cc:=\{t^{-i}a^mt^k \mid i,k \in \N \cup \{0\}, m \in \Z,
\text{ and either } p \not| m \text{ or } 0 \in \{i,k\}\}.
$$
The group $BS(1,2)$ has a bounded 
flow function whose tree corresponds to these
normal forms~\cite{britherm},
and the induced \rcnf\ is
exactly the \cnf\ induced by
the 
\dia\ \oc
constructed for $BS(1,p)$ in~\cite{chst},
which is $f$-tame for a linear function $f$.

Proceeding as in the proof of Corollary~\ref{cor:f},
if we let 
${\mathcal D}=\{(\dd_w,\ph_w) \mid w \in A^*, w=_{BS(1,p)} \ep\}$ 
be the 
%set of van Kampen diagrams and van Kampen homotopies
\cfl\ 
induced by this \rcnf\ using seashells,
then each $\pi_{\dd_w} \circ \ph_w$ is 
$f$-tame (Lemma~\ref{lem:edgeimpliesbdry}), and
property~($\dagger$) says that
for each $v \in \dd_w\skz$ of ${\mathcal D}$,
%vertex $v$ of a van Kampen diagram
%$\dd_w$ in this collection, 
there is a path in 
$\dd_w$ from $*$ to $v$ labeled by the normal form
of $\pi_{\dd_w}(v) \in BS(1,p)$.
%$t^{-i}a^mt^k$ 
%of the element $\pi_{\dd_w}(v)$ of $BS(1,p)$.
The normal form $y_g$ of any $g \in G$ can be obtained from
a geodesic representative by applying the set of rewriting rules above;
starting from any word of length $n$, these rules yield
a word of length less than $p^n$.
Then $d_{\dd_w}(*,v) \le l(y_{\pi_{\dd_w}(v)})
%\le i+m+k 
\le j(d_X(\ep,\pi_{\dd_w}(v))$
for the function $j:\N \ra \N$
given by $j(n)=p^n$.  Lemma~\ref{lem:itversuset} now
applies, to show that the 
exponential function $\tilde j \circ f$
is an intrinsic \tff\ for $BS(1,p)$.
\end{proof}

%%%%%%%%%%%%%%%%%%%%%%%%%%%%%%%%%%%%%%%%%%%%%%%%%%%%%%%%%%%%%%%%%%%%%%%%%%%%
%%%%%%%%%%%%%%%%%%%%%%%%%%%%%%%%%%%%%%%%%%%%%%%%%%%%%%%%%%%%%%%%%%%%%%%%%%%%
%%%%%%%%%%%%%%%%%%%%%%%%%%%%%%%%%%%%%%%%%%%%%%%%%%%%%%%%%%%%%%%%%%%%%%%%%%%%

%\subsection{Almost convex groups}\label{subsec:ac}

%%%%%%%%%%%%%%%%%%%%%%%%%%%%%%%%%%%%%%%%%%%%%%%%%%%%%%%%%%%%%%%%%%%%%%%%%%%%
%%%%%%%%%%%%%%%%%%%%%%%%%%%%%%%%%%%%%%%%%%%%%%%%%%%%%%%%%%%%%%%%%%%%%%%%%%%%

%$~$

%\vspace{.1in}

\medskip

\noindent{\bf Almost convex groups:}
One of the original motivations for the definition
of {\tcf}s in~\cite{hmeiermeastame} 
was to imitate Cannon's~\cite{cannon} notion of
almost convexity in a quasi-isometry invariant
property.  Let $G$ be a group with an
inverse-closed generating set $A$, and
%where as usual we assume that $A$ does not contain
%an element that represents $\ep$ in $G$.
let $d_\ga$ be the path metric on the associated
Cayley graph $\ga$.  For $n \in \N$, 
define the sphere $S(n)$ of radius $n$
to be the set of points in $\ga$ a distance
exactly $n$ from the vertex labeled by the
identity $\ep$.  Recall that the ball $B(n)$ of radius $n$
is the set of points in $\ga$ whose path metric distance
to $\ep$ is less than or equal to $n$.

\begin{definition}\label{def:ac}~\cite{cannon}
A group $G$ is {\em almost convex} with respect
to the finite symmetric
generating set $A$ if there is a constant $k$
such that
for all $n \in \N$ and for all $g,h$ in the
sphere $S(n)$ satisfying
$d_\ga(g,h) \leq 2$
(in the corresponding Cayley
graph), there is a path inside 
the ball $B(n)$ from
$g$ to $h$ of length at most $k$.
\end{definition}

Cannon~\cite{cannon} showed that every group $G$
satisfying an almost convexity condition
over a finite generating set $A$ is also finitely presented by
$\pp_k=\langle A \mid R_k\rangle$, where $R_k$ is the
set of nonempty words $w \in A^*$ satisfying $l(w) \le k+2$ and
$w=_G \ep$.
Thiel~\cite{thiel} showed that almost convexity
is a property that depends upon the finite
generating set used.  
%On the other hand, 
%Hermiller and Meier~\cite{hmeiermeastame} showed that
%every group that is almost convex on some
%finite generating set also
%lies in the quasi-isometry invariant 
%class of groups admitting a radial \tci\ 
%with respect to a linear function $\rho$.  Moreover, they
%showed that 
%a group $G$ with generating set $A$ is almost convex
%if and only if the there is a finite set $R$ of
%defining relations for $G$ over $A$ such that
%the pair $(G,\pp)$ satisfies a radial \tci\ 
%with respect to the identity function 
%$f:\nn \ra \nn$ (i.e. $f(n)=n$ for all $n$).

Given an almost convex group $G$, 
let $X$ be the Cayley complex for $\pp_k$, and 
let $\cc=\{y_g \mid g \in G\}$ be the 
set of shortlex  normal forms over $A$ for $G$. 
The ``AC flow function'' 
$\ff:\vece_X \ra \vec P_X$, associated to
the tree of these normal forms, is defined as follows.
As required, $\ff(e):=e$ if $e$ is degenerate.
Suppose that $e=\ega$ is recursive.  If
$g,ga \in S(n)$, then let $\ff(e)$ be any
choice of path $\phi(e)$ of length at most $k$
in $B(n)$ from $g$ to $ga$.  On the other
hand, if $d_X(\ep,g)=n$ and $d_X(\ep,ga)=n+1$,
then write the shortlex normal form
$y_{ga}=y'b$ for some $b \in A$, and
let $\ff(e)$ be a path $\phi(e)$ in $B(n)$ of
length at most $k$ from $g$ to $gab^{-1}$,
followed by the edge $e_{gab^{-1},b}$.
Similarly if $d_X(\ep,g)=n+1$ and $d_X(\ep,ga)=n$,
then write the shortlex normal form
$y_{g}=y'b$ with $b \in A$, and
let $\ff(e)$ be $e_{g,b^{-1}}$ followed by
a path $\phi(e)$ in $B(n)$ of
length at most $k$ from $gb^{-1}$ to $ga$.
Then $\ff$ is a bounded flow function 
for $G$; indeed, each edge of 
$\ff(e)$ is either degenerate or else
has midpoint closer to $\ep$ in the path
metric than the midpoint of $e$.  (See~\cite{britherm}
for more details.)  

In Theorem~\ref{thm:ac} below, 
we show that 
almost convexity
of $(G,A)$ is equivalent to the
existence of a finite set $R$ of
defining relations for $G$ over $A$ such that
the identity function $\iota$
%$\iota:\nn \ra \nn$ 
(i.e.~$\iota(n)=n$ for all $n \in \nn$)
is an intrinsic or extrinsic \tff\ for
$G$ over $\langle A \mid R \rangle$.
In the extrinsic case, an $\iota$-tame \dia\ \oc\ of $(X,\xx)$
for an almost convex group is constructed by Hermiller and Meier
in~\cite[Theorem C]{hmeiermeastame}, and this \oc\ 
induces the same \cnf\ as the AC flow function;
%equivalence of almost
%convexity and the existence of a \cnf\ with
%$\iota$-tame \ehs\ 
%%an extrinsic \tfi\ for $\iota$
%follows closely
%from the equivalence of 
%almost convexity with an $\iota$-tame \dia\ \oc\ 
%shown by Hermiller and Meier
%in~\cite[Theorem C]{hmeiermeastame}, 
%together with the proof of Corollary~\ref{cor:etistc}.
we give further details of that case here to include 
%description of the \stkg\ involved, and
%to include the intrinsic case
%and to remove ``normal'' from the extrinsic filling inequality, 
a minor correction to the proof in that
earlier paper.

\begin{theorem} \label{thm:ac}
Let $G$ be a group with finite generating set $A$, and
let $\iota: \nn \ra \nn$ denote the identity
function.  The following
are equivalent:
\begin{enumerate}
\item The pair $(G,A)$ is almost convex.
\item $\iota$ is an intrinsic \tff\ for $G$ over 
a finite
presentation $\pp=\langle A \mid R \rangle$.
\item $\iota$ is an extrinsic \tff\ for $G$ over 
a finite
presentation $\pp=\langle A \mid R \rangle$.
\end{enumerate}
\end{theorem}

\begin{proof}
%\noindent
{\em (1) implies (3):}
Let $X$ be the Cayley complex of $\pp_k$ and let
$\ff$ be the AC flow function for $G$.
Following the notation of Step 1 of the proof of
Theorem~\ref{thm:stkblhastff}, 
for each recursive edge $e=\ega$ write
$\lbl(\ff(e))=x_g^{-1}z_ex_{ga}$ where
the subword $z_e$ is the label of the subpath
$\phi(e)$ in the description of $\ff$ above.
Let $\cle=\{(\dd_e,\tht_e) \mid e \in E_X \}$
be the associated \rcnf.

Theorem~\ref{thm:stkblhastff} can now be applied, but
unfortunately this result is insufficient. Although
the fact that all of the normal forms in $\cc$
are geodesic implies that the functions $\kti$ and $\kxi$
are the identity, the \tff\ bounds
$\mui$ and $\mue$ are not.
Instead, we follow the 
steps of the algorithm that built the \rcnf\ more carefully.

Let $e \in E_X$ 
with endpoints $g$ and $h$,
and let $n:=\min\{d_X(\ep,g),d_X(\ep,h)\}$.
% ;that is, either $g,h \in S(n)$, or one of these
%points lies in $S(n)$ and the other is in $S(n+1)$.
Let $\he$, 
with endpoints $\hat g$ and $\hat h$,
be the edge corresponding to $e$ in the boundary of the
van Kampen diagram $\dd_e$.
Now the paths $\pi_{\dd_e} \circ \tht_e(\hat g, \cdot)$
and $\pi_{\dd_e} \circ \tht_e(\hat h,\cdot)$ follow the geodesic
paths in $X$ that start from $\ep$ and
that are labeled by the words $y_g$ and $y_h$ 
at constant speed, and so are $\iota-tame$.
Let $p$ be any
point in the interior of $\he$.

{\em Case I.  $e$ is degenerate.}
Then $\dd_e$ is a line segment with no 2-cells,
and the path $\pi_{\dd_e} \circ \tht_e(p,\cdot)$ follows a
geodesic in $\xx$.  
Hence this path is $\iota$-tame.

{\em Case II.  $e$ is recursive.}
We proceed by Noetherian induction, using 
the notation in Section~\ref{subsec:stkbl}.  
%In that construction, \ehs\ are constructed for directed
%edges; 
By slight abuse of notation,
let $e$ also denote the directed edge from $g$ to $h$
that yields the element $(\dd_e,\tht_e)$ of $\cle$.
%Recall that this recursive procedure uses a factorization of the 
%word $c(e)$ as $c(e)=x_g \tc_e x_h$.  In our
%definition of $c(e)$ above, we defined this factorization
%so that for each edge $e'$ (no matter whether $e'$
%is in $\dgd$ or $\ves$)
%in the $\tc_e$ path, we have
%$f_\ga(e')<f_\ga(e)$.
On the interval $[0,a_p]$,
the path $\tht_e(p,\cdot)$ follows a
path $\tht_e'(\Xi_e(p,0),\cdot)$ in a subdiagram of
$\dd_e$ that is either a path from an \ehy\  for an edge $e_i$
of $X$ satisfying $e_i <_\ff e$, or
a line segment labeled by a shortlex normal form.  
%In the case that the word $\tx_e$ is empty
%(which includes the base case of our induction),
%$\tht_e(p,\cdot)$ follows the geodesic path labeled
%$z_q$ in the subdiagram $\dd_e'$.
Hence either by induction or case I,
the path $\pi_{\dd_e} \circ \tht_e(p,\cdot)$ is $\iota$-tame
on $[0,a_p]$.

On the interval $[a_p,1]$, the path 
$\tht_e(p,\cdot)$ follows the path $\Xi_e(p,\cdot)$
from the point $\Xi_e(p,0)$ (in the subpath of $\bo \hc_e$
labeled $\tc_e$, whose image in $X$
is contained in $B(n)$) through the interior of
the 2-cell $\hc_e$ of $\dd_e$ to the point $p$.
%If $\td_X(\ep,\pi_{\dd_e}(p))=n+\frac{1}{4}$, 
We have
$\td_X(\ep,\pi_{\dd_e}(\Xi_e(p,0))) \le n$, 
$\td_X(\ep,\pi_{\dd_e}(\Xi_e(p,t))) = n+\frac{1}{4}$
for all $t \in (0,1)$, 
and $\td_X(\ep,\pi_{\dd_e}(\Xi_e(p,1))) = 
\td_X(\ep,p)=n+\frac{1}{2}$.
%If $\td_X(\ep,\pi_{\dd_e}(p))=n+\frac{3}{4}$ for some $n \in \N$, then
%%in the notation of the definition of the \ehy\ 
%%associated to the \stkg\  $(\cc,\ves,c)$, the edge $e"$
%%from $s'$ to $h$ in $c(e)$ is part of
%%the edge path in $\ct$ from $s$ to $h$ labeled by
%%the suffix $x_h$ of the normal form $z_h$,
%%and so the point $\Xi_e(p,0)$ does not lie on
%%the corresponding edge $\widehat{e"}$ of $\bo \hc_e$, 
%%with the possible exception of the endpoint 
%%$\widehat{s'}$.  
%%Then in this case we have
%$\td_X(\ep,\Xi_e(p,0)) \le n$, $\td_X(\ep,\Xi_e(p,t)) = n+\frac{3}{8}$
%for all $t \in (0,1)$, 
%and $\td_X(\ep,\Xi_e(p,1)) = n+\frac{3}{4}$.
Hence the path $\pi_{\dd_e} \circ \Xi_e(p,\cdot)$ is $\iota$-tame.
Combining these, we have that $\pi_{\dd_e} \circ \tht_e$
is also $\iota$-tame in Case II. 

Since in the \cnf\  
%given by 
$\cle$,
each map $\pi_{\dd_e} \circ \tht_e$  is $\iota$-tame,
%, and 
%hence the same is true for the van Kampen
%homotopies of the \rcf\ $(\cc,\clf)$ induced
%by $\cle$, by
%%the proof of (4) $\Rightarrow$ (1) in 
%%Proposition~\ref{prop:htpydomain}. 
Lemma~\ref{lem:edgeimpliesbdry} then shows that
$\iota$ is a \tff\ for $G$ over $\pp_k$.

{\em (1) implies (2)}: 
As noted in property~($\dagger$), the \rcnf\ 
%$(\cc,\cle)$ 
$\cle$ for the AC flow function above satisfies the 
property that 
for every vertex $v$ in a van Kampen diagram $\dd$ of $\cle$, there is
a path in $\dd$ from $*$ to $v$ labeled by the
shortlex normal form for the element $\pi_{\dd}(v)$
of $G$.  Since these normal forms label geodesics in $X$,
it follows that intrinsic and extrinsic distances (to the basepoints)
in the diagrams $\dd$ of $\cle$ are the same.
Lemma~\ref{lem:edgeimpliesbdry} again
applies to show $\iota$ is also an intrinsic \tff\ for $G$ over $\pp_k$.

{\em (2) or (3) implies (1)}: 
The proof of this direction in the extrinsic case
closely follows
the proof of~\cite[Theorem~C]{hmeiermeastame}, and
%Proposition~\ref{thm:etistc}, and 
the proof in the intrinsic case is quite similar.
\end{proof}

\subsection{Combable groups}\label{subsec:combable}
%Groups with a fellow traveler property

%%%%%%%%%%%%%%%%%%%%%%%%%%%%%%%%%%%%%%%%%%%%%%%%%%%%%%%%%%%%%%%%%%%%%%%%%%%%
%%%%%%%%%%%%%%%%%%%%%%%%%%%%%%%%%%%%%%%%%%%%%%%%%%%%%%%%%%%%%%%%%%%%%%%%%%%%

$~$

\vspace{.1in}

In this section we consider a class of finitely presented groups which
admit a rather different procedure for constructing van Kampen diagrams,
namely combable groups.  
Let $\cc=\{y_g \mid g \in G\}$ be a set of
of normal forms over a finite inverse-closed
generating set $A$ for the group $G$, such that
each normal form $y_g$ labels a simple path in
the Cayley graph $\ga$.
The set $\cc$
satisfies a (synchronous)
{\em $K$-fellow traveler property} for a constant
$K \ge 1$ if whenever $g,h \in G$ and $a \in A$ with $ga=_G h$, and
we write $y_g=a_1 \cdots a_m$ and $y_h=b_1 \cdots b_n$ with
each $a_i,b_i \in A$
(where, without loss of generality, we assume $m \le n$),
then for all $1 \le i \le m$ we have
$d_\ga(a_1 \cdots a_i,b_1 \cdots b_i) \le K$, and
for all $m< i \le n$ we have $d_\ga(g,b_1 \cdots b_i) \le K$.
The group $G$ is {\em combable} if $G$ admits a language
of normal forms satisfying a $K$-fellow traveler property.
(Note that this notion of combable is not the same
as the tame combability discussed earlier in this
paper.)

Before imposing a geometric restriction on the
normal forms, we first consider the more general
case of combable groups with respect to
simple word normal forms.

\begin{proposition}\label{prop:combing}
Let $G$ be a group with a finite generating set $A$ and 
Cayley graph $\ga$.  If $G$ 
has a
set $\cc$ of normal forms that label
simple paths in $\ga$ and satisfy a $K$-fellow traveler
property such that for all $n \in \N$ the set
$$T_n:=\{w \in A^* \mid d_\ga(\ep,w) \le n \text{ and } 
    w \text{ is a prefix of a word in }\cc\}$$
is finite, then
$G$ admits well-defined intrinsic and extrinsic \tffs.
\end{proposition}

\begin{proof}
The $K$-fellow traveler property implies that the presentation 
$\pp_K:=\langle A \mid R_K \rangle$, where
$R_K=\{w \in A^* \setminus \{1\} \mid l(w) \le 2K+2 \text{ and }
w=_G \ep\}$, is a finite (symmetric) presentation~for~$G$.  

We build a \cnf\ for $G$ over $\pp_K$ as follows.
Let $X$ be the Cayley complex, and let $\ega \in \vece_X$.
As above, write the normal forms $y_g,y_{ga} \in \cc$ as
$y_g=a_1 \cdots a_m$ and $y_{ga}=b_1 \cdots b_n$
with each $a_i,b_i \in A$.  For each $m<i \le n$, 
let $a_i$ denote the empty word, and conversely if
$n<i \le m$ let $b_i:=1$.
Define the words $c_0 := 1$, $c_n:=a$, and  
for each $1 \le i \le n-1$, let $c_i$ be a word in $A^*$
labeling a geodesic path in $X$ from $a_1 \cdots a_i$ to
$b_1 \cdots b_i$;
thus $l(c_i) \le K$ for all $i$.
The diagram $\dd_e$ is built by successively gluing 2-cells 
labeled $a_ic_ib_{i}^{-1}c_{i-1}^{-1}$, for $1 \le i \le n$,
along their common $c_i$ boundaries.  (When $c_{i-1}=c_i=1$
and $a_i=b_i$, an edge is glued rather than a 2-cell.)
The
diagram $\dd_e$ is ``thin'', in that it has only the
width of (at most) one 2-cell.
An \ehy\  $\tht_e$ for this diagram can be constructed to go
successively through each 2-cell in turn
from the basepoint $*$ to the edge $\he$ corresponding to $e$; 
see Figure~\ref{fig:ladder} for an illustration.
\begin{figure}
\begin{center}
\includegraphics[width=3.7in,height=0.8in]{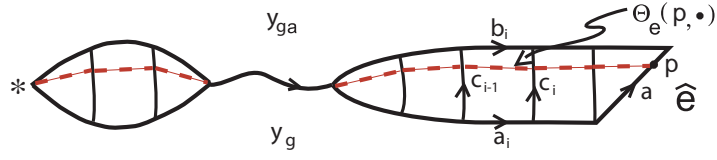}
\caption{``Thin'' van Kampen diagram $\dd_e$}\label{fig:ladder}
\end{center}
\end{figure}
Choosing one direction for each undirected edge of $X$, 
let $\cle=\{(\dd_e,\tht_e) \mid e \in E_X\}$ be the corresponding \cnf.

Let $\he=\path_{\dd_e}(y_g,y_{ga})$ be the edge of $\bo\dd_e$
corresponding to
the edge $e$ of $X$.  Also let $p$ be any
point in $\he$, and
suppose that $0 \le s < t \le 1$.
Applying the ``thinness'' of $\dd_e$,
every point of $\dd_e$ lies in some closed cell
of $\dd_e$ that also contains a vertex in the boundary
path $q:=\path_{\dd_e}(1,y_g)$.
In particular, there are vertices $v_s$ and $v_t$ on $q$ 
such that the point $\tht_e(p,s)$ and 
the point $v_s$ occupy
the same closed 0, 1, or 2-cell in $\dd_e$,
$\tht_e(p,t)$ and $v_t$ occupy
a common closed cell, and $v_s$ is reached before or at $v_t$ 
when traversing the directed path $q$.  
Let $\maxr \le 2K+2$ denote the length
of the longest relator in the presentation $\pp_K$. 
Then we have
$|\td_{\dd_e}(*,\tht_e(p,s))-\td_{\dd_e}(*,v_s)| \le \maxr+1$ and
$|\td_{X}(\ep,\pi_{\dd_e}(\tht_e(p,s)))-\td_{X}(\ep,\pi_{\dd_e}(v_s))| 
\le \maxr+1$, and similarly for the pair $\tht_e(p,t)$ and $v_t$.
Write the word $y_g=y_1y_2y_3$ where the vertex $v_s$
is the end of the subpath $\path_{\dd_e}(1,y_1)$ of $q$, and
and similarly the vertex $v_t$ occurs between the $y_2$ and $y_3$
subwords.  Note that 
%$y_1$ and 
$y_1y_2$ is a
prefix of a normal form word in $\cc$,
and so satisfies 
%$y_1 \in S_{d_X(\ep,\pi_{\dd_e}(v_s))}$ and 
$y_1y_2 \in T_{d_X(\ep,\pi_{\dd_e}(v_t))}$.

Define the function $t^i:\N \ra \N$ by
$t^i(n):=\max\{l(w) \mid w \in T_n\}$.  Since each 
set $T_n$ is 
finite, $t^i$ is well-defined.
Using the fact that $t^i$ is nondecreasing, we have
\begin{eqnarray*}
\td_{\dd_e}(*,\tht_e(p,s))
&\le&   \td_{\dd_e}(*,v_s)+ \maxr+1 
\text{\ \ \ \ }\le \text{\ \ \ \ }    l(y_1)+\maxr+1 \\
&\le&    l(y_1y_2)+\maxr+1 
\text{\ \ \ \ } \le \text{\ \ \ \ } t^i(d_X(\ep,\pi_{\dd_e}(v_t)))+\maxr+1 \\
&\le&    t^i(\td_{\dd_e}(*,v_t))+\maxr+1 
 \le  t^i(\lceil \td_{\dd_e}(*,\tht_e(p,t))\rceil+\maxr+1)+\maxr+1
\end{eqnarray*}
Then the
function $n \mapsto t^i(\lceil n \rceil+2K+3)+2K+3$
is an intrinsic \tff\ for $G$.

Next define $t^e:\N \ra \N$ by
$t^e(n):=\max\{d_X(\ep,v) \mid v \text{ is a 
prefix of a word in } T_n\}$.
Again, this is a well-defined nondecreasing function.
In this case, we note that since $y_1$ is a prefix of $y_1y_2$,
then $y_1$ is a prefix of a word in $T_{d_X(\ep,\pi_{\dd_e}(v_t))}$.
Then
\begin{eqnarray*}
\td_{X}(\ep,\pi_{\dd_e}(\tht_e(p,s)))
&\le&   \td_{X}(\ep,\pi_{\dd_e}(v_s))+\maxr+1 
\text{\ \ } = \text{\ \ } d_X(\ep,y_1)+\maxr+1\\
& \le & t^e(d_X(\ep,\pi_{\dd_e}(v_t)))+\maxr+1\\ 
&\le & t^e(\lceil \td_X(\ep,\pi_{\dd_e}(\tht_e(p,t)))\rceil+\maxr+1)+\maxr+1.
\end{eqnarray*}
Hence $n \mapsto t^e(\lceil n \rceil+2K+3)+2K+3$
is an extrinsic \tff\ for $G$.
\end{proof}

\begin{comment} 
Note that since the ball of radius $K$ about the vertex
$g$ in $X$ contains at most $(|A|+1)^K$
elements of $G$, and the words $b_1 \cdots b_i$, as prefixes
of $y_h$, must each be a normal form of one
of this finite set of elements, we have 
$n-m < (|A|+1)^K$.
\end{comment}

We highlight two special cases in which the hypothesis
of Proposition~\ref{prop:combing}, that each set
$T_n$ is finite, is satisfied.
The first is the case in which the set of normal
forms is prefix-closed.  For this case, the functions
$t^i=\kti$ and $t^e=\kte$ are the functions defined
in Section~\ref{subsec:stkbl}, and so we have the following.

\begin{corollary}
If $G$ has a prefix-closed set of normal forms that satisfies
a $K$-fellow traveler property, then 
$ n \mapsto \kti(\lceil n \rceil+2K+3)+2K+3$
is an intrinsic \tff, and
$n \mapsto \kte(\lceil n \rceil+2K+3)+2K+3$
is an extrinsic \tff, for $G$.
\end{corollary}

The second is the case in which the set of  normal forms
is quasi-geodesic; that is, there are constants
$\lambda,\lambda' \ge 1$ such that every word in this 
set is a   $(\lambda,\lambda')$-quasi-geodesic.
For a group $G$ with generators $A$
and Cayley graph $\ga$,
a word $y \in A^*$ is a {\em $(\lambda,\lambda')$-quasi-geodesic} 
if whenever $y=y_1y_2y_3$, then 
$l(y_2) \le \lambda d_\ga(\ep,y_2)+\lambda'$.
Actually, we only need a slightly weaker property,
that this inequality holds whenever $y_2$ is
a prefix of $y$ (i.e., when $y_1=1$).
In this case, the set $T_n$ is a subset of
the finite set %$\cup_{i=0}^{\lambda n+\lambda'} A^i$
of words over $A$ of length at most $\lambda n+\lambda'$.
Then $t^e(n) \le t^i(n) \le \lambda n+\lambda'$ for all $n$.
Putting these results together yields the following.

\begin{corollary}\label{cor:combable}
If a finitely generated group $G$ admits a 
quasi-geodesic language of normal forms
that label simple paths in the Cayley graph and that
satisfy a $K$-fellow traveler property, then $G$ has linear
intrinsic and extrinsic \tffs.
%If a finitely generated group $G$ admits a 
%quasi-geodesic language of simple word normal forms
%satisfying a $K$-fellow traveler property, then $G$ satisfies linear
%intrinsic and extrinsic \tfs.
\end{corollary}

%%%%%%%%%%%%%%%%%%%%%%%%%%%%%%%%%%%%%%%%%%%%%%%%%%%%%%%%%%%%%%%%%%%%%%%%%%%%
%%%%%%%%%%%%%%%%%%%%%%%%%%%%%%%%%%%%%%%%%%%%%%%%%%%%%%%%%%%%%%%%%%%%%%%%%%%%
%%%%%%%%%%%%%%%%%%%%%%%%%%%%%%%%%%%%%%%%%%%%%%%%%%%%%%%%%%%%%%%%%%%%%%%%%%%%

\section{Quasi-isometry invariance for \tffs}\label{sec:qiinv}

%%%%%%%%%%%%%%%%%%%%%%%%%%%%%%%%%%%%%%%%%%%%%%%%%%%%%%%%%%%%%%%%%%%%%%%%%%%%
%%%%%%%%%%%%%%%%%%%%%%%%%%%%%%%%%%%%%%%%%%%%%%%%%%%%%%%%%%%%%%%%%%%%%%%%%%%%

In this section we give the proof of Theorem~\rfthmitisqi,
showing that, as with the 
diameter functions~\cite{bridsonriley},~\cite{gersten}, 
\tffs\  are also
quasi-isometry invariants, up to Lipschitz equivalence
of functions (and in the intrinsic case, up to 
sufficiently large set of 
defining relations).  In the extrinsic case,
this follows from Corollary~\ref{cor:etistc}
and the proof of Theorem~\cite[Theorem~A]{hmeiermeastame},
but with a slightly different definition of
coarse distance.  We include the details for
both  here, to illustrate the difference between
the intrinsic and extrinsic cases.

\smallskip

\noindent{\bf Theorem~\rfthmitisqi.}  {\em
Suppose that $(G,\pp)$ and $(H,\pp')$ 
are quasi-isometric groups with
finite presentations.
If $f$ is an extrinsic \tff\ for $G$ over $\pp$,
then $(H,\pp')$
has an extrinsic
\tff\  that is Lipschitz equivalent to $f$.
If $f$ is an intrinsic \tff\ for $G$ over $\pp$,
then after adding
all relators of length up to a sufficiently 
large constant to the 
presentation $\pp'$, 
the pair $(H,\pp')$
has an intrinsic
\tff\  that is Lipschitz equivalent to $f$.
}

\smallskip

%\begin{theorem}\label{thm:itisqi}
%Suppose that $(G,\pp)$ and $(H,\pp')$ 
%are quasi-isometric groups with
%finite presentations.
%If $(G,\pp)$ satisfies an extrinsic 
%\tfi\  with respect to $f$, then $(H,\pp')$
%satisfies an extrinsic
%\tfi\  with respect to a function
%that is Lipschitz equivalent to $f$.
%If $(G,\pp)$ satisfies an intrinsic 
%\tfi\  with respect to $f$, then after adding
%all relators of length up to a sufficiently 
%large constant to the 
%presentation $\pp'$, 
%the pair $(H,\pp')$
%satisfies an intrinsic
%\tfi\  with respect to a function
%that is Lipschitz equivalent to $f$.
%\end{theorem}

\begin{proof}%[Proof of Theorem~\rfthmitisqi]
Write the finite presentations $\pp=\langle A \mid R \rangle$
and $\pp'=\langle B \mid S\rangle$; as usual
we assume that these presentations are symmetric.

If $G$ is a finite group, then $H$ is also finite.
In this case, Proposition~\ref{prop:finite}
shows that there is a constant function $f(n) \equiv C$
that is an intrinsic and extrinsic tame filling function 
for $H$ over $\pp'$.
Since increasing the function preserves
tameness, then for any intrinsic \tff\ $f^i$
for $G$ over $\pp$, the function
$f^i+f$ is also an intrinsic \tff\ for
$H$, and $f^i+f$ is Lipschitz  equivalent to $f^i$.
Similarly, for an extrinsic \tff\ $f^e$ for $G$,
the function $f^e+f$
is Lipschitz equivalent to $f^e$ and is
an extrinsic \tff\ for $H$.

For the remainder of this proof we assume that
the group $G$ (and hence also the group $H$) is infinite.
Let $X$ be the 2-dimensional Cayley complex for the
pair $(G,\pp)$, and let $Y$ be the Cayley complex associated
to $(H,\pp')$.  
Let $d_X$, $d_Y$ be the path metrics in $X$ and $Y$
(and hence also the word metrics in $G$ and $H$ with respect to the 
generating sets $A$ and $B$), respectively.
%Let $\maxr:=\max\{l(r) \mid r \in R\}$ be the length
%of the longest relator in $R$.
%and $\maxr'$ be the lengths of the longest
%relators in the presentations $\pp$ and $\pp'$, respectively.
%and let $\maxp:=\maxr+1$ and $\maxp':=\maxr'+1$.
  
Quasi-isometry of these groups means that there are
functions $\phi:G \ra H$ and $\theta:H \ra G$ 
and a constant $k>1$ such that for all $g_1,g_2 \in G$
and $h_1,h_2 \in H$, we have
\begin{enumerate}
\item $\frac{1}{k}d_X(g_1,g_2)-k \le d_Y(\phi(g_1),\phi(g_2)) 
          \le kd_X(g_1,g_2)+k$
\item $\frac{1}{k}d_Y(h_1,h_2)-k \le d_X(\theta(h_1),\theta(h_2)) 
              \le kd_Y(h_1,h_2)+k$
\item $d_X(g_1,\theta \circ \phi(g_1)) \le k$
\item $d_Y(h_1,\phi \circ \theta(h_1)) \le k$
\end{enumerate}
By possibly increasing the constant $k$, we may also
assume that $k>2$ and that
$\phi(\ep_G)=\ep_H$ and $\theta(\ep_H)=\ep_G$,
where $\ep_G$ and $\ep_H$ are the identity elements of 
the groups $G$ and $H$, respectively.

We extend the functions $\phi$ and $\theta$ to functions
$\tp:G \times A^* \ra B^*$ and $\tth:H \times B^* \ra A^*$
as follows.  
Let $\tilde A \subset A$ be a subset containing 
exactly one element for each inverse pair $a,a^{-1} \in A$. 
Given a pair $(g,a) \in G \times \tilde A$, using property (1) above
we let
$\tp(g,a)$ be (a choice of) a nonempty word 
of length at most $2k$ labeling a 
path in the Cayley graph $Y\sko$ from the vertex $\phi(g)$ to the
vertex $\phi(ga)$ (in the case that $\phi(g)=\phi(ga)$,
we can choose $\tp(g,a)$ to be the nonempty word 
$bb^{-1}$ for some choice of $b \in B$).
We also define $\tp(g,a^{-1}):=\tp(ga^{-1},a)^{-1}$.
Then for any $w = a_1 \cdots a_m$
with each $a_i \in A$, define $\tp(g,w)$ to be the
concatenation
$\tp(g,w) := \tp(g,a_1) \cdots \tp(ga_1 \cdots a_{m-1},a_m)$.
Note that for $w \in A^*$:

(5) the word lengths satisfy
$l(w) \le l(\tp(g,w)) \le 2kl(w)$, and

(6) the word $\tp(\ep_G,w)$ 
represents the element $\phi(w)$ in $H$.

\noindent The function $\tth$ is defined analogously.

Using Proposition~\ref{prop:htpydomain},
% and Corollary~\ref{cor:etistc}, 
we will
prove the theorem 
%using relaxed \tfs\ 
utilizing {\scf}s rather than {\cfl}s.
%via disk homotopies.
For the group $G$ with presentation $\pp$, fix a collection 
${\cld} = \{(\dd_w,\ph_w) \mid w \in A^*, w=_G \ep_G\}$
such that for each $w$, $\dd_w$ is a van Kampen
diagram for $w$ and $\ph_w$ is a \dhy\ of $\dd_w$.
Further, we assume that 
either all of the $\ph_w$ are $f^i$-tame or 
all $\pi_{\dd_w} \circ \ph_w$ are $f^e$-tame, 
where $f^i,f^e:\nn \ra \nn$
are nondecreasing functions.

Recall from Remark~\ref{rmk:infinite}, we know that
 $f^i(n) > n-2$ and $f^e(n) > n-2$
for all $n \in \nn$; we use these inequalities
repeatedly below.
Let $\maxr:=\max\{l(r) \mid r \in R\}$
denote the maximum length of a relator in the presentation $\pp$.
%%see my notes p. 269.

Now suppose that $u'$ is any word in $B^*$
with $u' =_H \ep_H$.  We will construct
a van Kampen diagram for $u'$,
following the method of
\cite[Theorem 9.1]{bridsonriley}.  At each of the four 
successive steps,
we obtain a van Kampen diagram for a specific word; 
we will also keep track of {\oc}s and analyze
their tameness, ending
with a diagram and \dhy\  
for $u'$.

\smallskip

{\em Step I.  For $u:=\tth(\ep_H,u')\in A^*$:} 
%Writing $u'=b_1 \cdots b_n$ with each $b_i \in B$,
%then $u'$ labels a circuit 
%$\tht(\ep_H,b_1) \cdots \tht(b_1 \cdots b_{n-1},b_n)$ starting
%and ending a $\ep_G$ in $\xx$.
Note that (6) implies 
$u=_G \theta(u')=_G\theta(\ep_H)=_G\ep_G$, and
so the collection ${\cld}$ contains a
van Kampen diagram $\dd_u$ for $u$ and an associated
\dhy\  
$\ph_u:C_{l(u)} \times [0,1] \ra \dd_u$.
Note that either $\ph_u$ is $f_1^i:=f^i$-tame
or $\pi_{\dd_u} \circ \ph_u$ is $f_1^e:=f^e$-tame.

\smallskip

{\em Step II.  For $z'':=\tp(\ep_G,u)=\tp(\ep_G,\tth(\ep_H,u')) \in B^*$:}  
We build a finite, planar, contractible, 
combinatorial 2-complex $\Omega$ from
$\dd_u$ as follows.
%As usual, let $\pi_{\dd_u}: \dd_u \ra X$
%be the map taking the basepoint $*$ of
%$\dd_u$ to $\ep_G$ and preserving directed
%labeled edges.
Given any edge $e$ in $\dd_u$, choose
a direction, and hence a label $a_e$, for $e$,
and let $v_1$ be the initial
vertex of $e$.
Replace $e$ with a directed edge path
$\hat e$ labeled by the (nonempty) word
$\tp(\pi_{\dd_u}(v_1),a_e)$.  Repeating 
this for every edge of the complex $\dd_u$
results in the 2-complex $\Omega$.

Note that 
$\Omega$ is a 
van Kampen diagram for 
the word $z''$ %:=\tp(\ep_G,u)=\tp(\ep_G,\tth(\ep_H,u')) \in B^*$
with respect to the presentation
$\pp''=\langle B \mid S \cup S'' \rangle$ of $H$,
where $S''$ is the set of all nonempty words over
$B$ of length at most 
$2k \maxr$ that represent $\ep_H$.
%Let $\maxp''=2k\maxp$, and note that $\maxp''$
%is at least the length of the longest
%relator in $\pp''$.
%Also, l
Let $Y''$ be the Cayley complex for $\pp''$
and let $\pi_\Omega:\Omega \ra Y''$ 
be the function preserving basepoints and directed
labeled edges.
Using the fact that the only difference between $\dd_u$ and $\Omega$
is a replacement of edges by edge paths, we 
define $\alpha:\dd_u \ra \Omega$ to be the continuous map
taking each vertex and each interior point of
a 2-cell of $\dd_u$ to the same point of $\Omega$,
and taking each edge $e$ to the corresponding edge path $\hat e$.

Writing $u=a_1 \cdots a_m$ with each $a_i \in A$,
then $z''=c_{1,1} \cdots c_{1,j_1} \cdots c_{m,1} \cdots c_{m,j_m}$
where each $c_{i,j} \in B$ and $c_{i,1} \cdots c_{i,j_i}$ is
the nonempty word labeling the edge path $\widehat{e_i}$
of $\bo \Omega$ that is the image under $\alpha$ of the $i$-th edge
of the boundary path of $\dd_u$.
Recall that $C_{l(u)}$ is the circle $S^1$ with a 
1-complex structure of $l(u)$ vertices and edges.
Let the 1-complex $C_{l(z'')}$ be a refinement of the
complex $C_{l(u)}$, so that the $i$-th edge of 
$C_{l(u)}$ is replaced by $j_i \ge 1$ edges for each $i$, and
let $\hat \alpha:C_{l(z'')} \ra C_{l(u)}$ be the identity
on the underlying circle.
%\coment{Here is where the problem occurs
%with allowing some of the edge paths
%$\hat e$ to be constant paths, and so
%edges of $\dd_u$ are replaced by a vertex
%(ie gluing the endpoints of $e$) in $\Omega$ -
%hence my choice of nonempty words in the
%definition of $\tilde \phi$.}
Finally, define the map 
$\omega:C_{l(z'')} \times [0,1] \ra \Omega$ by
$\omega:=\alpha \circ \ph_u \circ (\hat \alpha \times \idmap)$.
%$\omega(p,t):=  \alpha(\ph_u(\hat \alpha(p),t))
%This map $\omega$
%satisfies conditions (d1)-(d2) of the definition
%of \dhy.

Next we analyze the intrinsic tameness of $\omega$.  
In this
step we have only replaced edges by nonempty edge
paths of length at most $2k$, and hence
for each vertex $v$ in 
$\dd_u$ we have
$
\td_{\dd_u}(*,v) 
\le \td_{\Omega}(*,\alpha(v)) 
\le 2k\td_{\dd_u}(*,v)~.
$
For a point $q$ in the interior of an edge
of $\dd_u$, let $v$ be a vertex
in the same closed cell; then 
$|\td_{\dd_u}(*,q)-\td_{\dd_u}(*,v)| <1$
and $|\td_{\Omega}(*,\alpha(q))-\td_{\Omega}(*,\alpha(v))| < 2k$.
%in this case, 
%\[
%\td_{\dd_u}(*,q) 
%\le \td_{\Omega}(*,\alpha(q)) 
%\le \td_{\Omega}(*,\alpha(v))+2k 
%\le 2k\td_{\dd_u}(*,v)+2k 
%\le 2k(\td_{\dd_u}(*,q)+1)+2k~.
%\]
For a point $q$ in the interior of a 2-cell of $\dd_u$,
let $v$ be a vertex in the closure of this cell with 
$\td_{\dd_u}(*,v) \le \td_{\dd_u}(*,q)+1$.  Then $\alpha(v)$
is a vertex in the closure of the open 2-cell of $\Omega$
containing $\alpha(q)$, and the boundary path of this
cell has length at most $2k\maxr$.
That is, 
$|\td_{\dd_u}(*,q)-\td_{\dd_u}(*,v)| <1$
and $|\td_{\Omega}(*,\alpha(q))-\td_{\Omega}(*,\alpha(v))| < 2k\maxr$.
%Hence 
%\[
%\td_{\dd_u}(*,q) 
%\le \td_{\Omega}(*,\alpha(q)) 
%\le \td_{\Omega}(*,\alpha(v))+\maxp''
%\le  2k\td_{\dd_u}(*,v)+2k\maxp
%\le 2k(\td_{\dd_u}(*,q)+1)+2k\maxp~.
%\]
Thus for all $q\in \dd_u$, we have $\td_{\dd_u}(*,q) 
\le \td_{\Omega}(*,\alpha(q)) \le 2k \td_{\dd_u}(*,q) +4k+2k\maxr$.

Now suppose that $p$ is any point in $C_{l(z'')}$ and 
$0 \le s < t \le 1$.  In
the case that $\ph_u$ is
$f_1^i$-tame, combining the inequalities above and the fact
that $f_1^i$ is nondecreasing yields
%\hspace{-0.4in} $\td_\Omega(*,\omega(p,s)) = 
%          \td_\Omega(*,\alpha(\ph_u(\hat \alpha(p),s))$
%\hspace{-0.3in} $\le 2k \td_{\dd_u}(*,\ph_u(\hat \alpha(p),s)) +4k+2k\maxp
%  \le  2k f_1^i(\td_{\dd_u}(*,\ph_u(\hat \alpha(p),t)))+4k+2k\maxp$
%\hspace{-0.3in} $\le  2k f_1^i(\td_{\Omega}(*,\alpha(\ph_u(\hat \alpha(p),t))))
%                 +4k+2k\maxp
%       = 2k f_1^i(\td_{\Omega}(*,\omega(p,t))+4k+2k\maxp~$.
\begin{eqnarray*}
\td_\Omega(*,\omega(p,s)) &=& \td_\Omega(*,\alpha(\ph_u(\hat \alpha(p),s))
         \le 2k \td_{\dd_u}(*,\ph_u(\hat \alpha(p),s)) +4k+2k\maxr\\
  & < & 2k (f_1^i(\td_{\dd_u}(*,\ph_u(\hat \alpha(p),t)))+2)+4k+2k\maxr \\
  &\le & 2k f_1^i(\td_{\Omega}(*,\alpha(\ph_u(\hat \alpha(p),t))))+8k+2k\maxr
      = 2k f_1^i(\td_{\Omega}(*,\omega(p,t))+8k+2k\maxr~.
\end{eqnarray*}
Hence $\omega$ is $f_2^i$-tame for the nondecreasing function 
$f_2^i(n):=2k f_1^i(n)+8k+2k\maxr$.

Next consider the extrinsic tameness
of $\omega$.  
For any vertex $v$ in $\dd_u$,  let $w_v$ 
%$w_v=a_1 \cdots a_m$ (with each $a_i \in A$) 
be a word labeling
a path in $\dd_u$ from $*$ to $v$.  
%Then $\tp(\ep_G,a_1)\cdot \tp(a_1,
Using note (6) above,
we have 
$\phi(\pi_{\dd_u}(v)) =_H 
\phi(w_v) =_H \tp(\ep_G,w_v)=_H
\pi_\Omega(\alpha(v))$, by our
construction of $\Omega$. 
Quasi-isometry property (1) then gives
\[
\frac{1}{k}d_X(\ep_G,\pi_{\dd_u}(v))-k 
  \le d_Y(\ep_H,\phi(\pi_{\dd_u}(v))) 
    = d_Y(\ep_H,\pi_\Omega(\alpha(v)))
  \le kd_X(\ep_G,\pi_{\dd_u}(v))+k~.
\]
Since the generating sets of the presentations $\pp'$ and
$\pp''$ of $H$ are the same, the Cayley graphs 
and their path metrics $d_Y=d_{Y''}$ are also the same.  
% and
%$\td_Y=\td_{Y''}$ on the Cayley graphs.
%We note that 
As in the intrinsic case above,
for a point $q$ in the interior of an edge or 2-cell of $\dd_u$,
there is a vertex $v$ in the same closed cell with
$|\td_X(\ep_G,\pi_{\dd_u}(q))-\td_X(\ep_G,\pi_{\dd_u}(v))|<1$ and
$|\td_{Y''}(\ep_H,\pi_\Omega(\alpha(q)))-
         \td_{Y''}(\ep_H,\pi_\Omega(\alpha(v)))|<2k(\maxr+1)$.
%\begin{eqnarray*}
%\td_X(\ep_G,\pi_{\dd_u}(q)) 
%&\le& \td_X(\ep_G,\pi_{\dd_u}(v))+1
%  \le k\td_{Y''}(\ep_H,\pi_\Omega(\alpha(v)))+k^2+1\\
%&  \le& k\td_{Y''}(\ep_H,\pi_\Omega(\alpha(q))+2k)+k^2+1~,
%\hspace{0.2in} \text{ and}\\
%%\end{eqnarray*}
%%\begin{eqnarray*}
%\td_{Y''}(\ep_H,\pi_\Omega(\alpha(q)))
%     &\le& \td_{Y''}(\ep_H,\pi_\Omega(\alpha(v)))+2k
%  \le k\td_X(\ep_G,\pi_{\dd_u}(v))+3k\\
% &\le& k(\td_X(\ep_G,\pi_{\dd_u}(q))+1)+3k~.
%\end{eqnarray*}
%For a point $q$ in the interior of a 2-cell of $\dd_u$
%with boundary vertex $v$ within a coarse distance 1 from $q$, then
%\begin{eqnarray*}
%\td_X(\ep_G,\pi_{\dd_u}(q)) 
%&\le& \td_X(\ep_G,\pi_{\dd_u}(v))+1
%  \le k\td_{Y''}(\ep_H,\pi_\Omega(\alpha(v)))+k^2+1\\
%&  \le& k\td_{Y''}(\ep_H,\pi_\Omega(\alpha(q))+\maxp'')+k^2+1~,
%\hspace{0.2in} \text{ and}\\
%%\end{eqnarray*}
%%\begin{eqnarray*}
%\td_{Y''}(\ep_H,\pi_\Omega(\alpha(q)))
%     &\le& \td_{Y''}(\ep_H,\pi_\Omega(\alpha(v)))+\maxp''
%  \le k\td_X(\ep_G,\pi_{\dd_u}(v))+k+2k\maxp\\
% &\le& k(\td_X(\ep_G,\pi_{\dd_u}(q))+1)+k+2k\maxp~.
%\end{eqnarray*}
Then for all $q \in \dd_u$, we have
\begin{eqnarray*}
\td_X(\ep_G,\pi_{\dd_u}(q))
&\le&  k\td_{Y''}(\ep_H,\pi_\Omega(\alpha(q)))+2k^2\maxr+3k^2+1~,
\text{ and }\\
\td_{Y''}(\ep_H,\pi_\Omega(\alpha(q)))
&\le&  k\td_X(\ep_G,\pi_{\dd_u}(q))+4k+2k\maxr~.
\end{eqnarray*}

Now suppose that $p$ is any point in $C_{l(z'')}$ and 
$0 \le s < t \le 1$. In the case that $\pi_{\dd_u} \circ \ph_u$
is $f_1^e$-tame, then
\begin{eqnarray*}
\td_{Y''}(\ep_H,\pi_\Omega(\omega(p,s))) 
   &=& \td_{Y''}(\ep_H,\pi_\Omega(\alpha(\ph_u(\hat \alpha(p),s)))) \\
     & \le&  k \td_{X}(\ep_G,\pi_{\dd_u}(\ph_u(\hat \alpha(p),s))) 
                +4k+2k\maxr \\
  & < & k (f_1^e(\td_{X}(\ep_G,\pi_{\dd_u}(\ph_u(\hat \alpha(p),t))))+2)
                  +4k+2k\maxr\\
 &\le & k f_1^e(k\td_{Y''}(\ep_H,\pi_\Omega(\alpha(\ph_u(\hat \alpha(p),t))))
                     +2k^2\maxr+3k^2+1)+6k+2k\maxr\\
       &=& k f_1^e(k\td_{Y''}(\ep_H,\pi_\Omega(\omega(p,t))+2k^2\maxr+3k^2+1)
              +6k+2k\maxr~.
\end{eqnarray*}
Hence $\pi_\omega \circ \omega$ is $f_2^e$-tame for the 
%nondecreasing 
function 
$f_2^e(n):=k f_1^e(kn+2k^2\maxr+3k^2+1)+6k+2k\maxr$.

\smallskip

{\em Step III. For $u'$ over $\pp'''$:}  In this step we
construct another finite, planar, contractible,
and combinatorial 2-complex $\Lambda_{u'}$ starting from $\Omega$,
by adding a ``collar'' around the outside boundary.
Write the word $u'=b_1 \cdots b_n$ with each $b_i \in B$.
For each $1 \le i \le n-1$, let $w_i$ be
a word labeling a geodesic edge path in
$Y$ from $\phi(\theta(b_1 \cdots b_i))$
to $b_1 \cdots b_i$;
the quasi-isometry inequality in (3) above
implies that the length of $w_i$ is at most $k$. 
%If $w_i$ is not the empty word, 
We add to $\Lambda_{u'}$ a vertex $x_i$ and the vertices
and edges of a directed edge path $p_i$
labeled by $w_i$
from the vertex ${v_i}$ to $x_i$, 
where 
%$v_i$ is the vertex %in $\bo \dd_u$ 
%at the end of the path
$v_i:=\term(\path_{\Omega}(1,\tp(\ep_G,\tth(e_H,b_1 \cdots b_i))))$ 
%starting at the basepoint.
Note that if $w_i$ is the empty word, we identify
$x_i$ with the vertex ${v_i}$; the path $p_i$
is a constant path at this vertex.
Then $*={v_0}=x_0=x_n$
(and $p_0$ and $p_n$ are the constant path at
this vertex); let this vertex
be the basepoint of $\Lambda_{u'}$.

Next we add to $\Lambda_{u'}$ a directed edge 
$\check e_i$ labeled
by $b_i$ from the vertex $x_{i-1}$ to the vertex $x_i$.
The path $q_i$ from $v_{i-1}$ to $v_i$
along the boundary of 
the subcomplex $\Omega$ is labeled by the nonempty word
$z_i:=\tp(\theta(b_1 \cdots b_{i-1}),\tth(b_1 \cdots b_{i-1},b_i))$.
If both of the paths $p_{i-1},p_i$  are constant
and the label of path $q_i$ is the single letter $b_i$,
then we identify the edge $\check e_i$ with the path $q_i$.
Otherwise, we 
attach a 2-cell $\hat \sigma_i$ along the edge circuit
following the edge path starting at $v_{i-1}$ 
that traverses the path
$q_i$,
the path $p_i$, the reverse of the edge $\check e_i$, 
and finally the reverse of
the path $p_{i-1}$. 
\begin{figure}
\begin{center}
\includegraphics[width=3.4in,height=1.4in]{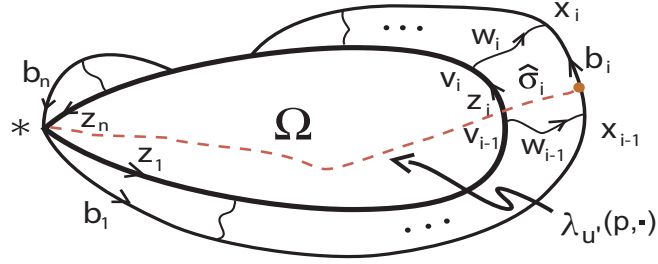}
\caption{Van Kampen diagram $\Lambda_{u'}$ and \dhy\ 
$\lambda_{u'}$}\label{fig:qicollar}
\end{center}
\end{figure}
See Figure~\ref{fig:qicollar} for a picture of the resulting
diagram.

Now the complex $\Lambda_{u'}$ is a van Kampen
diagram for the original word $u'$,
with respect to the presentation
$\pp'''=\langle B \mid S \cup S''' \rangle$ of $H$,
where $S'''$ is the set of all nonempty words in
$B^*$ of length at most 
$\maxr''':=2k\maxr+(2k)^2+2k+1$ that
represent $\ep_H$.
%(Note that the presentation $\langle B \mid S'''\rangle$
%also presents $H$, and $\Lambda_{u'}$ is also
%a diagram over this more restricted presentation.)
Let 
%$\maxp''':=1+L$ and let
$Y'''$ be the corresponding Cayley complex.

We define a \dhy\  
$\lambda_{u'}:C_{l(u')} \times [0,1] \ra \Lambda_{u'}$
by extending the map
$\omega$ on the subcomplex $\Omega$ (from Step II)
as follows.
First we let the cell complex $C_{l(u')}$ be the complex
$C_{l(z'')}$ with each subpath in $C_{l(z'')}$ mapping
to a path $q_i$ in $\bo \Omega$ replaced by a single edge.
From our definitions of $\tp$ and $\tth$, each $q_i$ path
is labeled by a nonempty word, and so
$C_{l(z'')}$ is a refinement of the cell complex structure
$C_{l(u')}$ on $S^1$, and we let 
$\hat \beta:C_{l(u')} \ra C_{l(z'')}$ be the identity
on the underlying circle.
Next define a homotopy 
$\tilde \lambda:C_{l(z'')} \times [0,1] \ra \Lambda_{u'}$
as follows.
For each $1 \le i \le n$, let $\tilde v_i$ be the point in $S^1$
with $\omega(\tilde v_i,1) = v_i$.
Define $\tilde \lambda(\tilde v_i,t):=\omega(\tilde v_i,2t)$
for $t \in [0,\frac{1}{2}]$, and let $\tilde \lambda(\tilde v_i,t)$
for $t \in [\frac{1}{2},1]$ be a constant speed path
along $p_i$ from $v_i$ to $x_i$.
On the interior of the edge $\tilde e_i$ from $\tilde v_{i-1}$
to $\tilde v_i$, 
define
the homotopy $\tilde \lambda|_{\tilde e_i \times [0,\frac{1}{2}]}$
to follow $\omega|_{\tilde e_i \times [0,1]}$ at double speed,
and let $\tilde \lambda|_{\tilde e_i \times [\frac{1}{2},1]}$
go through the 2-cell $\hat \sigma_i$ (or, if there is 
no such cell, let this portion of $\tilde \lambda$ be constant)
from $q_i$ to $\check e_i$.
Finally, we define the \dhy\ 
$\lambda_{u'}:C_{l(u'))} \times [0,1] \ra \Lambda_{u'}$ by 
$\lambda_{u'}:=\tilde \lambda \circ (\hat \beta \times \idmap)$.
(A path $\lambda_{u'}(p,\cdot)$ is illustrated by
the dashed path in Figure~\ref{fig:qicollar}).
This map $\lambda_{u'}$ is a 
\dhy\ for the diagram $\Lambda_{u'}$.

%For each $1 \le i \le n$,
%define $\lambda_{u'}(x_i,t):=\omega({v_i},2t)$
%for $t \in [0,\frac{1}{2}]$, and let $\lambda_{u'}(x_i,t)$
%for $t \in [\frac{1}{2},1]$ be a constant speed path
%along $p_i$ from $v_i$ to $x_i$.
%On the interior of the edge $\check e_i$, 
%define
%the homotopy $\lambda_{u'}|_{\check e_i \times [0,\frac{1}{2}]}$
%to follow $\omega|_{q_i \times [0,1]}$ at double speed,
%and let $\lambda_{u'}|_{\check e_i \times [\frac{1}{2},1]}$
%go through the 2-cell $\check \sigma_i$ (or, if there is 
%no such cell, let this portion of $\lambda_{u'}$ be constant)
%to $\check e_i$.
%This map $\lambda_{u'}$ is a disk
%homotopy for the diagram $\Lambda_{u'}$.

Next we analyze the intrinsic tameness of $\lambda_{u'}$.
Since $\Omega$ is a subdiagram of $\Lambda_{u'}$, 
for any vertex $v$ in $\Omega$, we have 
$d_{\Lambda_{u'}}(*,v) \le d_\Omega(*,v)$.  
Given any edge path $\beta$ in $\Lambda_{u'}$
from $*$ to $v$ that is not completely contained
in the subdiagram $\Omega$, the subpaths of
$\beta$ lying in the ``collar'' can be replaced
by paths along $\bo \Omega$ of length at most
a factor of $4k^2$ longer.  Then
$d_\Omega(*,v) \le 4k^2d_{\Lambda_{u'}}(*,v)$. 
Hence for any point $q \in \Omega$, we have
$\td_{\Lambda_{u'}}(*,q) \le \td_\Omega(*,q) 
\le 4k^2\td_{\Lambda_{u'}}(*,q)+4k^2+1+\maxr'''$.

Now suppose that $p$ is any point of $C_{l(u')}$ and
$0 \le s < t \le 1$, and that $\ph_u$ is $f^i$-tame.
Since $n < f^i(n)+2$ for all $n \in \nn$ from Remark~\ref{rmk:infinite},
from the definition of $f_2$ we also have $n<f_2(n)$ for all $n$.
If $t \le \frac{1}{2}$, then the path 
$\lambda_{u'}(p,\cdot)$ on $[0,t]$ is a reparametrization of
$\omega(p,\cdot)$, and so Step II, the 
fact that $f_2^i$ is nondecreasing,  and the inequalities above
give 
\begin{eqnarray*}
\td_{\Lambda_{u'}}(*,\lambda_{u'}(p,s)) 
&\le& \td_{\Omega}(*,\lambda_{u'}(p,s))\\
&<& f_2^i(\td_{\Omega}(*,\lambda_{u'}(p,t)))\\
&\le& f_2^i(4k^2\td_{\Lambda_{u'}}(*,\lambda_{u'}(p,t))+4k^2+1+\maxr''').
\end{eqnarray*}
If $t > \frac{1}{2}$ and $s \le \frac{1}{2}$, then we have 
$ \td_{\Lambda_{u'}}(*,\lambda_{u'}(p,s)) < 
f_2^i(4k^2\td_{\Lambda_{u'}}(*,\lambda_{u'}(p,\frac{1}{2}))+4k^2+1+\maxr''')$
and
$|\td_{\Lambda_{u'}}(*,\lambda_{u'}(p,t))-
     \td_{\Lambda_{u'}}(*,\lambda_{u'}(p,\frac{1}{2}))|<\maxr'''+1$, so 
$$ \td_{\Lambda_{u'}}(*,\lambda_{u'}(p,s)) < 
f_2^i(4k^2(\td_{\Lambda_{u'}}(*,\lambda_{u'}(p,t))+\maxr'''+1)+
          4k^2+1+\maxr''').$$
If $s>\frac{1}{2}$, then 
\begin{eqnarray*}
\td_{\Lambda_{u'}}(*,\lambda_{u'}(p,s)) &\le& 
\td_{\Lambda_{u'}}(*,\lambda_{u'}(p,t)) + \maxr'''+1
< f_2^i(\td_{\Lambda_{u'}}(*,\lambda_{u'}(p,t)))+\maxr'''+1.
\end{eqnarray*} 
%using the
%fact that $n \le f^i(n) +\maxr+1 \le f_2^i(n)$ for this
%infinite group case.
%%Then for all possibilities for $t$, we have
Then $\lambda_{u'}$ is $f_3^i$-tame for the function
$f_3^i(n):=f_2^i(4k^2n+8k^2+1+(4k^2+1)\maxr''')+\maxr'''+1$.

We note that we have now
completed the proof of Theorem~\rfthmitisqi\  in
the intrinsic case:  
The \dhs\ of the \scf\ 
$\{(\Lambda_{u'},\lambda_{u'}) \mid u' \in B^*, u'=_H\ep_H\}$
%of van Kampen diagrams and \dhs\ over the
(over the presentation $\pp'''=\langle B \mid S \cup S''' \rangle$)
%is a \scf\ with  
%implies an intrinsic relaxed \tfi\ for the
are $f_3^i$-tame, and so
Proposition~\ref{prop:htpydomain} says that 
$(H,\pp''')$ has an intrinsic \tff\ 
Lipschitz equivalent to $f_3^i$, 
and hence also to $f^i$.

The analysis of the extrinsic tameness in 
this step is simplified by the fact that 
for all $q \in \Omega$, we have
$\td_{Y''}(\ep_H,\pi_\Omega(q))=
   \td_{Y'''}(\ep_H,\pi_{\Lambda_{u'}}(q))$, 
since the 1-skeleta of $Y''$ and $Y'''$ are
determined by the generating sets of the 
presentations $\pp''$ and $\pp'''$, which are
the same.  A similar argument to those
above shows that if $\pi_{\dd_u} \circ \ph_u$ 
is $f^e$-tame, then $\pi_{\dd_{u'}} \circ \lambda_{u'}$ 
is $f_3^e$-tame for the function
$f_3^e(n):=f_2^e(n+\maxr'''+1)+\maxr+1$.

{\em Step IV.  For $u'$ over $\pp'$:}  
Finally, we turn to building a van Kampen
diagram $\dd_{u'}'$ for $u'$ over the original
presentation $\pp'$.
For each nonempty word $w$ over $B$ of length at
most $\maxr'''$
% $2k \cdot \max\{l(r) \mid r \in R\}+4k+1$
satisfying $w=_H \ep_H$, let $\dd_w'$ be a
fixed choice of van Kampen diagram for $w$
with respect to the presentation $\pp'$ of $H$, and 
let ${\mathcal F}$ be the (finite) collection of these
diagrams.
%$=\{\dd_w' \mid w \in B^*, w=_H \ep_H, l(w) \le L\}$
A diagram $\dd_{u'}'$ over the presentation $\pp'$ is built by
replacing 2-cells of $\Lambda_{u'}$,
proceeding through the 2-cells of $\Lambda_{u'}$
one at a time.  Let $\tau$ be a 
2-cell of $\Lambda_{u'}$, and 
let $*_\tau$ be a choice of basepoint vertex
in $\bo \tau$.
Let $x$ be the word labeling
the path $\bo \tau$ starting at $*_\tau$ and 
reading counterclockwise.  Since
$l(x) \le L$, there is an associated van Kampen diagram
$\dd_\tau'=\dd_{x}'$ in the collection ${\mathcal F}$.
Note that although $\Lambda_{u'}$
is a combinatorial 2-complex, and so the cell $\tau$
is a polygon, the boundary label $x$
may not be freely or cyclically reduced.  The 
van Kampen diagram $\dd_x'$ may not be a polygon, but
instead a collection of polygons connected by edge
paths, and possibly with edge path ``tendrils''.
We replace the 2-cell $\tau$ with a copy $\dd_\tau'$
of the van Kampen
diagram $\dd_x'$, identifying the boundary
edge labels as needed, obtaining another
planar diagram.  Repeating this
for each 2-cell of of the resulting complex at
each step, results in the
van Kampen diagram $\dd_{u'}'$ for $u'$ with
respect to $\pp'$.  

From the process of constructing $\dd_{u'}'$
from $\Lambda$, for each 2-cell $\tau$ there
is a continuous onto map $\tau \ra \dd_\tau'$
preserving the boundary edge path labeling,
and so there is an induced continuous surjection
$\gamma: \Lambda_{u'} \ra \dd_{u'}'$.  Note that the
boundary edge paths of $\Lambda_{u'}$ and $\dd_{u'}'$
are the same.  
Then the composition
$\ph_{u'}':=\gamma \circ \lambda_{u'}:C_{l(u')} \times [0,1] 
\ra \dd_{u'}'$ is a \dhy.

To analyze the extrinsic tameness, we first note
that for all points $\hat q \in \Lambda_{u'}\sko$, 
the images $\pi_{\Lambda_{u'}}(\hat q)$ in $Y'''$ and 
 $\pi_{\dd_{u'}'}(\gamma(\hat q))$ in $Y$ are
the same point in the 1-skeleta $Y\sko=(Y''')\sko$,
and so $\td_{Y'''}(\ep_H,\pi_{\Lambda_{u'}}(\hat q))=
\td_Y(\pi_{\dd_{u'}'}(\gamma(\hat q)))$.
Let $M:=2\max\{\td_\dd(*,r) \mid \dd \in {\mathcal F}, 
r \in \dd\}$.

Suppose that $p$ is any point in $C_{l(u')}$ and
$0 \le s<t \le 1$, and that $\pi_{\dd_u} \circ \ph_u$
is $f^e$-tame.  
If $\lambda_{u'}(p,s) \in \Lambda_{u'}\sko$,
then define $s':=s$; otherwise, let $0 \le s'<s$ satisfy
$\lambda_{u'}(p,s') \in \Lambda_{u'}\sko$ and 
$\lambda_{u'}(p,(s',s])$ is a subset of a single open
2-cell of $\Lambda_{u'}$.  Similarly, if 
$\lambda_{u'}(p,t) \in \Lambda_{u'}\sko$,
then define $t':=t$, and otherwise, let $t < t' \le 1$ satisfy
$\lambda_{u'}(p,t') \in \Lambda_{u'}\sko$ and 
$\lambda_{u'}(p,[t,t'))$ is a subset of a single open
2-cell of $\Lambda_{u'}$.
From Remark~\ref{rmk:infinite} and the
choice of $f_3^e$, we also have $n < f_3^e(n)$ for all $n \in \nn$. 
Then 
\begin{eqnarray*}
\td_Y(\ep_H,\pi_{\dd_{u'}'}(\ph_{u'}(p,s)))
& = &  \td_{Y}(\ep_H,\pi_{\dd_{u'}'}(\gamma(\lambda_{u'}(p,s)))) \\
&\le&  \td_{Y}(\ep_H,\pi_{\dd_{u'}'}(\gamma(\lambda_{u'}(p,s')))) +M \\
&= &   \td_{Y'''}(\ep_H,\pi_{\Lambda_{u'}}(\lambda_{u'}(p,s'))) +M \\
&<&  f_3^e(\td_{Y'''}(\ep_H,\pi_{\Lambda_{u'}}(\lambda_{u'}(p,t')))) +M \\
&=&    f_3^e(\td_{Y}(\ep_H,\pi_{\dd_{u'}'}(\gamma(\lambda_{u'}(p,t'))))) +M \\
&\le&  f_3^e(\td_{Y}(\ep_H,\pi_{\dd_{u'}'}(\gamma(\lambda_{u'}(p,t))))+M) +M.
\end{eqnarray*}
%Then for all $q \in \dd_{u'}'$, there is a point
%$\gamma(\hat q) \in \gamma(\Lambda_{u'}\sko)$ with 
%$|\td_Y(\ep_H,\pi_{\dd_{u'}'}(q))-
%       \td_{Y'''}(\ep_H,\pi_{\Lambda_{u'}}(\hat q))| \le I$.
%Finally, suppose that $p$ is any point in $C_{l(u')}$ and
%$0 \le s<t \le 1$.  Let $0 \le s' \le s$ and $t \le t' \le 1$
%so that $\ph_{u'}'(p,s')=\gamma(\hat \sigma)$ and
%$\ph_{u'}'(p,t')=\gamma(\hat \tau)$ for some 
%$\hat \sigma, \hat \tau \in 
%An argument similar to those above now shows that
Therefore
$\pi_{\dd_{u'}'} \circ \ph_{u'}'$ is 
$f_4^e$-tame, for the
function $f_4^e(n):=f_3^e(n+M)+M$.
Since the functions $f_j^e$ and $f_{j+1}^e$ are 
Lipschitz equivalent for all $j$, then $f_4^e$ is
Lipschitz equivalent to~$f^e$.  

Now the collection
$\{(\dd_{u'}',\ph_{u'}') \mid u' \in B^*, u'=_H\ep_H\}$
%of van Kampen diagrams and \dhs\ 
is a \scf\ 
%an extrinsic relaxed \tfi\ 
for the pair $(H,\pp')$ such that each 
$\pi_{\dd_{u'}'} \circ \ph_{u'}'$ is tame
with respect to a function that is Lipschitz equivalent to $f^e$,
and Proposition~\ref{prop:htpydomain} completes the proof.
\end{proof}

The obstruction to applying Step IV of the above
proof in the intrinsic case stems from the fact that
%We note that in Step IV of the above proof, 
the map 
$\gamma:\Lambda_{u'} \ra \dd_{u'}'$ behaves well with
respect to extrinsic coarse distance, but 
may not behave
well with respect to intrinsic coarse distance.
The latter results because
the replacement of a 2-cell $\tau$ of $\Lambda_{u'}$
with a van Kampen diagram $\dd_\tau'$ can result in
the identification of vertices of $\Lambda_{u'}$.  
%Thus another method would be needed to approach
%a proof that intrinsic \tfs\ are completely independent of
%the presentation of the group.

%%%%%%%%%%%%%%%%%%%%%%%%%%%%%%%%%%%%%%%%%%%%%%%%%%%%%%%%%%%%%%%%%%%%%%%%%%%%
%%%%%%%%%%%%%%%%%%%%%%%%%%%%%%%%%%%%%%%%%%%%%%%%%%%%%%%%%%%%%%%%%%%%%%%%%%%%
%%%%%%%%%%%%%%%%%%%%%%%%%%%%%%%%%%%%%%%%%%%%%%%%%%%%%%%%%%%%%%%%%%%%%%%%%%%%

%%%%%  THE BIBLIOGRAPHY

%%%%%%%%%%%%%%%%%%%%%%%%%%%%%%%%%%%%%%%%%%%%%%%%%%%%%%%%%%%%%%%%%%%%%%%%%%%%
%%%%%%%%%%%%%%%%%%%%%%%%%%%%%%%%%%%%%%%%%%%%%%%%%%%%%%%%%%%%%%%%%%%%%%%%%%%%

\end{document}